\documentclass[reqno]{amsart}
\usepackage[english]{babel}
\usepackage{graphicx,subfigure}
\usepackage{amsmath,amssymb}
\usepackage{appendix}
\usepackage{color}
\definecolor{oneblue}{rgb}{0,0.0,0.75}

\usepackage{fancyhdr}

\makeatletter
\newcommand{\sech}{\mathop{\operator@font sech}}
\newcommand{\sign}{\mathop{\operator@font sign}}
\makeatother

\newtheorem{lemma}{Lemma}[section]
\newtheorem{theorem}{Theorem}[section]
\newtheorem{proposition}{Proposition}[section]

\newtheorem{remark}{Remark}[section]
\numberwithin{equation}{section}

\begin{document}

\title[]{Notes on the Boussinesq-Full dispersion systems for internal waves: Numerical solution and solitary waves}

\author[V. A. Dougalis]{Vassilios A. Dougalis}
\address{Mathematics Department, University of Athens, 15784
Zographou, Greece \and Institute of Applied \& Computational
Mathematics, FO.R.T.H., 71110 Heraklion, Greece}
\email{doug@math.uoa.gr}

\author[A. Duran]{Angel Duran}
\address{ Applied Mathematics Department,  University of
Valladolid, 47011 Valladolid, Spain}
\email{angel@mac.uva.es}

\author[L. Saridaki]{Leetha Saridaki}
\address{Mathematics Department, University of Athens, 15784
Zographou, Greece \and Institute of Applied \& Computational
Mathematics, FO.R.T.H., 71110 Heraklion, Greece}
\email{leetha.saridaki@gmail.com}


\dedicatory{\large \it To the memory of Vassili}

\begin{abstract}
In this paper we study some theoretical and numerical issues of the Boussinesq/Full dispersion system. This is a a three-parameter system of pde's that models the propagation of internal waves along the interface of two-fluid layers with rigid lid condition for the upper layer, and under a Boussinesq regime for the upper layer and a full dispersion regime for the lower layer. We first discretize in space the periodic initial-value problem with a Fourier-Galerkin spectral method and prove error estimates for several ranges of values of the parameters. Solitary waves of the model systems are then studied numerically in several ways. The numerical generation is analyzed by approximating the ode system with periodic boundary conditions for the solitary-wave profiles with a Fourier spectral scheme, implemented in a collocation form, and solving iteratively the corresponding algebraic system in Fourier space with the Petviashvili method accelerated with the minimal polynomial extrapolation technique. Motivated by the numerical results, a new result of existence of solitary waves is proved. 
In the last part of the paper, the dynamics of these solitary waves is studied computationally, To this end, the semidiscrete systems obtained from the Fourier-Galerkin discretization in space are integrated numerically in time by a Runge-Kutta Composition method of order four. The fully discrete scheme is used to explore numerically the stability of solitary waves, their collisions, and the resolution of other initial conditions into solitary waves.

\end{abstract}

\maketitle

\section{Introduction}
This paper is concerned with the numerical approximation of the one-dimensional, three-parameter family of Boussinesq/Full dispersion (B/FD) systems
\begin{equation}\label{BFD}
\begin{array}{l}
\left(1-b\partial_{x}^{2} \right)\zeta_t+\frac{1}{\gamma}\left((1- \zeta)u \right)_x-\frac{1}{\gamma^2}(|D|{\rm coth}|D|)u_x+\frac{1}{\gamma}(a-\frac{1}{\gamma^{2}}{\rm coth}^{2}|D|)u_{xxx}=0\ ,\\
\left(1-d\partial_{x}^{2} \right)u_t+(1-\gamma)\zeta_x-\frac{1}{\gamma}uu_x+c(1-\gamma)\zeta_{xxx}=0\ ,
\end{array}
\end{equation}
where $\zeta=\zeta(x,t)$, $u=u(x,t)$, and the nonlocal operator $|D|$ has the Fourier symbol 
\begin{equation}\label{BFD1A}
\widehat{|D|f}(k)=|k|\widehat{f}(k),\; k\in\mathbb{R},
\end{equation} 
with $\widehat{f}(k)$ standing for the Fourier transfom of $f$ at $k$. The constants $a, b, c, d$ are given by
\begin{eqnarray}
a=\frac{1-\alpha_{1}-3\beta}{3},\quad b=\frac{\alpha_{1}}{3},\quad
c=\beta\alpha_{2},\quad d=\beta(1-\alpha_{2}).\label{BB1b}
\end{eqnarray}
where $\alpha_{1}\geq 0, \alpha_{2}\leq 1, \beta\geq 0$ are modelling parameters. The B/FD systems (\ref{BFD}) are nonlocal pde systems derived in \cite{BLS2008} as a model for the propagation of internal waves along the interface of a two-layers system with rigid-lid condition on the upper layer and the lower layer bounded below by a rigid, horizontal bottom. In (\ref{BFD}), $\zeta(x,t)$ denotes the interfacial deviation at position $x$ along the fluid channel and at time $t$, while $u=(1-\beta\partial_{x}^{2})^{-1}{v}$, being $v$ a velocity variable. If $\rho_{1}$ and $\rho_{2}$ denote, respectively, the density of the upper and lower layer, with $\rho_{2}>\rho_{1}$, then $\gamma=\frac{\rho_{1}}{\rho_{2}}<1$. The asymptotic regime of validation of (\ref{BFD}) is described in \cite{BLS2008}, sec. 3.1.2, in terms of the parameters
\begin{equation}\label{param}
\epsilon=\frac{a}{d_{1}},\; \mu=\frac{d_{1}}{\lambda^{2}},\; \epsilon_{2}=\frac{a}{d_{2}},\; \mu_{2}=\frac{d_{2}}{\lambda^{2}},
\end{equation}
where $d_{1}$ (resp. $d_{2}$) is the depth of the upper (resp. lower) layer, and $a$ and $\lambda$ stand, respectively, for a typical amplitude and wavelength of the interfacial wave. The Boussinesq/Full dispersion regime assumes that, cf. \cite{BLS2008,S}
\begin{itemize}
\item[(i)] The deformations are of small amplitude for both layers 
($\epsilon<<1, \epsilon_{2}<<1$).
\item[(ii)] A Boussinesq regime for the upper layer (dispersive and nonlinear effects of the same order) and the lower layer is not shallow; that means
$$\mu\sim\epsilon<<1, \mu_{2}\sim 1,$$ for which it holds that $\delta^{2}\sim\epsilon$ and $\epsilon_{2}\sim\epsilon^{3/2}<<1$.
\end{itemize}
The system (\ref{BFD}) is in unscaled, nondimensional form. (There is no $\delta$ in the unscaled formulation as here $\delta\sim\epsilon^{1/2}$ and $\epsilon$ was set equal to $1$.) 

By the usual Fourier analysis of the linearized system, it is noted in \cite{BLS2008} that the systems (\ref{BFD}) are linearly well posed when $b, d\geq 0, a, c\leq 0$. Many of the systems have been shown to be nonlinearly well posed, locally in time, in \cite{A}, but not all of them are relevant for the internal-wave problem, in view of the restrictions on the parameters $\alpha_{1}, \alpha_{2}, \beta$. In Table \ref{BFD_t1} we tabulate the possible cases of the linearly well-posed systems, along with their indication of their relevance for internal waves, and the reference to a theorem in \cite{A} for local nonlinear well-posedness of their initial-value problem (ivp) in the Sobolev space $(\zeta,u)\in H^{s_{1}}\times H^{s_{2}}$.
\begin{table}[ht]
\begin{center}
\begin{tabular}{|c|c|c|c|c|c|c|}
    \hline
No.&$b$&$d$&$a$&$c$&Relevant&Nonlinear well-posedness\\\hline\hline
1&$+$&$+$&$-$&$-$&Yes&Theorem 2.1(i) $H^{s}\times H^{s}, s\geq 0$ (\lq generic\rq\ B/FD)\\\hline
2&$+$&$+$&$-$&$0$&Yes&Theorem 2.1(ii) $H^{s}\times H^{s-1}, s\geq 0$\\\hline
3&$+$&$+$&$0$&$-$&Yes&Theorem 2.1(i) $H^{s}\times H^{s}, s\geq 0$ \\\hline
4&$+$&$+$&$0$&$0$&Yes&Theorem 2.1(ii) $H^{s-1}\times H^{s}, s\geq 0$ (\lq BBM-BBM\rq\  B/FD)\\\hline
5&$+$&$0$&$-$&$-$&No&Theorem 2.5 $H^{s+1}\times H^{s}, s> 3/2$ \\\hline
6&$+$&$0$&$-$&$0$&Yes&Theorem 2.4 $H^{s}\times H^{s}, s> 3/2$ \\\hline
7&$+$&$0$&$0$&$-$&No&Theorem 2.5 $H^{s+1}\times H^{s}, s> 3/2$\\\hline
8&$+$&$0$&$0$&$0$&Yes&Theorem 2.4 $H^{s}\times H^{s}, s> 3/2$\\\hline
9&$0$&$+$&$-$&$-$&Yes&Theorem 2.3 $H^{s}\times H^{s+1}, s> 1/2$ \\\hline
10&$0$&$+$&$-$&$0$&Yes&Theorem 2.2 $H^{s}\times H^{s+2}, s>1/2$ \\\hline
11&$0$&$+$&$0$&$-$&Yes&Theorem 2.3 $H^{s}\times H^{s+1}, s> 1/2$\\\hline
12&$0$&$+$&$0$&$0$&Yes&Theorem 2.2 $H^{s}\times H^{s+2}, s>1/2$\\\hline
13&$0$&$0$&$-$&$-$&No&\\\hline
14&$0$&$0$&$-$&$0$&No&\\\hline
15&$0$&$0$&$0$&$-$&No&\\\hline
16&$0$&$0$&$0$&$0$&No&\\\hline
\end{tabular}
\end{center}
\caption{B/FD systems: Relevance for internal waves, existing nonlinear well-posedness theory, \cite{A}.}
\label{BFD_t1}
\end{table}
It is to be noted that the systems that are indicated to be relevant to internal waves are not so for all admissible values of $\alpha_{1}, \alpha_{2}$, and $\beta$, but only for subsets thereof, in general.
\begin{remark}
We recall that when $b=d$, the system (\ref{BFD}) admits a Hamiltonian structure
\begin{eqnarray*}
\partial_{t}\begin{pmatrix}
\zeta\\ u
\end{pmatrix}={\mathcal J}\nabla E(\zeta,u),
\end{eqnarray*}
in appropriate spaces for $(\zeta,u)$, where
\begin{eqnarray*}
\mathcal{J}:=-\begin{pmatrix}
0&\left(1-b\partial_{x}^{2} \right)^{-1}\partial_{x}\\\left(1-b\partial_{x}^{2} \right)^{-1}\partial_{x}&0
\end{pmatrix},
\end{eqnarray*}
\begin{eqnarray}
E(\zeta,u):=\frac{1}{2}\int_{\mathbb{R}}\left((1-\gamma)\zeta (1+c\partial_{x}^{2})\zeta+u\mathcal{L}_{\mu_{2}}u-\frac{1}{2\gamma}\zeta u^{2}\right)dx,\label{BFDEnergy}
\end{eqnarray}
where
\begin{eqnarray*}
\mathcal{L}_{\mu_{2}}=\frac{1}{\gamma}-\frac{1}{\gamma^{2}}|D|{\rm coth}(|D|)+\frac{1}{\gamma}\left(a-\frac{1}{\gamma^{2}}{\rm coth}^{2}(|D|)\right)\partial_{x}^{2},
\end{eqnarray*}
and $\nabla$ denotes variational derivative.
In particular, (\ref{BFDEnergy}) is a conserved quantity. On the other hand, since the group of spatial translations
\begin{eqnarray*}
g_{\epsilon}: (\zeta(x),v(x))\mapsto (\zeta(x-\epsilon),v(x-\epsilon)),\; \epsilon\in\mathbb{R},
\end{eqnarray*}
is a symmetry group of (\ref{BFD}) an its infinitesimal generator is given by
\begin{eqnarray*}
\frac{d}{d\epsilon}g_{\epsilon} (\zeta,v)\Big|_{\epsilon=0}=-\begin{pmatrix}
\partial_{x}\zeta\\\partial_{x}v
\end{pmatrix}=\mathcal{J}\nabla I(\zeta,v_{\beta}),
\end{eqnarray*}
where
\begin{eqnarray}
I(\zeta,u):=\int_{\mathbb{R}}\left(\zeta (1-b\partial_{x}^{2})u\right)dx,\label{BFDMom}
\end{eqnarray}
then, (\ref{BFDMom}) is also a conserved quantity, \cite{Olver1993}.
\end{remark}

Long time existence for the corresponding two-dimensional version of (\ref{BFD}) (existence on time scales of order $1/\epsilon$) for all the linearly well-posed systems (for which $a,c\leq 0, b,d\geq 0$) is proved in \cite{SautX2020}. In the same paper, a result of global existence of small solutions in the Hamiltonian case $b=d>0, a\leq 0, c<0$ is established. On the other hand, the existence of solitary wave solutions of (\ref{BFD}) in the Hamiltonian case $b=d>0, a,c\leq 0$, and for a certain range of speeds (depending on the parameters $a,b,c,d$ and $\gamma$) is proved in \cite{AnguloS2019}, as well as smoothness and exponential decay of the solitary waves.

Concerning the numerical approximation of systems of Boussinesq type for surface and internal waves, error estimates of Galerkin-Finite Element semidiscrete schemes and for associated fully discrete schemes with high-order, explicit Runge-Kutta (RK) methods for some initial-boundary-value problems for several surface-wave Boussinesq systems can be found in \cite{AD1,AD2,DougalisMS2007}. Spectral methods of collocation type with the classical explicit $4$th-order RK time integrator to discretize some Boussinesq systems in the surface wave case are analyzed in \cite{XRAA}, while the $L^{2}$ convergence of semidiscrete approximations of the Boussinesq/Boussinesq systems modelling internal wave propagation with the spectral Fourier-Galerkin method is established in \cite{DDS1}. A full discretization of the resulting semidiscrete systems with a fourth-order RK Composition method is studied computationally in the previous paper and its extended version \cite{DDS_arxiv}.

The present paper analyzes several aspects of the B/FD systems (\ref{BFD}). The first question, developed in section \ref{sec2}, concerns the derivation of error estimates for the spectral Fourier-Galerkin semidiscretization of the periodic ivp of some of the B/FD systems. Explicitly, we prove $L^{2}$ convergence of the spectral approximations in the following cases of Table \ref{BFD_t1}:
\begin{itemize}
\item The \lq BBM-BBM\rq\  B/FD case ($b,d>=, a=c=0$),
\item The \lq generic\rq\ B/FD case ($b,d>0, a,c<0$),
\item The B/FD cases with $b,d>0$ and either $a<0,c=0$ or $a=0, c<0$.
\end{itemize}
The error estimates depend indeed on the regularity of the solution of the periodic initial-value problem for (\ref{BFD}), in the sense that if $\zeta, u$ are in the Sobolev space of order $\mu\geq 1$, then the $L^{2}$ error behaves like $O(N^{-\mu})$, where $N\geq 1$ is the degree of the Fourier-Galerkin approximation as trigonometric polynomial. In particular, this proves spectral $L^{2}$ convergence for smooth solutions.

The rest of the paper is devoted to the study of solitary-wave solutions of (\ref{BFD}). In section \ref{sec3} we study numerically the existence of solitary waves by approximating the ode system for the wave profiles with periodic boundary conditions with a Fourier collocation scheme and solving iteratively the resulting algebraic system in the Fourier space with the Petviashvili method, \cite{Petv1976,pelinovskys}, accelerated with a vector extrapolation algorithm, \cite{sidifs,smithfs,sidi}, a technique which turned out to be successful in other models, cf. \cite{AlvarezD2015}. We first generate approximations to solitary wave solutions in the Hamiltonian case whose existence is proved in \cite{AnguloS2019}. Within this Hamiltonian case but beyond the assumptions made in \cite{AnguloS2019}, the method can still generate approximate solitary wave profiles. Motivated by these numerical results, in Appendix \ref{appendixA} we determine a limit value $c_{\gamma}$ such that a solitary wave of speed $c_{s}$ with $|c_{s}|<c_{\gamma}$ exists. The proof is a modification of that in \cite{AnguloS2019} when applying the Concentration-Compactness theory, \cite{Lions}, and fits the numerical results. In Section \ref{sec3} we also explore the system of an analogous speed limit for the nonhamiltonian case.

The second point on solitary waves considered in the present paper is developed in section \ref{sec4} and concerns their dynamics. We integrate numerically the ode semidiscrete systems from the Fourier-Galerkin approximation of the periodic ivp of (\ref{BFD}) with a fourth-order RK Composition-type method based on the implicit midpoint rule, mentioned above. The high accuracy of the resulting fully discrete scheme and its performance when approximating other nonlinear dispersive wave equations, cf. \cite{DD} and references therein, enable us to study computationally the stability of the solitary waves from numerical experiments of different type: monitoring the evolution of the numerical approximation from small and large perturbations of the computed solitary waves as initial condition, studying the behaviour of overtaking and head-on collisions of solitary waves, and analyzing the resolution of smooth initial conditions into a train of solitary waves. Some concluding remarks are outlined in section \ref{sec5}.

Throughout the paper, the following notation will be used. We denote by $(\cdot,\cdot)$ (resp. $||\cdot ||$) the inner product (resp. norm) on $L^{2}=L^{2}(0,1)$. For real $s$, $H^{s}$ will denote the $L^{2}$-based Sobolev spaces of periodic functions on $[0,1]$. The norm of  $g\in H^{s}$ is given by
$$||g||_{s}=\left(\sum_{k\in\mathbb{Z}}(1+k^{2})^{s}|\widehat{g}(k)|^{2}\right)^{1/2},
$$ where $\widehat{g}(k)$ stands for the $k$th Fourier coefficient of $g$. The norm on $L^{\infty}$
will be denoted by $|\cdot |_{\infty}$.
If $N\geq 1$, is an integer, the Euclidean inner product (resp. norm) in $\mathbb{C}^{2N}$ is denoted by $(\cdot,\cdot)_{N}$ (resp. $||\cdot||_{N}$).
\section{Error estimates of the spectral semidiscretizations}
\label{sec2}
\subsection{The \lq BBM-BBM\rq\  B/FD case}
\label{sec21}
We first consider the systems of the case (4) of Table \ref{BFD_t1}, namely the \lq BBM-BBM\rq\ B/FD systems, for which $b, d>0, a=c=0$.
The initial-periodic boundary-value problem for these systems is: For $0\leq x\leq 1, 0\leq t\leq T$ we seek $u=u(x,t), \zeta=\zeta(x,t)$, $1$-periodic with respect to $x$, such that
\begin{equation}\label{dds21}
\begin{array}{l}
\left[1-b\partial_{x}^{2} \right]\zeta_t+\frac{1}{\gamma}\left((1- \zeta)u \right)_x-\frac{1}{\gamma^2}(|D|{\rm coth}|D|)u_x-\frac{1}{\gamma^{3}}({\rm coth}^{2}|D|)u_{xxx}=0\ ,\\
\left[1-d\partial_{x}^{2} \right]u_t+(1-\gamma)\zeta_x-\frac{1}{\gamma}uu_x=0\ ,
\end{array}
\end{equation}
where $u(x,0)=u_{0}(x), \zeta(x,0)=\zeta_{0}(x), 0\leq x\leq 1$. Here $u_{0},\zeta_{0}$ are $1$-periodic given functions and $D$ is the operator $\frac{1}{i}\partial{x}$. We assume that (\ref{dds21}) has a unique solution which is smooth enough for the purposes of the error estimation.

Fourier analysis yields the following representation of (\ref{dds21}) for $k\in\mathbb{Z}, t\in [0,T]$. (Note that the symbol of the operator $|D|$ is $|k|$, cf. (\ref{BFD1A}).)
\begin{equation*}
\begin{array}{l}
(1+bk^{2})\widehat{\zeta}_t+\frac{ik}{\gamma}\widehat{u}-\frac{ik}{\gamma}\widehat{\zeta u}-\frac{ik}{\gamma^2}(|k|{\rm coth}|k|)\widehat{u}+\frac{ik^{3}}{\gamma^{3}}({\rm coth}^{2}|k|)\widehat{u}=0\ ,\\
(1+dk^{2})\widehat{u}_t+(ik)(1-\gamma)\widehat{\zeta}-\frac{ik}{2\gamma}\widehat{u^{2}}=0\ ,\\
\widehat{\zeta}(k,0)=\widehat{\zeta_{0}}(k), \widehat{u}(k,0)=\widehat{u_{0}}(k)\ .
\end{array}
\end{equation*}
(Here, for any function $f(x,t)$, $1$-periodic in $x$, we let $\widehat{f}=\widehat{f}(k,t)$ denote its $k$th Fourier component at $t$.) Since $b,d>0$, we will write the above for $k\in\mathbb{Z}, 0\leq t\leq T$, as
\begin{equation}\label{dds22}
\begin{array}{l}
\widehat{\zeta}_t+\frac{1}{\gamma}\frac{ik}{1+bk^{2}}\widehat{u}-\frac{1}{\gamma}\frac{ik}{1+bk^{2}}\widehat{\zeta u}-\frac{ik}{\gamma^2}\frac{|k|{\rm coth}|k|}{1+bk^{2}}\widehat{u}+\frac{ik^{3}}{\gamma^{3}}\frac{{\rm coth}^{2}|k|}{1+bk^{2}}\widehat{u}=0\ ,\\
\widehat{u}_t+(1-\gamma)\frac{ik}{1+dk^{2}}\widehat{\zeta}-\frac{i}{2\gamma}\frac{k}{1+dk^{2}}\widehat{u^{2}}=0\ ,\\
\widehat{\zeta}(k,0)=\widehat{\zeta_{0}}(k), \widehat{u}(k,0)=\widehat{u_{0}}(k)\ ,
\end{array}
\end{equation}
and view it, for each $k\in\mathbb{Z}$, as an ode ivp on $[0,T]$. The spectral implementation of the schemes will be based on this formulation.

For the error estimation it is convenient to work in physical space. For this purpose we introduce some nonlocal operators acting on the $L^{2}$-based Sobolev spaces of periodic functions $H^{s}, s\in\mathbb{R}$. In what follows we formally define
\begin{eqnarray}
T_{b}&=&(1- b\partial_{x}^{2})^{-1},\;{\rm with}\;{\rm symbol}\; \frac{1}{1+bk^{2}},\label{dds23i}\\
T_{d}&=&(1- d\partial_{x}^{2})^{-1},\;{\rm with}\;{\rm symbol}\; \frac{1}{1+dk^{2}},\label{dds23ii}\\
S_{1}&=&|D|{\rm coth}|D|,\;{\rm with}\;{\rm symbol}\;|k|{\rm coth}|k|,\label{dds23iii}\\
S_{2}&=&{\rm coth}^{2}|D|,\;{\rm with}\;{\rm symbol}\;{\rm coth}^{2}|k|\label{dds23iv}.
\end{eqnarray}
Note that all the operators commute on appropriate domains.

For the operators (\ref{dds23i}), (\ref{dds23ii}), of order $-2$, we easily obtain that $T_{\kappa}:H^{s-2}\rightarrow H^{s}, \forall s\in\mathbb{R}, \kappa=b\;{\rm or}\; d$, and that
\begin{eqnarray}
||T_{\kappa}f||_{s}\leq C_{\kappa}||f||_{s-2},\; s\in\mathbb{R}, f\in H^{s-2}.\label{dds24}
\end{eqnarray}
To study the other nonlocal operators (\ref{dds23iii}), (\ref{dds23iv}), consider the function $\phi(x)=x{\rm coth}{x}$, which is even and continuous on $\mathbb{R}$. Since $\phi(x)=x\frac{e^{2x}+1}{e^{2x}-1}$, we have that $\phi(x)\sim x$ for $x>>1$, and by Taylor's theorem $\phi(x)=1+\frac{x^{2}}{3}-\frac{x^{4}}{45}+O(x^{6})$ for small $|x|$. We conclude that the symbol $\tau_{1}(k)=|k|{\rm coth}|k|$ of the operator $S_{1}$ behaves like $|k|$ for $|k|>>1$ and like $1+\frac{k^{2}}{3}+O(k^{4})$ for small $|k|$. We also note that the symbol $\tau_{2}(k)={\rm coth}^{2}|k|$ of the operator $S_{2}$ behaves like $1/k^{2}$ as $|k|\rightarrow 0$ and tends to $1$ as $|k|\rightarrow\infty$.

Getting back to (\ref{dds22}), we rewrite it in physical space as
\begin{equation}\label{dds24b}
\begin{array}{l}
{\zeta}_t+\frac{1}{\gamma}\mathcal{R}_{b}{u}-\frac{1}{\gamma}\mathcal{R}_{b}{\zeta u}-\frac{1}{\gamma^2}\mathcal{S}_{1}{u}-\frac{1}{\gamma^{3}}\mathcal{S}_{2}{u}=0\ ,\\
{u}_t+(1-\gamma)\mathcal{R}_{d}{\zeta}-\frac{1}{2\gamma}\mathcal{R}_{d}{u^{2}}=0\ ,\\
{\zeta}(x,0)={\zeta_{0}}(x), {u}(x,0)={u_{0}}(x)\ ,
\end{array}
\end{equation}
for $x\in [0,1], t\in [0,T]$, where
\begin{equation}\label{dds25}
\begin{array}{l}
\mathcal{R}_{b}=T_{b}\partial_{x},\quad \mathcal{R}_{d}=T_{d}\partial_{x}\ ,\\
\mathcal{S}_{1}=T_{b}S_{1}\partial_{x},\quad \mathcal{S}_{2}=T_{b}S_{2}\partial_{x}^{3}\ .
\end{array}
\end{equation}
The symbol of $\mathcal{R}_{\kappa}, \kappa=b\;{\rm or}\; d$, is
$$\sigma(\mathcal{R}_{\kappa})=\frac{ik}{1+\kappa k^{2}}.$$ Therefore $\mathcal{R}_{\kappa}$ is of order $-1$ and by (\ref{dds24}) we have for $\kappa=b\;{\rm or}\; d$
\begin{equation}\label{dds26}
||\mathcal{R}_{\kappa}f||_{j}\leq C_{\kappa}||f||_{j-1},\; f\in H^{j-1},\; j\in\mathbb{R}.
\end{equation}
The symbol of $\mathcal{S}_{1}$ is $\frac{ik|k|{\rm coth}|k|}{1+bk^{2}}$. Hence in view of the behaviour of the function $|x|{\rm coth}{|x|}$, we see that this symbol is bounded for $k\in\mathbb{R}$, i.~e. that $\mathcal{S}_{1}$ is of order $0$ and it holds that
\begin{equation}\label{dds27}
||\mathcal{S}_{1}f||_{j}\leq C||f||_{j},\; f\in H^{j},\; j\in\mathbb{R}\ .
\end{equation}
Finally, since the symbol of $S_{2}$ is $\frac{-ik^{3}{\rm coth}^{2}|k|}{1+bk^{2}}$, which is bounded for bounded intervals of $k$ and is of $O(|k|)$, as $|k|\rightarrow\infty$, in view of the behaviour of the function ${\rm coth}^{2}{|x|}$, hence $\mathcal{S}_{2}$ is of order $1$ and we have
\begin{equation}\label{dds28}
||\mathcal{S}_{2}f||_{j}\leq C||f||_{j+1},\; f\in H^{j+1},\; j\in\mathbb{R}\ .
\end{equation}

Let $N\geq 1$ be an integer and consider the finite dimensional space $S_{N}$ given by
$$S_{N}:={\rm span}\{e^{ikx},\; k\in\mathbb{Z},\; -N\leq k\leq N\}.$$

The Fourier-Galerkin semidiscrete approximation of (\ref{dds24b}) in $S_{N}$  is denoted by $(\zeta_{N},u_{N})$ and satisfies, for $0\leq t\leq T$
\begin{equation}\label{dds29}
\begin{array}{l}
{\zeta}_{N,t}+\frac{1}{\gamma}\mathcal{R}_{b}{u_{N}}-\frac{1}{\gamma}\mathcal{R}_{b}P_{N}({\zeta_{N} u_{N}})-\frac{1}{\gamma^2}\mathcal{S}_{1}{u_{N}}-\frac{1}{\gamma^{3}}\mathcal{S}_{2}{u_{N}}=0\ ,\\
{u}_{N,t}+(1-\gamma)\mathcal{R}_{d}{\zeta_{N}}-\frac{1}{2\gamma}\mathcal{R}_{d}P_{N}({u_{N}^{2}})=0\ ,\\
{\zeta}_{N}(0)=P_{N}{\zeta_{0}}, {u}_{N}(0)=P_{N}{u_{0}}\ ,
\end{array}
\end{equation}
where $P_{N}$ is the $L^{2}$-projection onto $S_{N}$ (see \cite{DDS1}). Note that
 for $0\leq j\leq \mu$, and for any $v\in H^{r}, r\geq 1$,
\begin{eqnarray*}
||v-P_{N}v||_{j}&\leq &C_{j,r}N^{j-r}||v||_{\mu},\label{epn1}\\
||v-P_{N}v||_{j,\infty}&\leq &C_{j,r}N^{1/2+j-r}||v||_{r},\label{epn2}
\end{eqnarray*}
Note also that $P_{N}$ commutes with $\mathcal{R}_{\kappa}, \kappa=b,d, \mathcal{S}_{i}, i=1,2$, and $\partial_{x}$. The system (\ref{dds29}) has certainly a local in time solution.

We let now $\theta=\zeta_{N}-P_{N}\zeta, \rho=P_{N}\zeta-\zeta$, so that $\zeta_{N}-\zeta=\theta+\rho$, $\xi=u_{N}-P_{N}u, \sigma=P_{N}u-u$, so that $u_{N}-u=\xi+\sigma$. Then, applying to the first pde in (\ref{dds24b}) the operator $P_{N}$, substracting the resulting equation from the first equation in (\ref{dds29}), and performing the same operations between the second pde of (\ref{dds24b}) and the second equation of (\ref{dds29}), we have
\begin{equation}\label{dds210}
\begin{array}{l}
{\theta}_{t}+\frac{1}{\gamma}\mathcal{R}_{b}{\xi}-\frac{1}{\gamma}\mathcal{R}_{b}P_{N}A-\frac{1}{\gamma^2}\mathcal{S}_{1}{\xi}-\frac{1}{\gamma^{3}}\mathcal{S}_{2}{\xi}=0\ ,\\
{\xi}_{t}+(1-\gamma)\mathcal{R}_{d}{\theta}-\frac{1}{2\gamma}\mathcal{R}_{d}P_{N}B=0\ ,\\
{\theta}(0)=0, {\xi}(0)=0\ ,
\end{array}
\end{equation}
where, as in \cite{DDS1},
\begin{eqnarray}
A&:=&\zeta_{N}u_{N}-\zeta u=u\rho+\zeta\sigma+u\theta+\zeta\xi+\sigma\theta+\rho\xi+\rho\sigma+\theta\xi,\label{dds211}\\
B&:=&u_{N}^{2}-u^{2}=u\sigma+u\xi+\sigma\xi+\frac{1}{2}\sigma^{2}+\frac{1}{2}\xi^{2}.\label{dds212}
\end{eqnarray}
The solution of (\ref{dds210}) exists for as long as we have existence of $\zeta_{N}, u_{N}$ in (\ref{dds29}). Part of the proof of the proposition below will be to show that $\theta, \xi$, and therefore $\zeta_{N}, u_{N}$, exist up to $t=T$.
\begin{proposition}
\label{pro21}
Suppose that $\zeta, u\in H^{\mu},\mu\geq 1$ for $0\leq t\leq T$. If $N$ is sufficiently large, 
\begin{equation}\label{dds213}
\max_{0\leq t\leq T}\left(||\zeta_{N}-\zeta||+||u_{N}-u||\right)\leq CN^{-\mu}.
\end{equation}
\end{proposition}
\begin{proof}
While the semidiscrete solution exists, taking the $L^{2}$ norms in the first equation of (\ref{dds210}) and $H^{1}$ norms in the second, we get by the triangle inequality
\begin{eqnarray*}
||\theta_{t}||&\leq & C\left(|1\mathcal{R}_{b}\xi||+||\mathcal{S}_{1}\xi||+||\mathcal{S}_{2}\xi||+||\mathcal{R}_{b}P_{N}A||\right),\\
||\xi_{t}||_{1}&\leq & C\left(||\mathcal{R}_{d}\theta||_{1}+||\mathcal{R}_{d}P_{N}B||_{1}\right).
\end{eqnarray*}
Using now the properties (\ref{dds26}), (\ref{dds27}), and (\ref{dds28}) for the nonlocal operators $\mathcal{R}_{b}, \mathcal{R}_{d}, \mathcal{S}_{1}, \mathcal{S}_{2}$ we see that
\begin{eqnarray}
||\theta_{t}||&\leq & C\left(||\xi||+||\xi||+||\xi||_{1}+||A||\right)\leq C\left(||\xi||_{1}+||A||\right),\label{dds214}\\
||\xi_{t}||_{1}&\le &C\left(||\theta||+||B||\right),\label{dds215}
\end{eqnarray}
where, to get (\ref{dds215}) we used the fact that $||\mathcal{R}_{d}P_{N}B||\leq C||P_{N}B||\leq C||B||$, in view of (\ref{dds26}).

Now
\begin{eqnarray}
||A||&\leq &|u|_{\infty}||\rho||+|\zeta|_{\infty}|1\sigma||+|u|_{\infty}||\theta||+|\zeta|_{\infty}||\xi||+|\sigma|_{\infty}||\theta||\nonumber\\
&&+C||\rho||||\xi||_{1}+|\rho|_{\infty}||\sigma||+|\xi|_{\infty}||\theta||.\label{dds216}
\end{eqnarray}
Let now $0<t_{N}\leq T$ be the maximal time instance for which
\begin{equation}\label{dds217}
|\xi|_{\infty}\leq 1,\quad 0\leq t\leq t_{N}.
\end{equation}
(The existence of such a $t_{N}$ follows by continuity and the fact that $\xi(0)=0$.) Therefore by (\ref{dds216}), (\ref{dds217}), and the properties of $S_{N}$ we have
\begin{equation}\label{dds218}
||A||\leq C\left(N^{-\mu}+||\theta||+||\xi||_{1}\right),\quad 0<t\leq t_{N}.
\end{equation}
For $B$ we have
\begin{eqnarray}
||B||\leq |u|_{\infty}||\sigma||+|u|_{\infty}||\xi||+|\sigma|_{\infty}||\xi||+\frac{1}{2}|\sigma|_{\infty}||\sigma||+\frac{1}{2}|\xi|_{\infty}||\xi||.\label{dds219}
\end{eqnarray}
Hence, in view of (\ref{dds217}), (\ref{dds219}) gives
\begin{equation}\label{dds220}
||B||\leq C\left(N^{-\mu}+||\theta||+||\xi||_{1}\right),\quad 0<t\leq t_{N}.
\end{equation}
We conclude from (\ref{dds214}), (\ref{dds215}), (\ref{dds218}), (\ref{dds220}) that
\begin{equation}\label{dds221}
||\theta_{t}||+||\xi_{t}||_{1}\leq C\left(N^{-\mu}+||\theta||+||\xi||_{1}\right),\quad 0<t\leq t_{N},
\end{equation}
where $C$ is independent of $N, t_{N}$. By Gronwall's inequality therefore (since $\theta(0)=\xi(0)=0$)
\begin{equation*}\label{dds222}
||\theta||+||\xi||_{1}\leq CN^{-\mu},\quad 0<t\leq t_{N},
\end{equation*}
Thus we see, since $|\xi|_{\infty}\leq ||\xi||_{1}$, that $t_{N}$ was not maximal in (\ref{dds217}) if $N$ was taken sufficiently large. The argument may continue up to $t_{N}=T$ and it follows that
\begin{equation*}\label{dds223}
||\theta||+||\xi||_{1}\leq CN^{-\mu},\quad 0<t\leq T.
\end{equation*}
Hence (\ref{dds213}) follows.
\end{proof}

\subsection{The \lq generic\rq\  B/FD case}
\label{sec22}
We next consider the systems of the case (1) of Table \ref{BFD_t1}, namely the B/FD systems in the \lq generic\rq\ case $b,d>0, a,c<0$. The initial-periodic bvp for these systems takes the following form: For $0\leq x\leq 1, 0\leq t\leq T$, we seek $u=u(x,t), \zeta=\zeta(x,t)$, $1$-periodic with respect to $x$, such that
\begin{equation}\label{dds31}
\begin{array}{l}
\left[1-b\partial_{x}^{2} \right]\zeta_t+\frac{1}{\gamma}\left((1- \zeta)u \right)_x-\frac{1}{\gamma^2}(|D|{\rm coth}|D|)u_x+\frac{1}{\gamma}(a-\frac{1}{\gamma^{2}}{\rm coth}^{2}|D|)u_{xxx}=0\ ,\\
\left[1-d\partial_{x}^{2} \right]u_t+(1-\gamma)\zeta_x-\frac{1}{\gamma}uu_x+c(1-\gamma)\zeta_{xxx}=0\ ,
\end{array}
\end{equation}
with $u(x,0)=u_{0}(x), \zeta(x,0)=\zeta_{0}(x), 0\leq x\leq 1$, where $u_{0},\zeta_{0}$ are given $1$-periodic functions. We assume that (\ref{dds31}) has a unique solution $(\zeta,u)$, which is smooth enough for the purposes of the error estimation.

Fourier analysis yields, as in section \ref{sec21}, the following equations valid for $k\in\mathbb{Z}, t\in [0,T]$:
\begin{equation}\label{dds32}
\begin{array}{l}
\widehat{\zeta}_t+\frac{ik}{1+bk^{2}}\left(\frac{1}{\gamma}-\frac{1}{\gamma^2}{|k|{\rm coth}|k|}-\frac{ak^{2}}{\gamma}+\frac{k^{2}}{\gamma^{3}}{{\rm coth}^{2}|k|}\right)
\widehat{u}=\frac{i}{\gamma}\frac{k}{1+bk^{2}}\widehat{\zeta u}\ ,\\
\widehat{u}_t+ik\frac{(1-\gamma)(1-ck^{2})}{1+dk^{2}}\widehat{\zeta}=\frac{i}{2\gamma}\frac{k}{1+dk^{2}}\widehat{u^{2}}\ ,\\
\widehat{\zeta}(k,0)=\widehat{\zeta_{0}}(k), \widehat{u}(k,0)=\widehat{u_{0}}(k)\ ,
\end{array}
\end{equation}
In the sequel we let for $k$ real
\begin{equation}\label{dds33}
g(k):=\frac{1}{\gamma}-\frac{1}{\gamma^2}{|k|{\rm coth}|k|}-\frac{ak^{2}}{\gamma}+\frac{k^{2}}{\gamma^{3}}{{\rm coth}^{2}|k|}.
\end{equation}
It is clear from the properties of the functions $\tau_{1}(k), \tau_{2}(k)$, defined in section \ref{sec21}, that $g$ is well defined, even, and continuous, when viewed as a function of $k\in\mathbb{R}$. It is not hard to see, using Taylor expansions, that for $|x|<<1$
\begin{equation*}\label{dds34}
g(x)=\frac{1}{\gamma}\left(1+\frac{1}{\gamma}\left(\frac{1}{\gamma}-1\right)+|a|x^{2}+\frac{x^{2}}{3\gamma}\left(\frac{1}{\gamma}-1\right)\right)+O(x^{4}).
\end{equation*}
Since $0<\gamma<1$, the coefficients of this expression are positive. Hence $g(x)$ is positive for small $|x|$ and behaves approximately like a quadratic with positive coefficients. For $|x|>>1$ we have
\begin{equation*}\label{dds35}
g(x)\sim \frac{1}{\gamma}\left(1-\frac{x}{\gamma}+|a|x^{2}+\frac{1}{\gamma^{2}}x^{2}\right).
\end{equation*}
It can also be seen that for all $x$, $g(x)$ is positive and that $g'(x)>0$ for $x>0$. We conclude that there exist positive constants $d_{i}, d_{i}', i=1,2$ such that
\begin{equation}\label{dds36}
d_{1}'+d_{2}'k^{2}\leq g(k)\leq d_{1}+d_{2}k^{2},\quad k\in\mathbb{R}.
\end{equation}
The Fourier-Galerkin semidiscrete approximation to (\ref{dds32}) in Fourier space is of the form
\begin{equation}\label{dds37}
\begin{array}{l}
\left(\widehat{\zeta_{N}}\right)_{t}+\frac{ik}{1+bk^{2}}g(k)
\widehat{u_{N}}=-\frac{i}{\gamma}\frac{k}{1+bk^{2}}\widehat{\zeta_{N} u_{N}},\; 0\leq t\leq T\ ,\\
\left(\widehat{u_{N}}\right)_t+ik\frac{(1-\gamma)(1-ck^{2})}{1+dk^{2}}\widehat{\zeta_{N}}=\frac{i}{2\gamma}\frac{k}{1+dk^{2}}\widehat{u_{N}^{2}},\; k=-N,\ldots,N\ ,\\
\widehat{\zeta_{N}}(k,0)=\widehat{\zeta_{0}}(k), \widehat{u_{N}}(k,0)=\widehat{u_{0}}(k)\ ,
\end{array}
\end{equation}
The implementation of the semidiscrete scheme uses this formulation.
\begin{proposition}
\label{pro31}
Suppose that $\zeta, u\in H^{\mu},\mu\geq 1$ for $0\leq t\leq T$. Then, for $N$ is sufficiently large, 
\begin{equation}\label{dds38}
\max_{0\leq t\leq T}\left(||\zeta_{N}-\zeta||+||u_{N}-u||\right)\leq CN^{-\mu}.
\end{equation}
\end{proposition}
\begin{proof}
As usual we define $\theta=\zeta_{N}-P_{N}\zeta, \rho=P_{N}\zeta-\zeta$, so that $\zeta_{N}-\zeta=\theta+\rho$, $\xi=u_{N}-P_{N}u, \sigma=P_{N}u-u$, so that $u_{N}-u=\xi+\sigma$. Then, considering the equations (\ref{dds32}) for $k\in\mathbb{Z}, -N\leq k\leq N$ (i.~e. effectively applying $P_{N}$ on both sides of the pde's in (\ref{dds31})) and subtracting them from the relevant semidiscrete equations (\ref{dds37}) we obtain, for $-N\leq k\leq N$
\begin{eqnarray}
&&\widehat{\theta}_{t}+\frac{ik}{1+bk^{2}}g(k)
\widehat{\xi}=\frac{i}{\gamma}\frac{k}{1+bk^{2}}\widehat{P_{N}A},\label{dds391}\\
&&\widehat{\xi}_t+ik\frac{(1-\gamma)(1-ck^{2})}{1+dk^{2}}\widehat{\theta}=\frac{i}{2\gamma}\frac{k}{1+dk^{2}}\widehat{P_{N}B},\label{dds392}\\
&&\widehat{\theta}(k,0)=0, \widehat{\xi}(k,0)=0.\label{dds393}
\end{eqnarray}
Here $A$ and $B$ are given by (\ref{dds211}) and (\ref{dds212}) respectively. The semidiscrete ivp (\ref{dds37}) has a local in time unique solution and so does therefore the ivp (\ref{dds391})-(\ref{dds393}). It will be part of the proof in the sequel to show that the validity of these solutions may be extended up to $t=T$.

For the convergence proof we adapt in our finite-dimensional case a general \lq ode in Banach space\rq-style proof given e.~g. in section 2.2 (Theorem 2.5) of \cite{BonaChS2004}; see also \cite{A,DougalisMS2007}. We write  (\ref{dds391})-(\ref{dds393}) as
\begin{eqnarray}
&&\begin{pmatrix} \widehat{\theta}_{t}\\ \widehat{\xi}_t\end{pmatrix}+ik\mathcal{A}(k)\begin{pmatrix} \widehat{\theta}\\ \widehat{\xi}\end{pmatrix}=\begin{pmatrix} r_{1}\\ r_{2}\end{pmatrix},\; t\geq 0,\label{dds3101}\\
&&\widehat{\theta}(k,0)=0, \widehat{\xi}(k,0)=0,\label{dds3102}
\end{eqnarray}
where
\begin{equation*}\label{dds311}
\mathcal{A}(k)=\begin{pmatrix} 0&\frac{g(k)}{1+bk^{2}}\\ \frac{(1-\gamma)(1-ck^{2})}{1+dk^{2}}&0\end{pmatrix},
\end{equation*}
and
\begin{equation}\label{dds312}
r_{1}=r_{1}(k)=\frac{i}{\gamma}\frac{k}{1+bk^{2}}\widehat{P_{N}A},\;
r_{2}=r_{2}(k)=\frac{i}{2\gamma}\frac{k}{1+dk^{2}}\widehat{P_{N}B}.
\end{equation}
We now diagonalize the system (\ref{dds3101}). Note that the eigenvalues of $\mathcal{A}$ are real and are given by $\lambda_{\pm}=\pm\sigma$, where
\begin{equation*}\label{dds313}
\sigma=\sigma(k)=\left(\frac{(1-\gamma)g(k)(1-ck^{2})}{(1+bk^{2})(1+dk^{2})}\right)^{1/2},
\end{equation*}
while the corresponding eigenvectors are $v_{\pm}=(\pm\alpha,1)^{T}$, where
\begin{equation}\label{dds314}
\alpha=\alpha(k)=\left(\frac{g(k)(1+dk^{2})}{(1-\gamma)(1+bk^{2})(1-ck^{2})}\right)^{1/2}.
\end{equation}
Defining therefore the matrix $\mathcal{J}=\mathcal{J}(k)=\begin{pmatrix} \alpha&-\alpha\\1&1\end{pmatrix}$ we have
$$\mathcal{J}^{-1}=\frac{1}{2\alpha}\begin{pmatrix} 1&\alpha\\-1&\alpha\end{pmatrix}\;{\rm and}\; 
\mathcal{J}^{-1}\mathcal{A}\mathcal{J}=\begin{pmatrix}\sigma&0\\ 0&-\sigma\end{pmatrix}.$$
We now let $\eta$ and $\nu$  be elements of the finite-dimensional space $S_{N}$ defined for $-N\leq k\leq N$, $t\geq 0$, by
\begin{equation}\label{dds315}
\mathcal{J}^{-1}\begin{pmatrix} \widehat{\theta}\\ \widehat{\xi}\end{pmatrix}=\begin{pmatrix} \widehat{\eta}\\ \widehat{\nu}\end{pmatrix},
\end{equation}
so that
\begin{equation}\label{dds316}
\begin{array}{l}
\widehat{\eta}=\frac{1}{2\alpha}\widehat{\theta}+\frac{1}{2}\widehat{\xi},\quad \widehat{\xi}=\widehat{\eta}+\widehat{\nu}\ ,\\
\widehat{\nu}=\frac{-1}{2\alpha}\widehat{\theta}+\frac{1}{2}\widehat{\xi},\quad \widehat{\theta}=\alpha(k)(\widehat{\eta}-\widehat{\nu})\ .
\end{array}
\end{equation}
Note that since $\alpha(k)$ is of order $0$, in view of (\ref{dds36}), (\ref{dds314}), and $\alpha(k)\neq 0$ we have, by (\ref{dds316}) and Plancherel's formula that
\begin{equation}\label{dds317}
C_{2}\left(||\eta||_{s}+||\nu||_{s}\right)\leq ||\theta||_{s}+||\xi||_{s}\leq C_{1}\left(||\eta||_{s}+||\nu||_{s}\right)\ , s\geq 0\ ,
\end{equation}
for some positive constants $C_{1}, C_{2}$ independent of $N, s, \theta, \xi, \eta, \nu$. In view of (\ref{dds315}), (\ref{dds316}) we transform the system (\ref{dds3101}), (\ref{dds3102}) into
\begin{equation}\label{dds318}
\begin{array}{l}
\begin{pmatrix} \widehat{\eta}_{t}\\ \widehat{\nu}_t\end{pmatrix}+ik\begin{pmatrix}\sigma&0\\ 0&-\sigma\end{pmatrix}\begin{pmatrix} \widehat{\eta}\\ \widehat{\nu}\end{pmatrix}=\begin{pmatrix} p_{1}(k)\\ p_{2}(k)\end{pmatrix}:=\mathcal{J}^{-1}\begin{pmatrix} r_{1}(k)\\ r_{2}(k)\end{pmatrix}\ ,\\
\widehat{\eta}\big|_{t=0}=0, \widehat{\nu}\big|_{t=0}=0\ .
\end{array}
\end{equation}
We write this ivp in physical variables as
\begin{equation}\label{dds319}
\begin{array}{l}
\begin{pmatrix} {\eta}_{t}\\ {\nu}_t\end{pmatrix}+\mathcal{B}\begin{pmatrix} {\eta}\\ {\nu}\end{pmatrix}=F(\eta,\nu)\ ,x\in [0,1]\ , t\geq 0\ ,\\
{\eta}\big|_{t=0}=0, {\nu}\big|_{t=0}=0\ ,
\end{array}
\end{equation}
where $\mathcal{B}$ is a $2\times 2$ matrix operator with symbol $ik\begin{pmatrix}\sigma&0\\ 0&-\sigma\end{pmatrix}$ and $F(\eta,\nu)$ is a $2$-vector of periodic functions with Fourier coefficients given by
\begin{equation}\label{dds320}
\widehat{F(\eta,\nu)}(k)=\begin{pmatrix} p_{1}(k)\\ p_{2}(k)\end{pmatrix}\ ,
\end{equation}
where $(p_{1},p_{2})^{T}$ is the right-hand side of (\ref{dds318}).

Since $\mathcal{B}=i\mathcal{C}$, where $\mathcal{C}$ is a real symmetric operator matrix, $\mathcal{B}$ is a skew-adjoint operator with domain $H^{1}\times H^{1}$ and $U(t)=e^{-it\mathcal{B}}$ is a unitary group, cf. \cite{K}, p. 435, say on the Hilbert space $X=L^{2}\times L^{2}$.

Solving (\ref{dds319}) we obtain by Duhamel's formula
$$\begin{pmatrix}\eta\\ \nu\end{pmatrix}(t)=\int_{0}^{t}U(t-s)F(\eta,\nu)ds,$$ and since $||UF||_{X}=||F||_{X}$, we obtain
\begin{equation}\label{dds321}
\left\|\begin{pmatrix}\eta(t)\\ \nu(t)\end{pmatrix}\right\|_{X}\leq \int_{0}^{t}||F(\eta,\nu)||_{X}ds\ .
\end{equation}
We now estimate the right-hand side of (\ref{dds321}). From (\ref{dds320}), since $\alpha\neq 0$ and $\alpha$ is bounded, we have by (\ref{dds320}), (\ref{dds318}), (\ref{dds312})
\begin{eqnarray}
||F(\eta,\nu)||_{X}&\leq &C\left(||p_{1}||+||p_{2}||\right)\leq C\left(||r_{1}||+||r_{2}||\right)\nonumber\\
&\leq & C\left(||(1-b\partial_{x}^{2})^{-1}\partial_{x}P_{N}A||+ ||(1-d\partial_{x}^{2})^{-1}\partial_{x}P_{N}B||\right)\nonumber\\
&\leq & C\left(||A||+||B||\right),\label{dds322}
\end{eqnarray}
where we recall that $A$ and $B$ are given by (\ref{dds211}), (\ref{dds212}), respectively, and for the last inequality in (\ref{dds322}) see e.~g. (\ref{dds25}), (\ref{dds26}). Hence, by (\ref{dds322}), (\ref{dds321}), and (\ref{dds317}), we obtain, as long as $\theta$ and $\xi$ exist, that
\begin{equation}\label{dds323}
||\theta||+||\xi||\leq C\int_{0}^{t}\left(||A||+||B||\right) ds\ .
\end{equation}
We now note that
\begin{eqnarray*}
||A||&\leq & |u|_{\infty}||\rho||+|\zeta|_{\infty}||\sigma||+|u|_{\infty}||\theta||+|\zeta|_{\infty}||\xi||\\
&&+|\sigma|_{\infty}||\theta||+|\rho|_{\infty}||\xi||+|\rho|_{\infty}||\sigma||+|\xi|_{\infty}||\theta||.
\end{eqnarray*}
Since $|\sigma|_{\infty}, |\rho|_{\infty}$ are bounded because of our assumption on $\mu$, if $0<t_{N}\leq T$ is the maximal temporal instance for which
\begin{equation}\label{dds324}
|\xi|_{\infty}\leq 1,\quad 0\leq t\leq t_{N}\ ,
\end{equation}
holds, we obtain from the above estimate of $A$, that
\begin{equation}\label{dds325}
||A||\leq C\left(N^{-\mu}+||\theta||+||\xi||\right),\quad 0\leq t\leq t_{N}\ ,
\end{equation}
where $C$ is independent of $N$ and $t_{N}$. For $B$ we see that
\begin{equation*}
||B||\leq |u|_{\infty}||\sigma||+|u|_{\infty}||\xi||+|\sigma|_{\infty}||\xi||+\frac{1}{2}|\sigma|_{\infty}||\sigma||+\frac{1}{2}|\xi|_{\infty}||\xi||\ .
\end{equation*}
Therefore, from (\ref{dds324}), as before,
\begin{equation}\label{dds326}
||B||\leq C\left(N^{-\mu}+||\xi||\right),\quad 0\leq t\leq t_{N}\ .
\end{equation}
From (\ref{dds323}), (\ref{dds325}), (\ref{dds326}) and Gromwall's Lemma we see that $||\theta||+||\xi||\leq CN^{-\mu}, 0\leq t\leq t_{N}$, where $C$ is independent of $N, t_{N}$. Therefore $|\xi|_{\infty}\leq CN^{1/2-\mu}<1$ for $N$ large enough, and $t_{N}$ was not maximal in (\ref{dds324}). In the usual manner we conclude that $\theta$ and $\xi$ exist up to $t=T$ and that $||\theta||+||\xi||\leq CN^{-\mu}, 0\leq t\leq T$. Hence (\ref{dds38}) follows.
\end{proof}
\subsection{B/FD systems with $b,d>0$, and either $a<0, c=0$ or $a=0,c<0$}
\label{sec23}
We now consider the B/FD systems of categories (2) and (3) of Table \ref{BFD_t1}, i.~e. those for which $b,d>0$, and either $a<0, c=0$ or $a=0, c<0$. Since $b,d>0$, the Fourier equations (\ref{dds32}) are well defined and so are the semidiscrete equations (\ref{dds37}), {\it mutatis mutandis}. Note that $g(k)$, defined in (\ref{dds33}), depends only on $a$ and for those categories of systems $a\leq 0$. Hence $g$ is well defined for $k\in\mathbb{R}$, has all its previous properties (evenness, continuity), and (\ref{dds36}) holds for positive constants that we again label $d_{i}, d_{i}', i=1,2$.

For the convergence of the semidiscrete schemes we have an analogous result to that of Proposition \ref{pro31}:
\begin{proposition}
\label{pro41}
Let the solution $\zeta, u$ of the periodic ivp for the B/FD systems of either category (2) or (3) in Table \ref{BFD_t1} belongs to $H^{\mu}$ for $\mu\geq 1$, for $0\leq t\leq T$. With our usual notation, if $N$ is sufficiently large,
\begin{equation}\label{dds41}
\max_{0\leq t\leq T}\left(||\zeta_{N}-\zeta||+||u_{N}-u||\right)\leq CN^{-\mu}.
\end{equation}
\end{proposition}

\begin{proof}
(i) We first consider the case of the systems of category (3), i.~e. those with $b,d>0, a=0, c<0$. We use the same notation as in the proof of Proposition \ref{pro31} {\it mutatis mutandis}. Note that the solution of the analogous to (\ref{dds391})-(\ref{dds393}) ivp exists locally in time, and that $\alpha(k)$ is still of order $0$ and does not vanish for $k\in\mathbb{R}$. Hence (\ref{dds317}) holds here too. The rest of the proof follows.

(ii) Let now $b,d>0, a<0, c=0$. Using the same notation as in the proof of Proposition \ref{pro31} we see that in this case as well the solution of (\ref{dds391})-(\ref{dds393}) exists locally in $t$. Now $\alpha$ is given for $k\in\mathbb{R}$ by $\alpha(k)=\left(\frac{g(k)(1+dk^{2})}{(1-\gamma)(1+bk^{2})}\right)^{1/2}$; hence its order is equal to $1$ and it does not vanish for $k\in\mathbb{R}$. We easily conclude that now $(\theta,\xi)\in H^{s-1}\times H^{s}\Leftrightarrow (\eta,\nu)\in H^{s}\times H^{s}$ and that (\ref{dds317}) is replaced by
\begin{equation*}\label{dds42}
C_{2}\left(||\eta||_{s}+||\nu||_{s}\right)\leq ||\theta||_{s-1}+||\xi||_{s}\leq C_{1}\left(||\eta||_{s}+||\nu||_{s}\right),\quad s\geq 1\ .
\end{equation*}
Working now in the Hilbert space $X=H^{1}\times H^{1}$, we conclude that (\ref{dds321}) holds and that
\begin{eqnarray*}
||F(\eta,\nu)||_{X}
&\leq & C\left(||(1-b\partial_{x}^{2})^{-1}\partial_{x}P_{N}A||+ ||(1-d\partial_{x}^{2})^{-1}\partial_{x}P_{N}B||_{1}\right)\nonumber\\
&\leq & C\left(||A||+||B||\right),\label{dds43}
\end{eqnarray*}
where for the last inequality see e.~g. (\ref{dds25}), (\ref{dds26}). Proceeding now as in the proof of Proposition \ref{pro31} we obtain the slightly better result that $\theta, \xi$ exist up to $t=T$ and satisfy $||\theta||+||\xi||_{1}\leq CN^{-\mu}, 0\leq t\leq T$. Hence (\ref{dds41}) holds.
\end{proof}

\section{Solitary wave solutions. Numerical generation}
\label{sec3}
In this section we investigate numerically the existence of solitary wave solutions of (\ref{BFD}) and provide a numerical method to compute approximate solitary-wave profiles. They will be used in section \ref{sec4} to perform a computational study of their dynamics.

To this end, in this and the following sections we will consider the scaled version of (\ref{BFD}), written in terms of the parameters (\ref{param}) and given by
\begin{eqnarray}
J_{b}\partial_{t}\zeta+\mathcal{L}_{\mu_{2}}\partial_{x}u-\frac{\epsilon}{\gamma}\partial_{x}(\zeta u)&=&0,\nonumber\\
J_{d}\partial_{t}u+(1-\gamma)J_{c}\partial_{x}\zeta-\frac{\epsilon}{2\gamma}\partial_{x}(u^{2})&=&0,\label{BFD1}
\end{eqnarray}
where
\begin{eqnarray*}
J_{b}=1-\mu b\partial_{x}^{2},\; J_{d}=1-\mu d\partial_{x}^{2},\; J_{c}=1+\mu c\partial_{x}^{2}, 
\end{eqnarray*}
\begin{eqnarray*}
\mathcal{L}_{\mu_{2}}=\frac{1}{\gamma}-\frac{\sqrt{\mu}}{\gamma^{2}}|D|{\rm coth}(\sqrt{\mu_{2}}|D|)+\frac{\mu}{\gamma}\left(a-\frac{1}{\gamma^{2}}{\rm coth}^{2}(\sqrt{\mu_{2}}|D|)\right)\partial_{x}^{2}.
\end{eqnarray*}
Solitary wave solutions of (\ref{BFD1}) are solutions of the form $\zeta=\zeta(x-c_{s} t), u=u(x-c_{s} t), c_{s}\neq 0$, for smooth profiles $\zeta=\zeta(X), u=u(X), X=x-c_{s}t$, vanishing at $|X|\rightarrow\infty$ and satisfying
\begin{eqnarray}
-c_{s} J_{b}\zeta+\mathcal{L}_{\mu_{2}}u&=&\frac{\epsilon}{\gamma}(\zeta u),\nonumber\\
-c_{s} J_{d}u+(1-\gamma)J_{c}\zeta&=&\frac{\epsilon}{2\gamma}(u^{2}).\label{BFD2}
\end{eqnarray}
In terms of the Fourier transform of the profiles, (\ref{BFD2}) reads, for $k\in\mathbb{R}$
\begin{eqnarray}
-c_{s} j_{b}(k)\widehat{\zeta}(k)+{l}_{\mu_{2}}(k)\widehat{u}(k)&=&\frac{\epsilon}{\gamma}\widehat{(\zeta u)}(k),\nonumber\\
-c_{s} j_{d}(k)\widehat{u}(k)+(1-\gamma)j_{c}(k)\widehat{\zeta}(k)&=&\frac{\epsilon}{2\gamma}\widehat{(u^{2})}(k),\label{BFD2b}
\end{eqnarray}
where 
\begin{eqnarray}
j_{\alpha}(k)&=&1+\mu |\alpha|{k}^{2},\; \alpha=b,c,d,\nonumber\\
{l}_{\mu_{2}}(k)&=&\frac{1}{\gamma}-\frac{\sqrt{\mu}}{\gamma^{2}}|{k}|{\rm coth}(\sqrt{\mu_{2}}|{k}|)-\frac{\mu}{\gamma}\left(a-\frac{1}{\gamma^{2}}{\rm coth}^{2}(\sqrt{\mu_{2}}|{k}|)\right){k}^{2},\; k\neq 0,\nonumber\\
{l}_{\mu_{2}}(0)&=&\frac{1}{\gamma}-\frac{\sqrt{\mu}}{\gamma^{2}}\frac{1}{\sqrt{\mu_{2}}}+\frac{\mu}{\gamma^{3}}\frac{1}{\mu_{2}}.\label{BFD2c}
\end{eqnarray}
\subsection{Numerical technique of approximation}
Described here is the numerical procedure to approximate solitary wave solutions $(\zeta, u)$ of (\ref{BFD1}). For an integer $N\geq 1$, the system (\ref{BFD2}) is discretized on a long enough interval $(-L,L)$, with periodic boundary conditions, by Fourier pseudospectral approximations to the values of the profiles at a uniform grid of collocation points
\begin{eqnarray}
x_{j}=-L+j h,\; j=0,\ldots, N,\; h=\frac{2L}{N}.\label{34b}
\end{eqnarray}
The vector approximation $(\zeta_{h},u_{h})$, where $$\zeta_{h}=(\zeta_{h,0},\ldots,\zeta_{h,N-1})^{T},\quad u_{h}=(u_{h,0},\ldots,u_{h,N-1})^{T},$$ with $\zeta_{h,j}$ (resp. $u_{h,j}$) approximating $\zeta(x_{j})$ (resp. $u(x_{j})$), $j=0,\ldots,N-1$, must satisfy the system
\begin{eqnarray}
-c_{s} J_{b,h}\zeta_{h}+\mathcal{L}_{\mu_{2},h}u_{h}&=&\frac{\epsilon}{\gamma}(\zeta_{h}\cdot u_{h}),\nonumber\\
-c_{s} J_{d,h}u_{h}+(1-\gamma)J_{c,h}\zeta_{h}&=&\frac{\epsilon}{2\gamma}(u_{h}\cdot^{2}),\label{BFD3}
\end{eqnarray}
where
\begin{equation}\label{BFD4}
\begin{array}{l}
J_{b,h}=1-\mu bD_{N}^{2},\; J_{d,h}=1-\mu dD_{N}^{2},\; J_{c,h}=1+\mu cD_{N}^{2}, \\
\mathcal{L}_{\mu_{2},h}=\frac{1}{\gamma}-\frac{\sqrt{\mu}}{\gamma^{2}}|D_{N}|{\rm coth}(\sqrt{\mu_{2}}|D_{N}|)+\frac{\mu}{\gamma}\left(a-\frac{1}{\gamma^{2}}{\rm coth}^{2}(\sqrt{\mu_{2}}|D_{N}|)\right)D_{N}^{2}.
\end{array}
\end{equation}
In (\ref{BFD3}), (\ref{BFD4}), $I_{N}$ stands for the $N\times N$ identity matrix and 
$D_{N}$ is the $N\times N$ Fourier pseudospectral differentiation matrix based on the $x_{j}$. Furthermore, if $F_{N}$ denotes the $N\times N$ matrix of the discrete Fourier transform, note that
$$F_{N}^{-1}D_{N}F_{N}=C_{N},$$ where $C_{N}$ is diagonal and its diagonal $k$th entry is given by $i\widetilde{k}, \widetilde{k}=\pi k/L, k=0,\ldots,N-1$. Then the matrix $|D_{N}|$ is defined as
$$|D_{N}|=F_{N}|C_{N}|F_{N}^{-1},$$ where $|C_{N}|$ denotes the $N\times N$ diagonal matrix with diagonal entries given by $|\widetilde{k}|$. In addition,
the dots in the nonlinear terms of (\ref{BFD3}) denote the Hadamard product of vectors (they are dropped from now on). The resolution of (\ref{BFD3}) makes use of the Fourier representation, cf. (\ref{BFD2b}), (\ref{BFD2c})
\begin{eqnarray}
-c_{s} j_{b}(\widetilde{k})\widehat{\zeta_{h}}(k)+{l}_{\mu_{2}}(\widetilde{k})\widehat{u_{h}}(k)&=&\frac{\epsilon}{\gamma}\widehat{(\zeta_{h}u_{h})}(k),\nonumber\\
-c_{s} j_{d}(\widetilde{k})\widehat{u_{h}}(k)+(1-\gamma)j_{c}(\widetilde{k})\widehat{\zeta_{h}}(k)&=&\frac{\epsilon}{2\gamma}\widehat{(u_{h}^{2})}(k),\label{BFD5}
\end{eqnarray}
where $\widehat{\zeta_{h}}(k), \widehat{u_{h}}(k)$ denote, respectively, the $k$-th discrete Fourier coefficient of $\zeta_{h}$ and $u_{h}$, while for $k=0,\ldots, N-1$, we define $\widetilde{k}=\displaystyle\frac{\pi k}{L}$. 

Now, systems (\ref{BFD5}), written as
\begin{eqnarray*}
\underbrace{
\begin{pmatrix}
-c_{s} j_{b}(\widetilde{k})&{l}_{\mu_{2}}(\widetilde{k})\\
(1-\gamma)j_{c}(\widetilde{k})&-c_{s} j_{d}(\widetilde{k})
\end{pmatrix}
}_{\mathcal{S}_{k}}
\begin{pmatrix}
\widehat{\zeta_{h}}(k)\\ \widehat{u_{h}}(k)
\end{pmatrix}=\underbrace{\frac{\epsilon}{\gamma}
\begin{pmatrix}
\widehat{(\zeta_{h}u_{h})}(k)\\
\frac{\widehat{(u_{h}^{2})}(k)}{2}
\end{pmatrix}
}_{\mathcal{N}(\widehat{\zeta_{h}}(k), \widehat{u_{h}}(k))},
\end{eqnarray*}
are iteratively solved by using the Petviashvili iteration, \cite{Petv1976}
\begin{eqnarray}
m_{h}^{[n]}&=&\frac{\left(\mathcal{L}_{k}z^{[n]},z^{[n]}\right)_{N}}{\left(\mathcal{N}_{k}(z^{[n]}),z^{[n]}\right)_{N}},\label{BFD4b}\\
\mathcal{L}_{k}z^{[n+1]}&=&(m_{h}^{[n]})^{2}\mathcal{N}_{k}(z^{[n]}),\; n=0,1,\ldots,\nonumber
\end{eqnarray}
for $k=0,\ldots, N-1$, where $z^{[n]}=(\zeta_{h}^{[n]},u_{h}^{[n]})$ denotes the $n$-th iterate. The so-called minimal polynomial extrapolation method (MPE), \cite{smithfs}, is applied to (\ref{BFD4b}) in the experiments below, in order to accelerate the convergence. The procedure is performed up to a maximum number of iterations or when the magnitude of the residual error
\begin{eqnarray*}
R^{[n]}=(R_{k}^{[n]})_{k=0}^{N-1},\;R_{k}^{[n]}:=\mathcal{L}_{k}z^{[n]}-\mathcal{N}(z^{[n]}),
\end{eqnarray*}
is, in Euclidean norm, less than a prescribed tolerance. (In the experiments, this was taken about $10^{-12}$.)
\subsection{Numerical experiments and discussion of the results}
In this section we will check the performance of the numerical generation of the solitary waves with (\ref{BFD4b}) and illustrate the theoretical results given in \cite{AnguloS2019}. Furthermore, we will show that the numerical experiments suggest new conditions for the existence of solitary waves, finally proved in Appendix \ref{appendixA}.

For the sake of simplicity, in the computations below we will reduce the dependence on the parameters as follows: define $\epsilon_{d-b}=d-b$. By fixing the values of the modelling parameters $\alpha_{1}, \alpha_{2}$ so as to have $b,d\geq 0, a,c\leq 0$, then  $\beta=\displaystyle\frac{1}{1-\alpha_{2}}\left(\frac{\alpha_{1}}{3}+\epsilon_{d-b}\right)$. The parameter $\epsilon_{d-b}$ will serve us to distinguish the Hamiltonian and the nonhamiltonian cases. In addition, we will also fix the values of $\epsilon, \mu$ and $\mu_{2}$. Specifically, in the experiments below we take $\alpha_{1}=1, \alpha_{2}=-1/2, \epsilon=\mu=1,\mu_{2}=10$. By (\ref{BB1b})
\begin{eqnarray}
a=-\beta, b=\frac{1}{3}, c=-\frac{\beta}{2}, d=b+\epsilon_{d-b},\beta=\frac{2}{3}\left(\frac{1}{3}+\epsilon_{d-b}\right).\label{321}
\end{eqnarray} 
The existence of solutions of (\ref{BFD2}) for the Hamiltonian case $b=d>0, a,c\leq 0$ is analyzed in \cite{AnguloS2019}. More specifically, smooth solitary wave solutions of (\ref{BFD2}) are shown to exist for speeds $c_{s}$ satisfying 
\begin{equation}
|c_{s}|<\omega_{m},\quad \gamma^{2}\alpha_{0}(c_{s})>\sqrt{\frac{\mu}{\mu_{2}}},\label{cond1}
\end{equation}
 where $\omega_{m}=(1-\gamma)|c|/b$ and
\begin{equation}\label{cond2}
\alpha_{0}=\alpha_{0}(c_{s})=\frac{1}{\gamma}-|c_{s}|-\frac{1}{\beta_{0}},\;
\beta_{0}=\beta_{0}(c_{s})=-4\gamma^{4}\left(b|c_{s}|+\frac{1}{\gamma}\left(a-\frac{1}{\gamma^{2}}\right)\right).
\end{equation}
 Futhermore, it is proved that the solitary waves decay exponentially as $|X|\rightarrow\infty$.
\begin{figure}[htbp]
\centering
\subfigure[]
{\includegraphics[width=7cm]{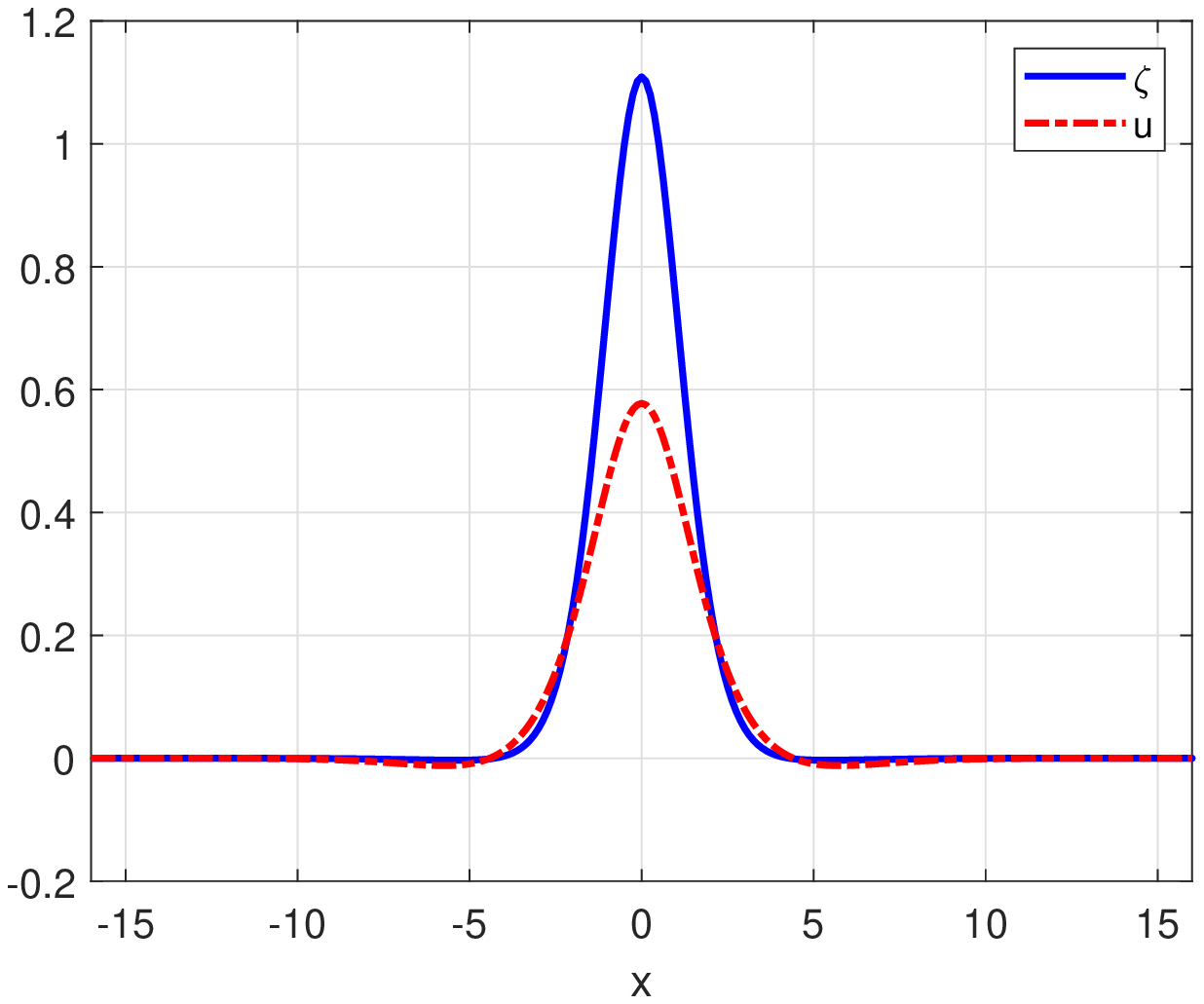}}
\subfigure[]
{\includegraphics[width=6.2cm]{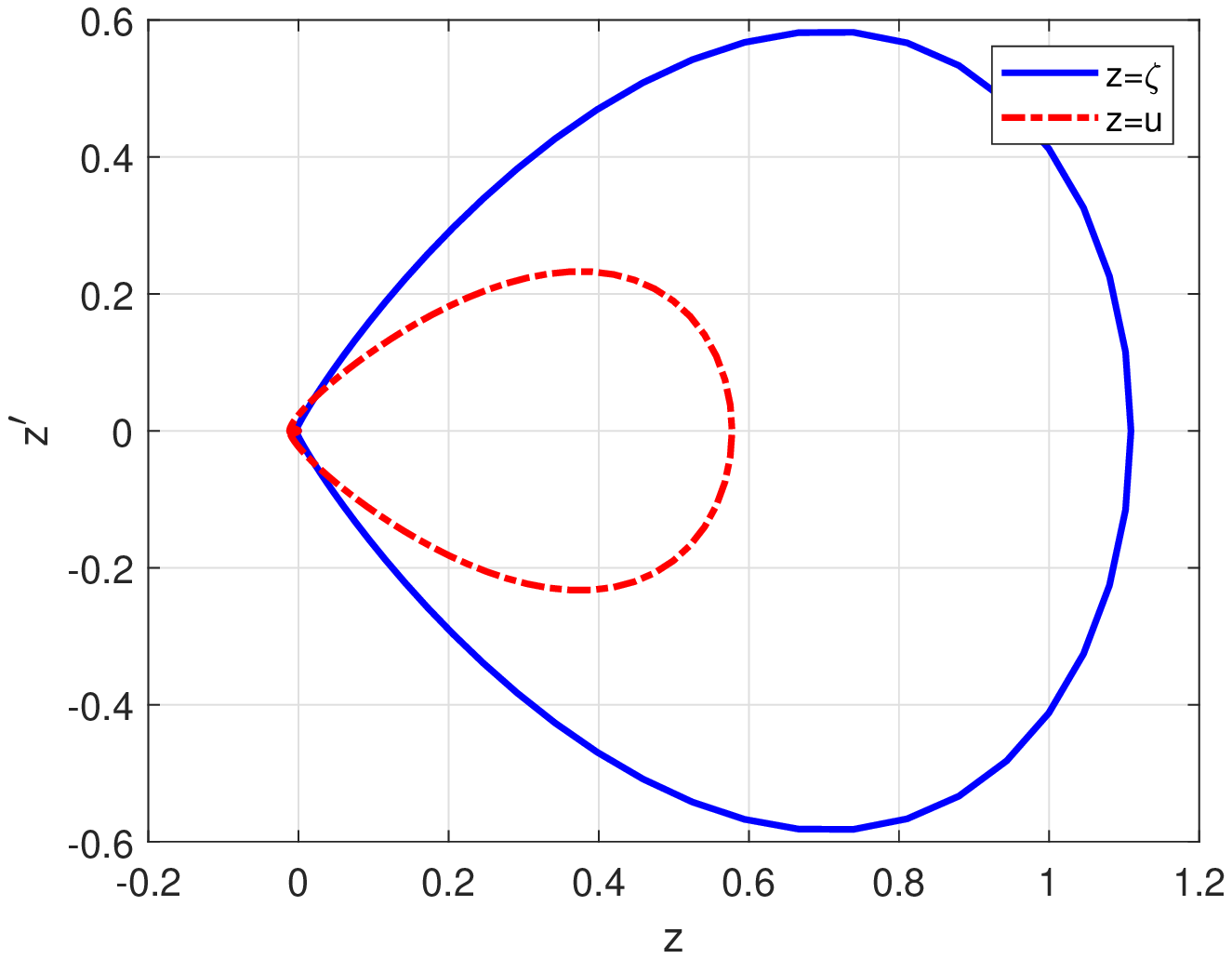}}
\subfigure[]
{\includegraphics[width=6.2cm]{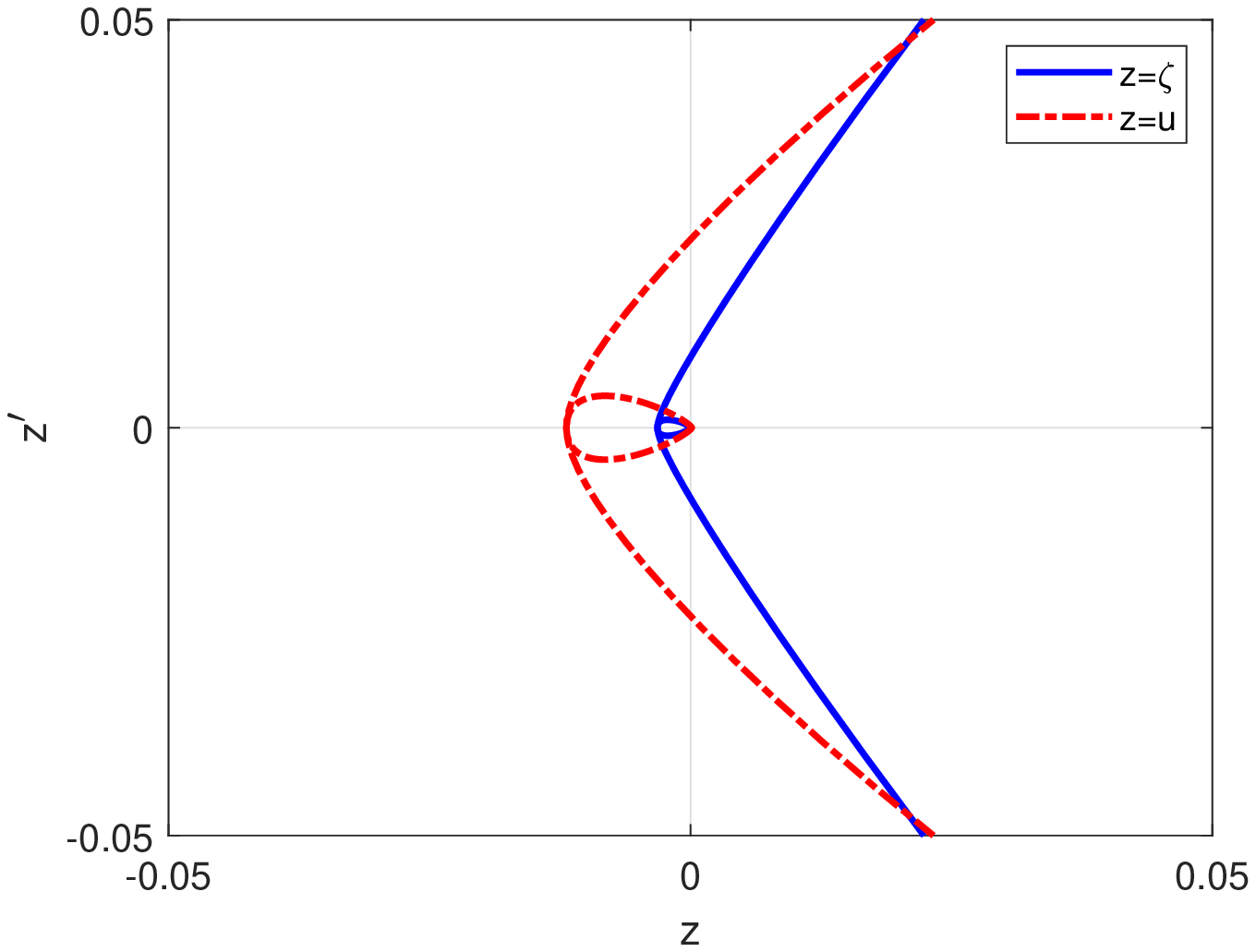}}
\caption{Approximate solitary wave profiles for the case $\epsilon_{d-b}=0$, $a,b,c,d$ given by (\ref{321}), $\gamma=0.8, c_{s}=5\times 10^{-2}$. (a) $\xi$ and $u$ profiles; (b) $\zeta$ and $u$ phase portraits; (c) magnification of (b).}
\label{Fig_one}
\end{figure}
Figure \ref{Fig_one}(a) shows the numerical profile computed by running (\ref{BFD4b}) with $L=256, N=4096,\gamma=0.8$ (thus $\omega_{m}=6.6667\times 10^{-2}$), $c_{s}=5\times 10^{-2}$, and an initial iteration given by a profile of ${\rm sech}^{2}$-type, are shown in Figure \ref{Fig_one}. The method attains an approximate solitary-wave profile with a residual error below $10^{-12}$ in $35$ iterations. The phase portraits of the computed profiles show that the exponential decay of the waves is not monotone.
\begin{remark}
Let $\gamma\in (0,1), b=d>0, a,c<0$. We observe that the last condition in (\ref{cond1}) is equivalent to requiring 
\begin{equation*}\label{cond3}
Q(|c_{s}|)>0,
\end{equation*}
where $Q(x)=Q_{0}+Q_{1}x+x^{2}$ with
\begin{eqnarray}
Q_{0}(\gamma)&=&\frac{1}{b\gamma^{2}}\left(\frac{1}{\gamma}\left(\sqrt{\frac{\mu}{\mu_{2}}}-\gamma\right)\left(a-\frac{1}{\gamma^{2}}\right)-\frac{1}{4\gamma^{2}}\right),\label{q0}\\
Q_{1}(\gamma)&=&\frac{1}{b\gamma^{2}}\left(b\left(\sqrt{\frac{\mu}{\mu_{2}}}-\gamma\right)+\gamma\left(a-\frac{1}{\gamma^{2}}\right)\right).\nonumber
\end{eqnarray}
It can be seen that, under the conditions on the parameters given above, we have $Q_{1}<0$. Let $x_{\pm}(\gamma)$ be the roots of $Q(x)$, 
\begin{eqnarray*}\label{roots}
x_{\pm}(\gamma)=\frac{1}{2}\left(-Q_{1}(\gamma)\pm\sqrt{Q_{1}(\gamma)^{2}-4Q_{0}(\gamma)}\right).
\end{eqnarray*} 
Then $x_{-}(\gamma)$ is real and positive (with $0<x_{-}(\gamma)<x_{+}(\gamma)$) if and only if $Q_{0}(\gamma)>0$, that is
\begin{eqnarray}\label{311}
\sqrt{\frac{\mu}{\mu_{2}}}<C(\gamma):=\gamma-\frac{1}{4\gamma}\frac{1}{\frac{1}{\gamma^{2}}-a}=\frac{\gamma(3-4a\gamma^{2})}{4(1-a\gamma^{2})}.
\end{eqnarray}
Therefore, (\ref{cond1}) can be rewritten as
\begin{eqnarray*}
|c_{s}|<\min\{\omega_{m},x_{-}(\gamma)\},
\end{eqnarray*}
for values of $\mu/\mu_{2}$ satisfying (\ref{311}).

Note also that (\ref{311}) can be read in a two-fold way: a condition on $\sqrt{\mu/\mu_{2}}$ for fixed $\gamma\in (0,1)$ or a condition on $\gamma$ for fixed $\sqrt{\mu/\mu_{2}}$. In this last sense, we have the following result:
\begin{lemma}\label{lemm31}
If
\begin{eqnarray}\label{312}
\sqrt{\frac{\mu}{\mu_{2}}}<\frac{3+|a|}{4+|a|}.
\end{eqnarray}
Then there exists $\gamma_{*}\in (0,1)$ such that (\ref{311}) holds for $\gamma\in (\gamma_{*},1)$. The value of $\gamma_{*}$ satisfies $Q_{0}(\gamma_{*})=0$, where $Q_{0}$ is given by (\ref{q0}), that is
\begin{eqnarray}\label{313}
\gamma_{*}^{3}-\sqrt{\frac{\mu}{\mu_{2}}}\gamma_{*}^{2}+\frac{3}{4|a|}\gamma_{*}-\frac{1}{|a|}\sqrt{\frac{\mu}{\mu_{2}}}=0.
\end{eqnarray}
\end{lemma}
\begin{proof}
In terms of $\gamma$ and for $\sqrt{\mu/\mu_{2}}$, condition (\ref{311}) reads $4|a|P(\gamma)>0$, where $P$ is the polynomial
$$P(\gamma)=\gamma^{3}-\sqrt{\frac{\mu}{\mu_{2}}}\gamma^{2}+\frac{3}{4|a|}\gamma-\frac{1}{|a|}\sqrt{\frac{\mu}{\mu_{2}}},
$$
which can be defined for $\gamma\in [0,1]$. Note that $P(0)<0$ and, using (\ref{312}), $P(1)>0$. Then there exists $\gamma_{*}\in (0,1)$ such that $P(\gamma_{*})=0$. Furthermore, it is not hard to see that if (\ref{312}) holds, then $P(\gamma)$ is increasing for $\gamma\in [0,1]$, which implies that $\gamma_{*}$ is unique and that $P(\gamma)>0$ for $\gamma\in (\gamma_{*},1)$, and the lemma follows.
\end{proof}
Observe that $\gamma_{*}$ can be estimated by solving (\ref{313}) with some iterative procedure.
\end{remark}

It was observed that the method (\ref{BFD4b}) still generates approximate solitary wave profiles for speeds beyond (\ref{cond1}), (\ref{cond2}). The experiments suggest the existence of solitary wave solutions with $|c_{s}|<c_{\gamma}$ for some $c_{\gamma}$ specified in Appendix \ref{appendixA}.
%
\begin{figure}[htbp]
\centering
\subfigure[]
{\includegraphics[width=6.2cm]{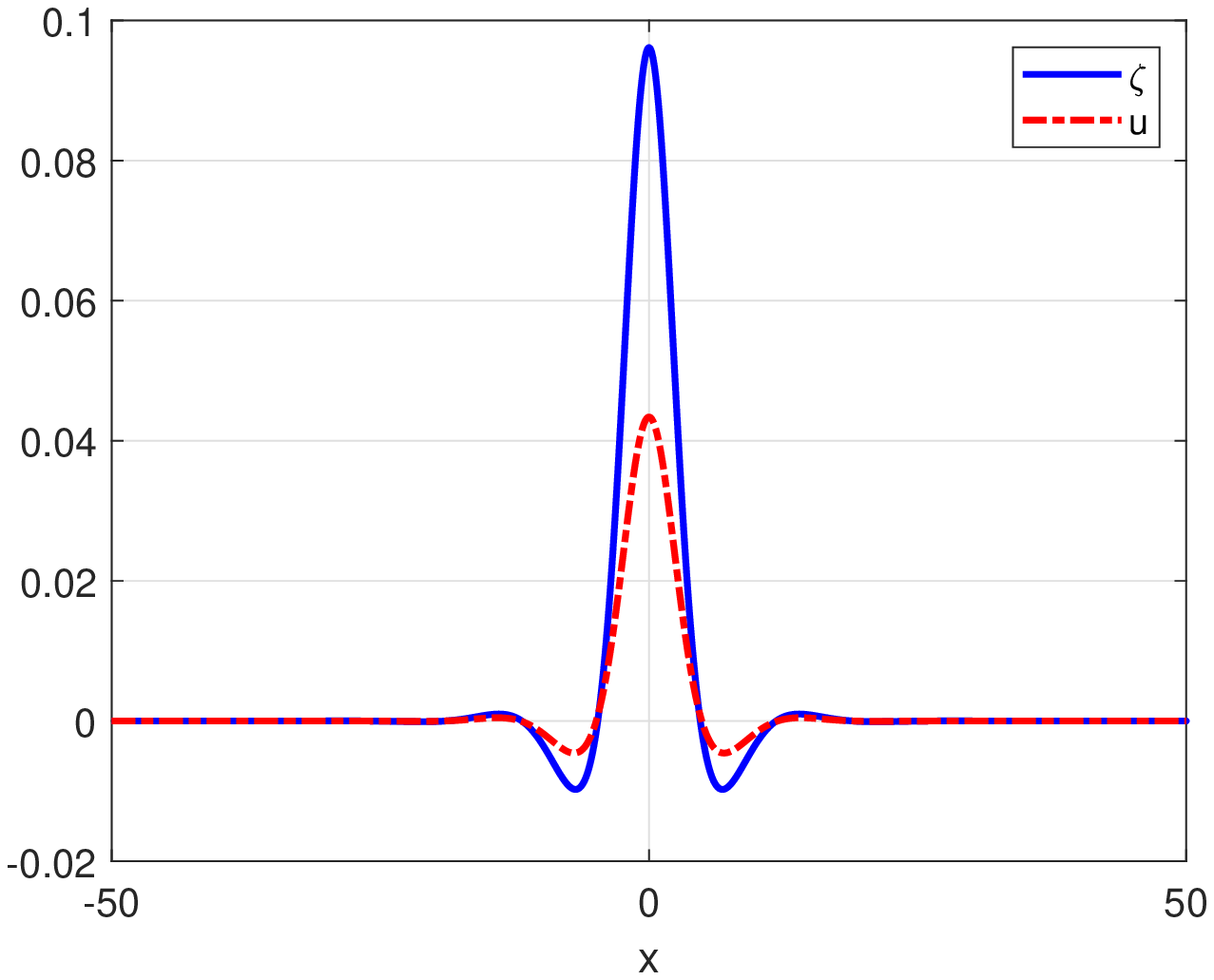}}
\subfigure[]
{\includegraphics[width=6.2cm]{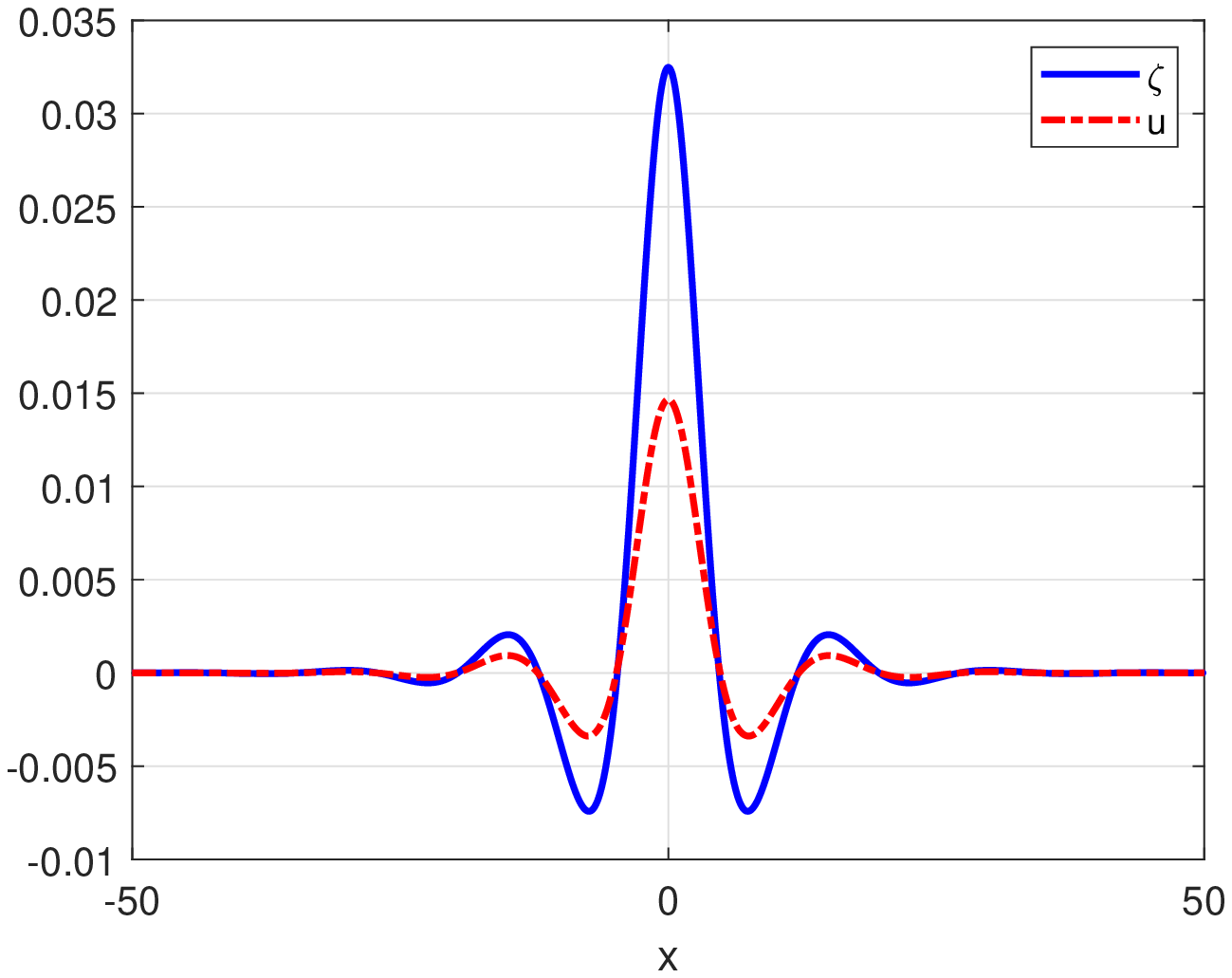}}
\caption{Approximate solitary wave profiles for the case $\epsilon_{d-b}=0$, $a,b,c,d$ given by (\ref{321}), $\gamma=0.8$. (a) $c_{s}=4\times 10^{-1}$; (b) $c_{s}=4.2\times 10^{-1}$.}
\label{Fig_two}
\end{figure}
This is here illustrated in Figure \ref{Fig_two}, which shows two of these computed profiles for the data used in Figure \ref{Fig_one}, for which the bound $c_{\gamma}$ is approximately $4.2646\times 10^{-1}$ (cf. Appendix \ref{appendixA}). It is also noted that as the speed approaches the limit $c_{\gamma}$, the computed profiles reduce their amplitudes (suggesting that this parameter is an increasing function of $c_{\gamma}-c_{s}$) and increase the nonmonotone behaviour, generating more oscillations. 

The numerical experiments also suggest that the nonmonotone decay also depends on the size of $\nu=\mu/\mu_{2}$ and $\gamma$, in the sense that the smaller these parameters, the more the computed wave oscillates in its decay at infinity. Figure \ref{Fig_three_a} illustrates this fact, with computed solitary waves corresponding to $\nu=1,1/4,1/8$ with $\gamma=0.8$ fixed and $c_{s}=c_{\gamma}-10^{-2}$ for the three cases. (The rest of the parameters is the same as in Figure \ref{Fig_one}.)
\begin{figure}[htbp]
\centering
\subfigure[]
{\includegraphics[width=4cm]{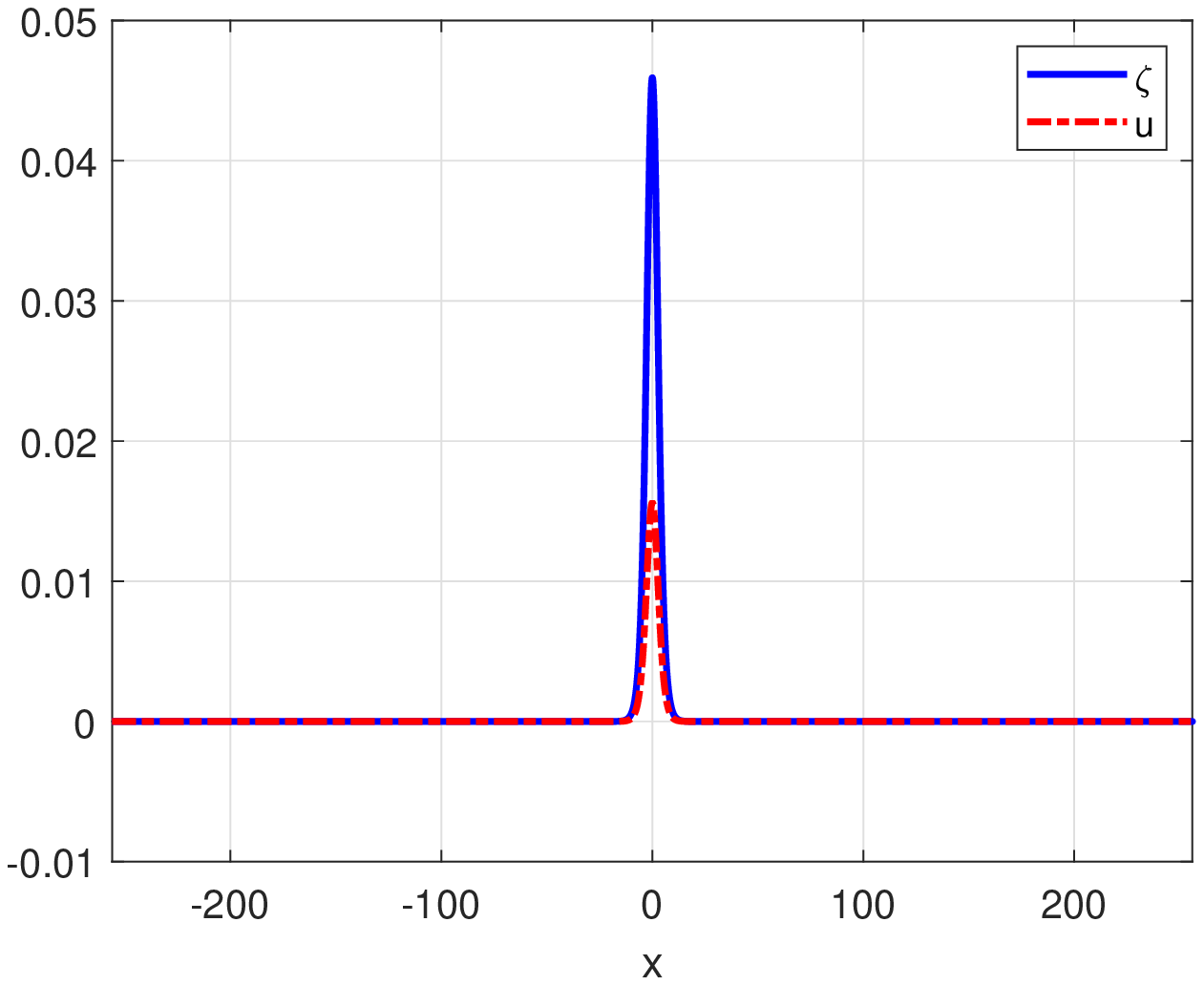}}
\subfigure[]
{\includegraphics[width=4cm]{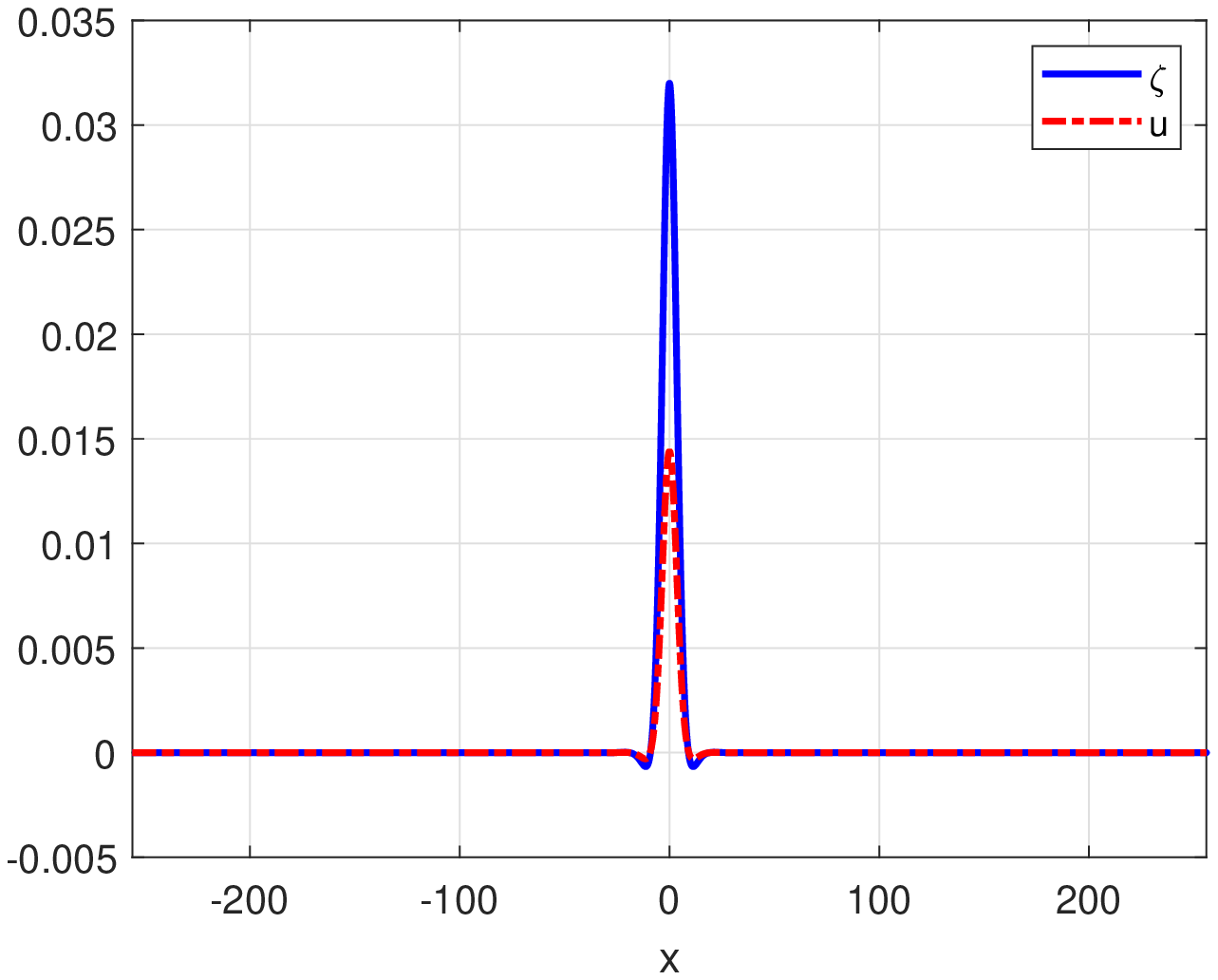}}
\subfigure[]
{\includegraphics[width=4cm]{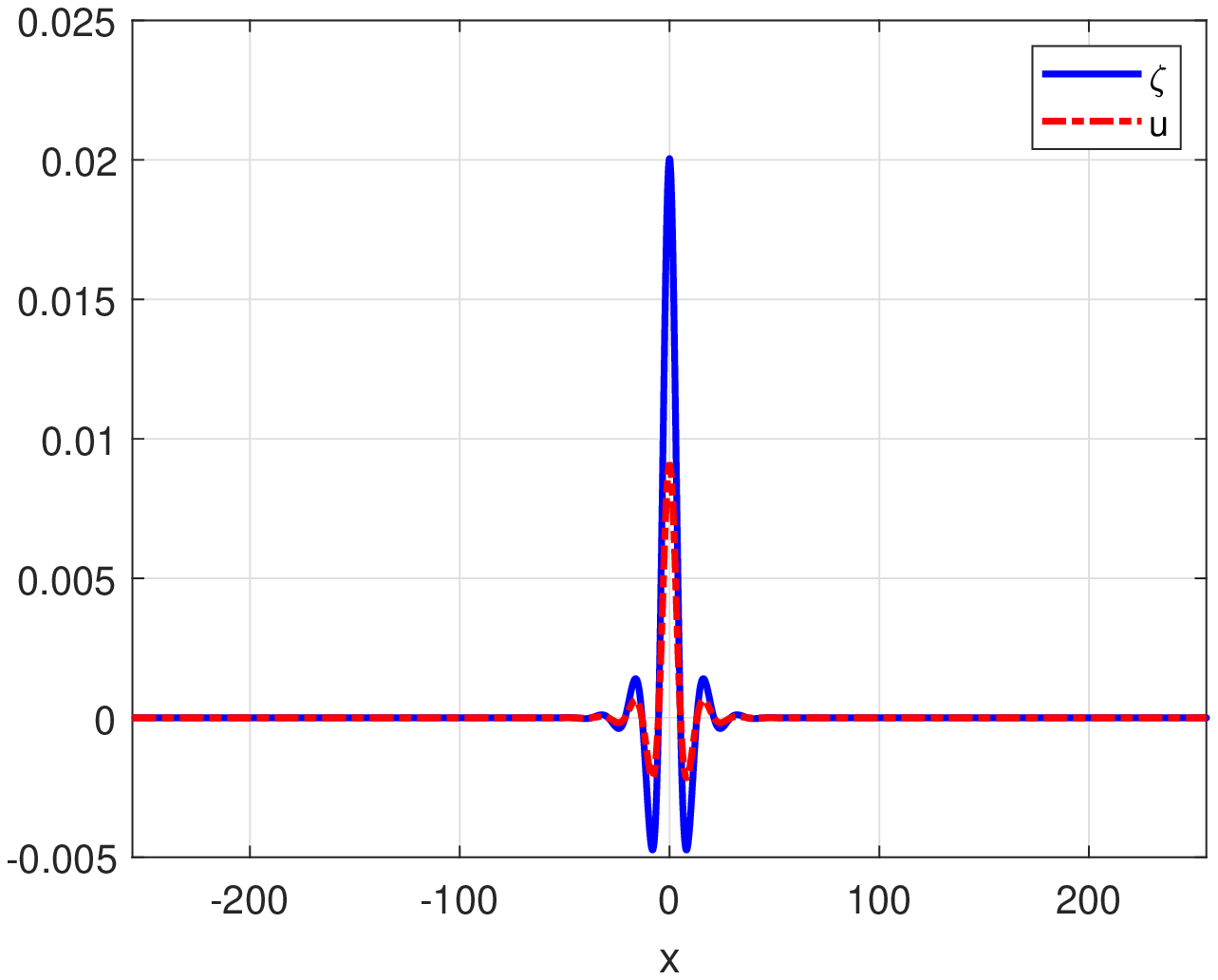}}
\caption{Approximate solitary wave profiles for the case $\epsilon_{d-b}=0$, $a,b,c,d$ given by (\ref{321}), $\gamma=0.8$, and $c_{s}=R(\gamma)-10^{-2}$. (a) $\nu=\mu/\mu_{2}=1$; (b) $\nu=\mu/\mu_{2}=1/4$; (c) $\nu=\mu/\mu_{2}=1/8$.}
\label{Fig_three_a}
\end{figure}
In a similar way, the influence of $\gamma$ is shown in Figure \ref{Fig_three_b}, where the approximate solitary waves are generated with fixed $\nu=1/10$ and three values $\gamma=0.3, 0.6, 0.8$ respectively.
\begin{figure}[htbp]
\centering
\subfigure[]
{\includegraphics[width=4.2cm]{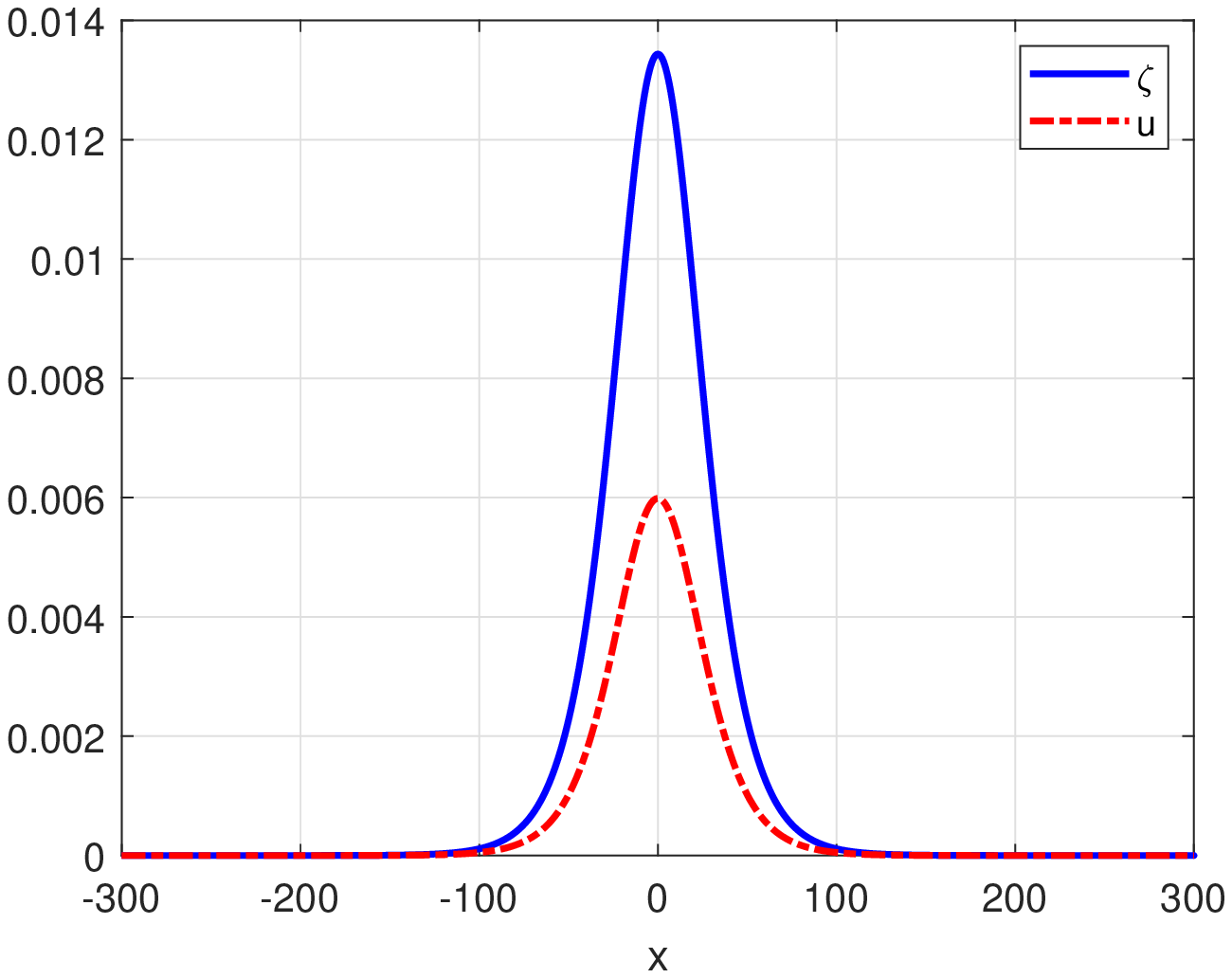}}
\subfigure[]
{\includegraphics[width=4cm]{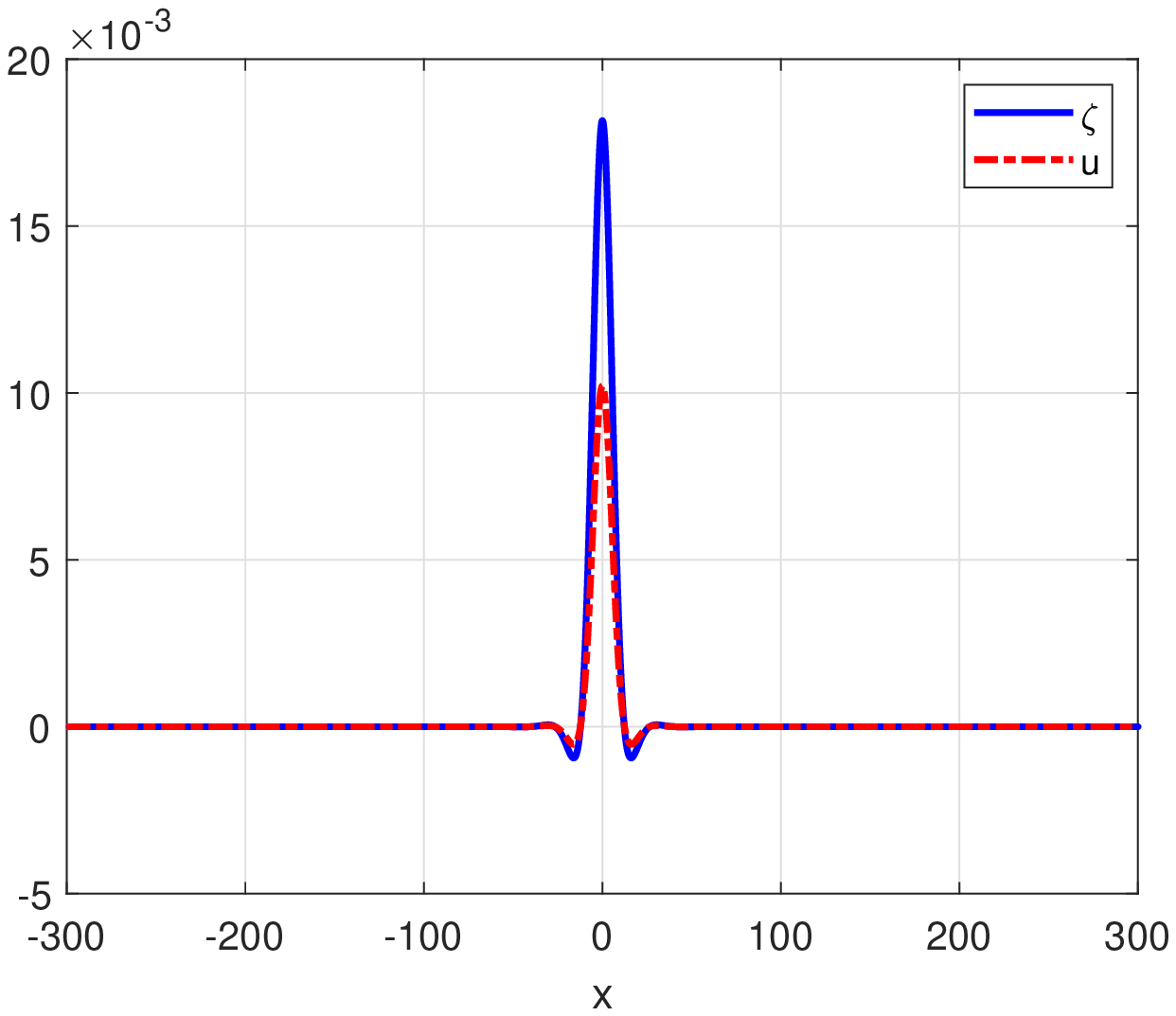}}
\subfigure[]
{\includegraphics[width=4cm]{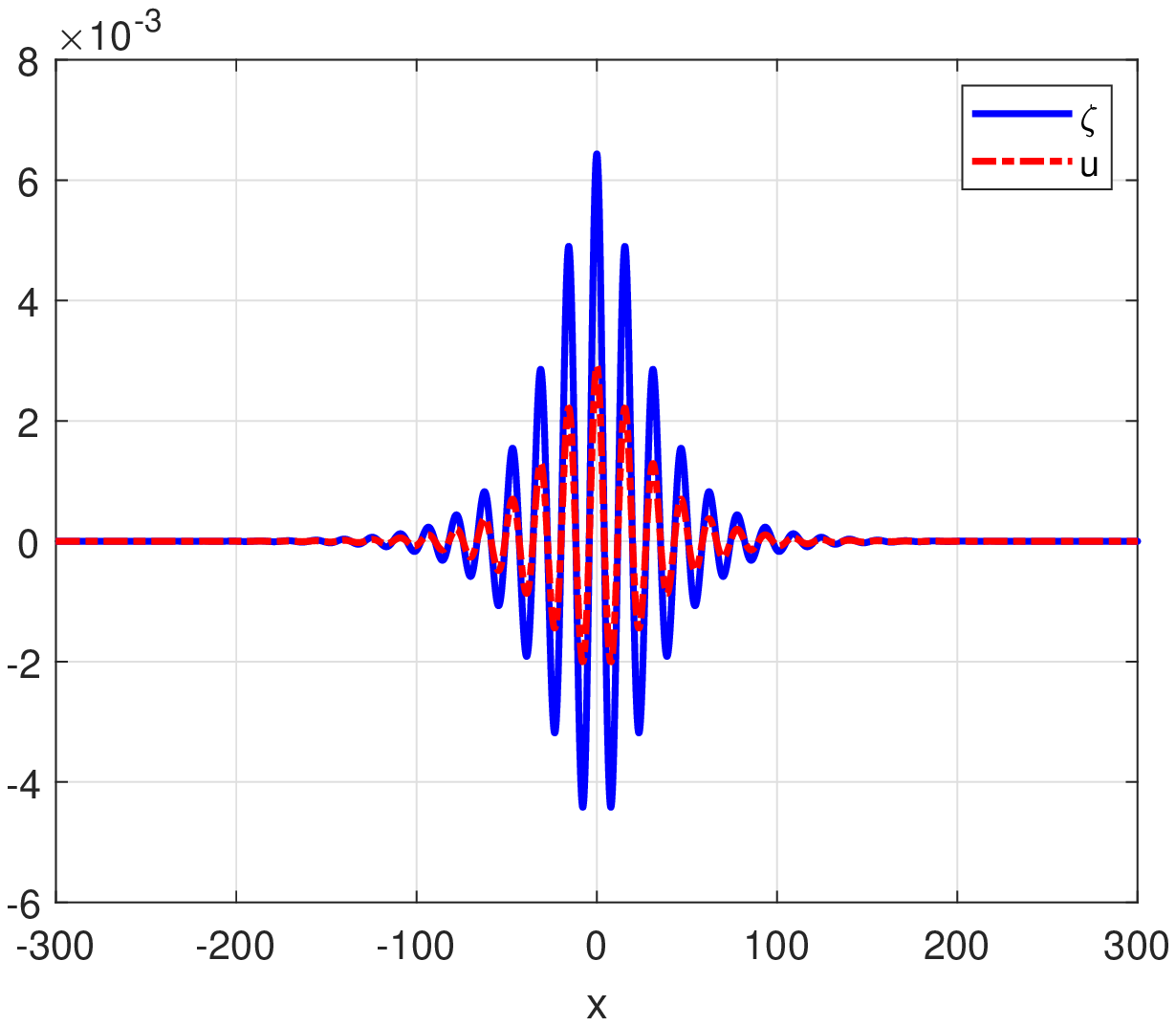}}
\caption{Approximate solitary wave profiles for the case $\epsilon_{d-b}=0$, $a,b,c,d$ given by (\ref{321}), $\nu=\mu/\mu_{2}=0.1$, and $c_{s}=R(\gamma)-10^{-2}$. (a) $\gamma=0.3$; (b) $\gamma=0.6$; (c) $\gamma=0.8$.}
\label{Fig_three_b}
\end{figure}
The numerical experiments shown above are all concerned with the \lq generic\rq\ B/FD case, that is, $b,d>0, a,c<0$. Similar results can be obtained for the two other cases of B/FD systems considered in section \ref{sec2}, but they will not be shown here.
\subsubsection{Nonhamiltonian case}
The method (\ref{BFD4b}) also suggests the existence of solitary wave solutions of (\ref{BFD}) in the nonhamiltonian case (represented, in this experimental section, by nonzero values of $\epsilon_{d-b}$ in (\ref{321})). Focused again on the \lq generic\rq\ B/FD systems, this case is illustrated in Figure \ref{Fig_four} for two values of $\epsilon_{d-b}$.
\begin{figure}[htbp]
\centering
\subfigure[]
{\includegraphics[width=6cm]{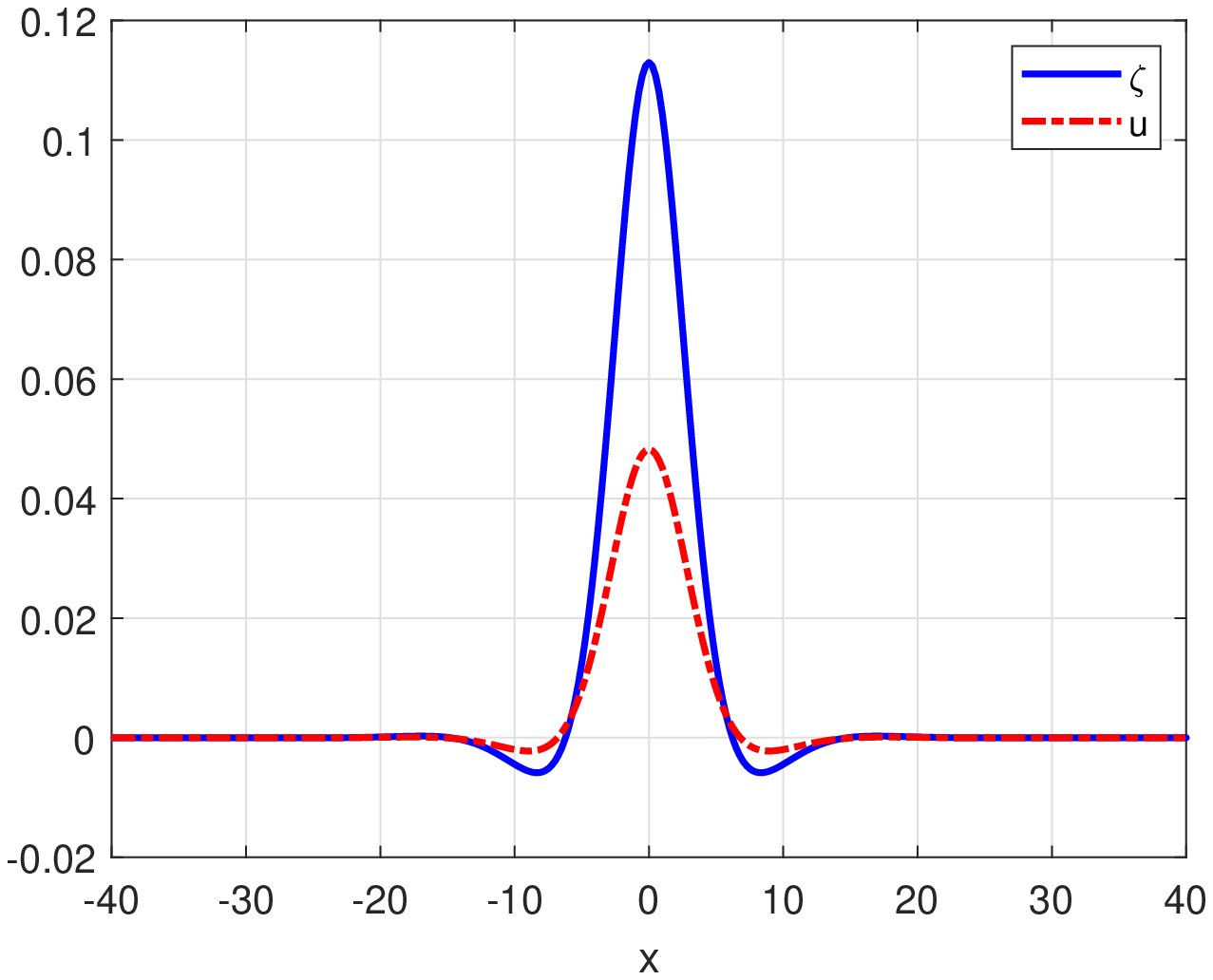}}
\subfigure[]
{\includegraphics[width=6.3cm]{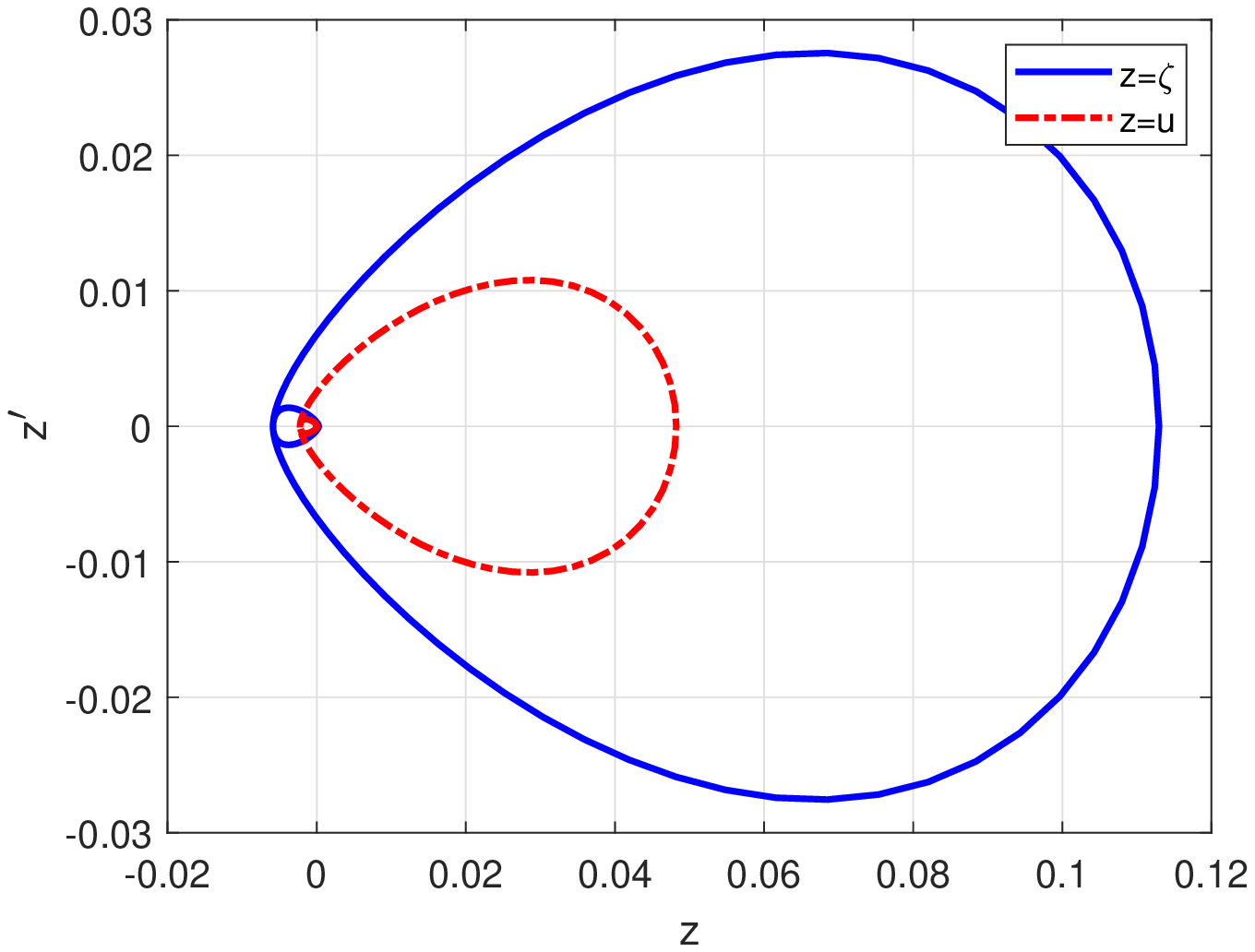}}
\subfigure[]
{\includegraphics[width=6cm]{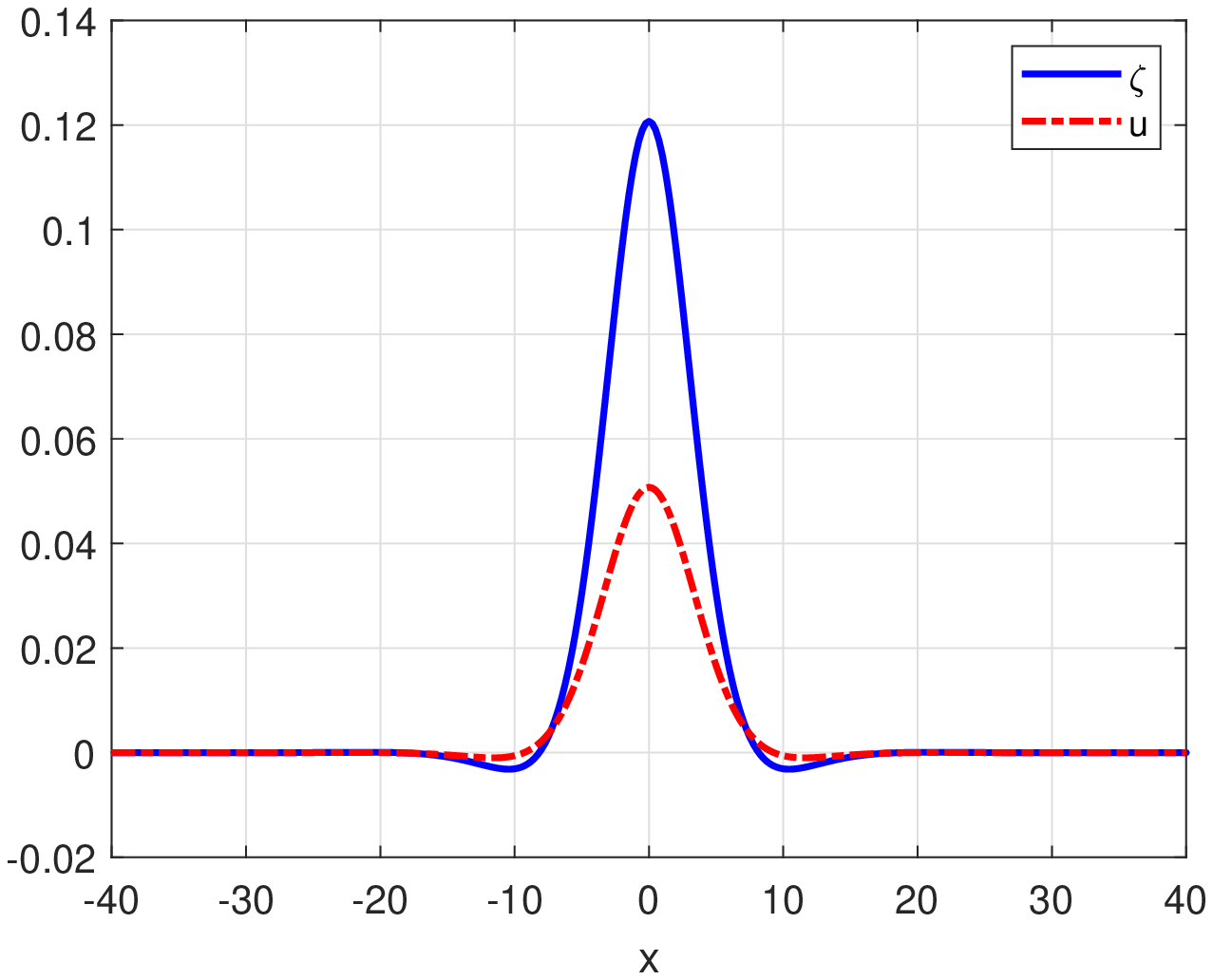}}
\subfigure[]
{\includegraphics[width=6.3cm]{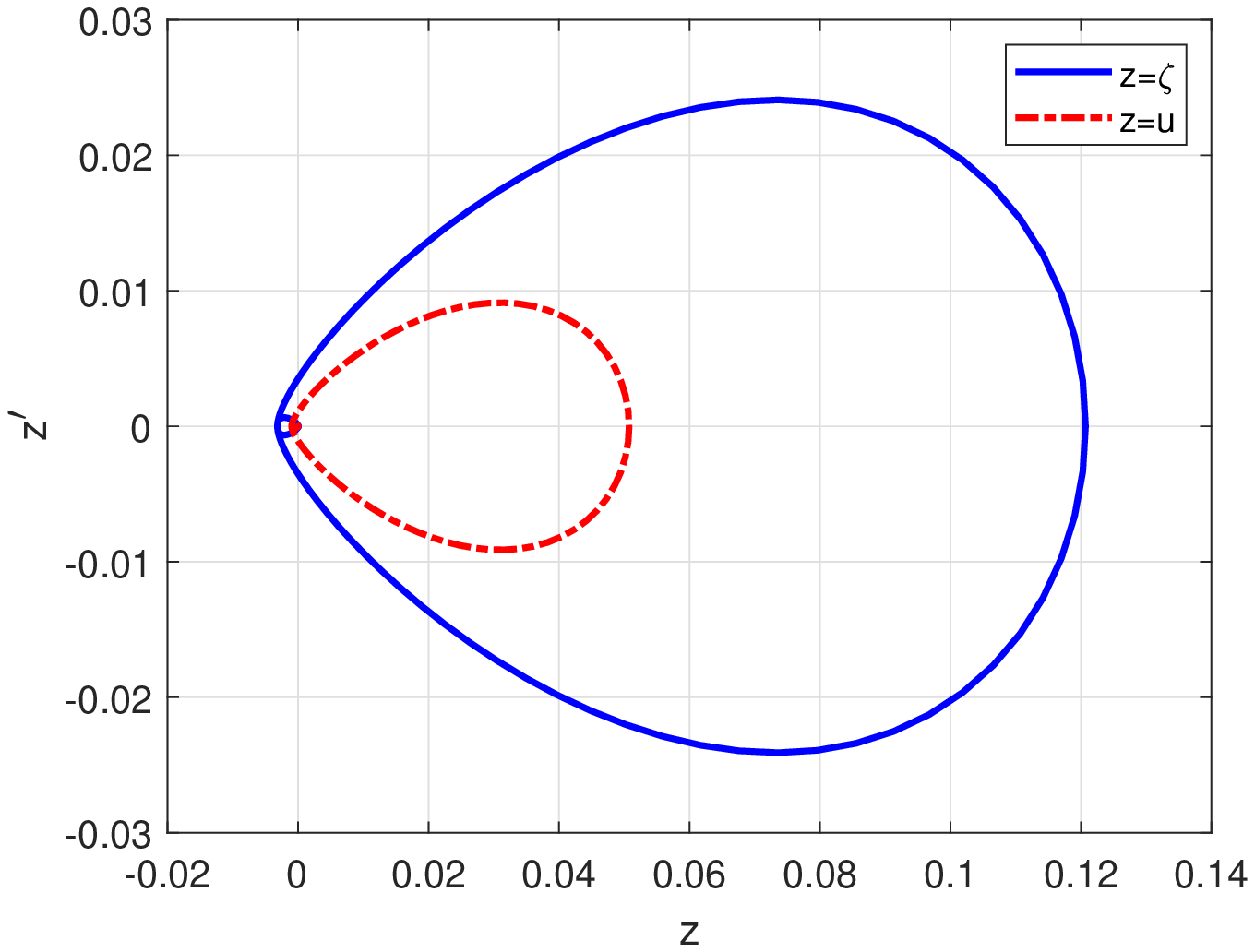}}
\caption{Approximate solitary wave profiles and corresponding phase portraits for $a,b,c,d$ given by (\ref{321}), $\gamma=0.8$, $c_{s}=4\times 10^{-1}$, $\epsilon=\mu=1, \mu_{2}=10$, and (a), (b) $\epsilon_{d-b}=1$, (c), (d) $\epsilon_{d-b}=2$.}
\label{Fig_four}
\end{figure}
As far as the properties of the solitary waves are concerned, the computations suggest that they are similar to those of the waves generated in the Hamiltonian case. The solitary waves decay exponentially in a nonmonotone way and for fixed $\gamma$, there seems to be some analogous speed limit  such that as the magnitude of the speed approaches this limit, the corresponding amplitudes decrease and the nonmonotone behaviour of the decay is stronger. Figure \ref{Fig_five} illustrates this fact. It shows the solitary waves computed with (\ref{BFD4b}) with $\epsilon=\mu=1, \mu_{2}=10$,  $a,b,c,d$ given by (\ref{321}) with $\epsilon_{d-b}=2$, $\gamma=0.8$, and three speeds $c_{s}=0.2,0.3,0.4$.
\begin{figure}[htbp]
\centering
\subfigure[]
{\includegraphics[width=6.2cm]{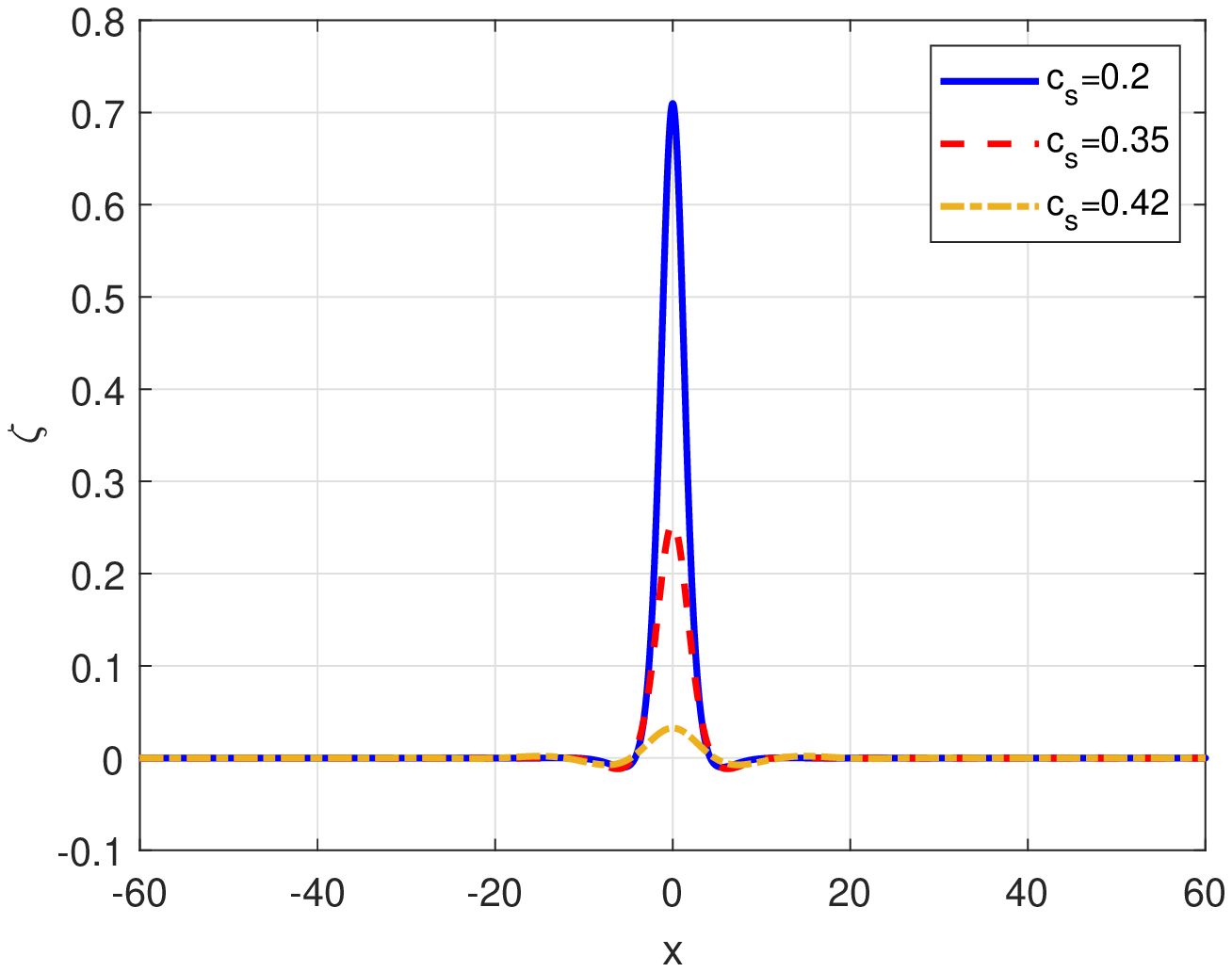}}
\subfigure[]
{\includegraphics[width=6.2cm]{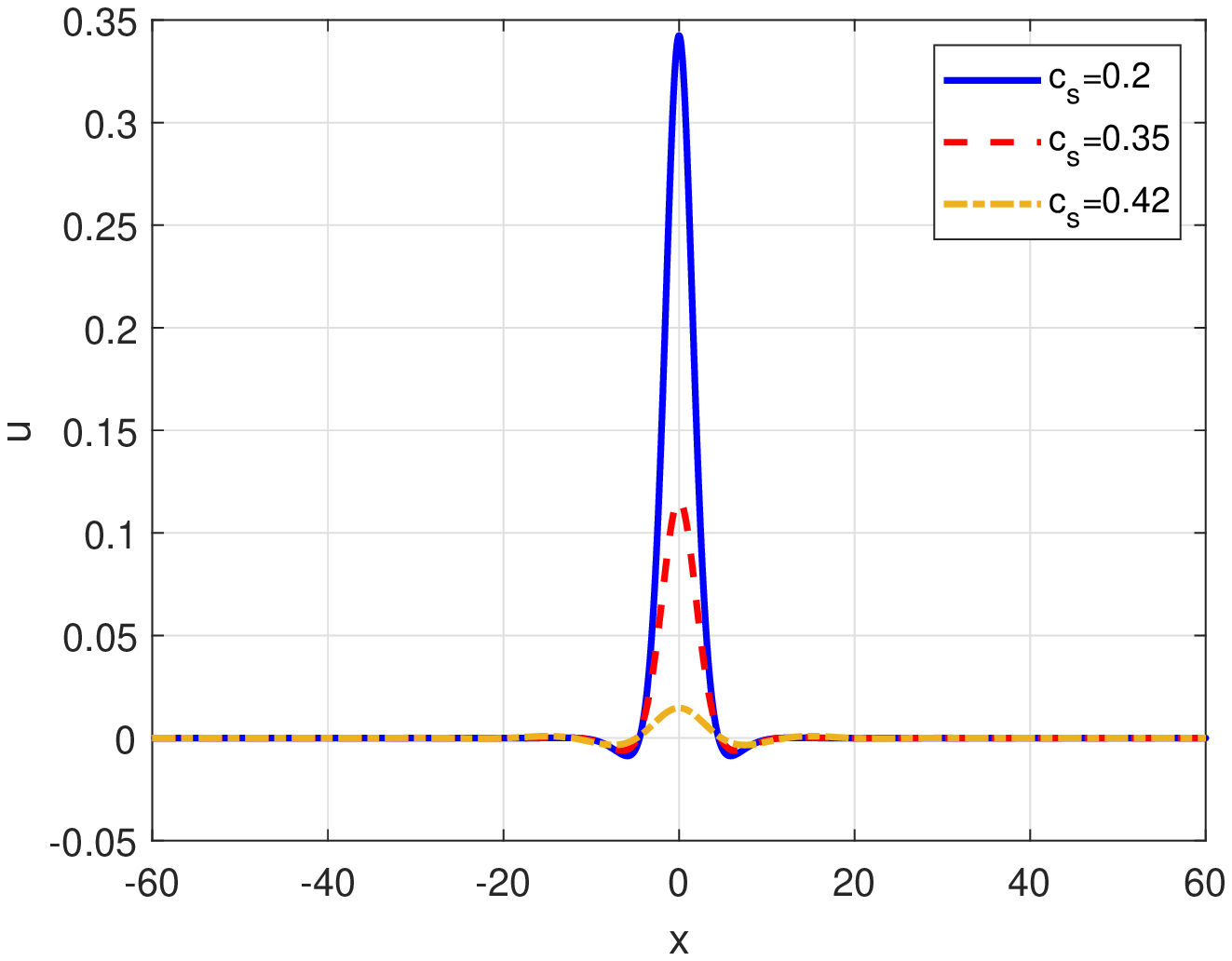}}
\caption{Approximate solitary wave profiles and corresponding phase portraits for $a,b,c,d$ given by (\ref{321}), $\gamma=0.8$, $\epsilon=\mu=1, \mu_{2}=10$, $\epsilon_{d-b}=2$, and three values of $c_{s}$. (a) $\zeta$ profile; (b) $u$ profile.}
\label{Fig_five}
\end{figure}
\section{Solitary wave solutions. Computational study of the dynamics}
\label{sec4}
In this section we study, by computational means, some aspects of the dynamics of the solitary wave solutions of (\ref{BFD}). We first introduce the numerical method used for the computations and then analyze numerically the stability of the solitary waves (computed from the technique described in section \ref{sec3}) under small and large perturbations, the interactions of solitary waves, and the resolution of general initial data into a train of solitary waves. The experiments below are concerned with the Hamiltonian case into the \lq generic\rq\ B/FD systems ($b=d>0, a,c<0$); for the nonhamiltonian case we did not observe relevant differences in the corresponding solitary wave dynamics. The other two types of B/FD systems considered in section \ref{sec2} will be studied elsewhere.
\subsection{Description of the numerical method}
In this section we introduce the numerical scheme used in the computational study and illustrate its accuracy and performance with some experiments of validation.
\subsubsection{Semidiscretization in space}
For the spatial discretization, we will take advantage of the convergence analysis developed in section \ref{sec2} and the scheme of approximation of solitary wave profiles introduced in section \ref{sec3} and discretize the periodic ivp of (\ref{BFD1}) in space with the Fourier spectral method, formulated in collocation form and based on the nodes (\ref{34b}). (The scheme is equivalent to the Fourier-Galerkin method in the sense specified in e.~g., \cite{CHQZ}.) Then, for an integer $N\geq 1$, the semidiscrete solution, defined as a map $(\zeta^{N},u^{N}):[0,\infty)\rightarrow S_{N}$ satisfying (\ref{BFD1}) at the collocation points (\ref{34b}), is represented by the nodal values $Z_{N}(t)=(\zeta^{N}(x_{j},t))_{j=0}^{N-1}, U_{N}(t)=(u^{N}(x_{j},t))_{j=0}^{N-1}$ which satisfy the semidiscrete system
\begin{eqnarray}
J_{b,h}\frac{d}{dt}Z_{N}+\mathcal{L}_{\mu_{2},h}(D_{N}U_{N})-\frac{\epsilon}{\gamma}(D_{N}(Z_{N}.U_{N}))&=&0,\nonumber\\
J_{d,h}\frac{d}{dt}U_{N}+(1-\gamma)J_{c,h}(D_{N}Z_{N})-\frac{\epsilon}{2\gamma}(D_{N}(U_{N}.U_{N}))&=&0,\label{BFD23}
\end{eqnarray}
for $0\leq j\leq N-1$, where the operators $J_{\alpha,h}, \alpha=b,d,c, \mathcal{L}_{\mu_{2},h}$, and $D_{N}$ are defined in (\ref{BFD4}). Recall that the nonlinear terms contain products that are understood in the Hadamard sense. The Fourier representation of (\ref{BFD23}) makes use of  (\ref{BFD2c}) and takes the form
\begin{eqnarray*}
j_{b}(\widetilde{k})\frac{d}{dt}\widehat{\zeta^{N}}(k,t)+(i\widetilde{k}){l}_{\mu_{2}}(\widetilde{k})\widehat{u^{N}}(k,t)-\frac{\epsilon}{\gamma}(i\widetilde{k})\widehat{(\zeta^{N}u^{N})}(k,t)&=&0,\nonumber\\
j_{d}(\widetilde{k})\frac{d}{dt}\widehat{u^{N}}(k,t)+(1-\gamma)(i\widetilde{k})j_{c}(\widetilde{k})\widehat{\zeta^{N}}(k,t)-\frac{\epsilon}{2\gamma}(i\widetilde{k})\widehat{(u^{N})^{2}}(k,t)&=&0,\label{BFD24}
\end{eqnarray*}
where $\widehat{\zeta^{N}}(k,t), \widehat{u^{N}}(k,t)$ denote, respectively, the $k$th discrete Fourier coefficient of $\zeta^{N}(t), u^{N}(t)$, and $\widetilde{k}=\pi k/L, k=0,\ldots,N-1$.

We note that when $b=d$ the ODE system (\ref{BFD23}) has a Hamiltonian structure
\begin{eqnarray*}
\partial_{t}\begin{pmatrix}
Z_{N}\\U_{N}
\end{pmatrix}=\mathcal{J}_{N}\nabla E_{h}(Z_{N},U_{N}),
\end{eqnarray*}
where $\mathcal{J}_{N}$ is the $2N\times 2N$ matrix with $N\times N$ blocks
\begin{eqnarray*}
\mathcal{J}:=-\begin{pmatrix}
0&J_{b,h}^{-1}D_{N}\\J_{b,h}^{-1}D_{N}&0
\end{pmatrix}
\end{eqnarray*}
and
\begin{eqnarray}
E_{h}(Z,U)&:=&\frac{1}{2}\left((1-\gamma)\left( Z,J_{c,h}Z\right)_{N}+\left(U,\mathcal{L}_{\mu_{2},h}U\right)_{N}\right)\nonumber\\
&&-\frac{\epsilon}{2\gamma}\left(G_{h}(Z,U),1\right)_{N},\label{BFDdEnergy}\\
G_{h}(Z,U)&:=&Z.U.^{2},\nonumber
\end{eqnarray}
where the products in $G_{h}$ are componentwise. In particular, $E_{h}$ is preserved in time by the solutions of (\ref{BFD23}). The proof of this fact follows from the symmetry of the operators $J_{c,h}$ and $\mathcal{L}_{\mu_{2},h}$, and uses similar arguments to those of \cite{Cano2006}. Note that $hE_{h}$ is the natural discretization, given by the spectral approximation, of the Hamiltonian (\ref{BFDEnergy}). As far as the second invariant (\ref{BFDMom}) is concerned, a natural discretization is $hI_{h}$ where
\begin{equation}\label{BFDdMom}
I_{h}(Z,U)=\langle Z,J_{b,h}U\rangle_{N}.
\end{equation}
The time preservation holds in this case for symmetric solutions, in the sense given by the following result, cf. \cite{Cano2006}.
\begin{theorem}
\label{th41}
Assume that $N$ is even. Let $D$ be the $N$-by-$N$ matrix such that if $U=(U_{1},U_{2},\ldots,U_{N})^{T}$ then 
\begin{equation}\label{Dop}
DU=(U_{N},U_{N-1},\ldots,U_{2},U_{1})^{T},
\end{equation} and let $(Z_{N},U_{N})$ be a solution of (\ref{BFD23}) satisfying
\begin{equation}\label{BFD46}
DZ_{N}(t)=Z{N}(t),\quad DU_{N}(t)=U_{N}(t),\quad t\geq 0.
\end{equation}
Then
\begin{equation}\label{BFD47}
I_{h}(Z_{N}(t),U_{N}(t))=I_{h}(Z_{N}(0),U_{N}(0)),\quad t\geq 0.
\end{equation}
\end{theorem}
\begin{proof}
As in \cite{Cano2006}, taking the symmetric collocation trigonometric interpolant, we have, for $1\leq l,j\leq N$
\begin{equation*}
(D_{N})_{lj}={\rm Re}\left(\frac{\pi}{NL}\sum_{m=-M}^{M-1} \omega_{N}^{m(j-l)}im\right),
\end{equation*}
where $M=N/2, \omega_{N}=e^{-2\pi i/N}$. This implies $DD_{N}=-D_{N}D$, and therefore
\begin{equation}\label{BFD48}
D_{N}^{2}D=DD_{N}^{2},\quad J_{c,h}D=DJ_{c,h}.
\end{equation}
In a similar way
\begin{eqnarray*}
(|D_{N}|)_{lj}&=&{\rm Re}\left(\frac{\pi}{NL}\sum_{m=-M}^{M-1} \omega_{N}^{m(j-l)}|m|\right),\\
(|D_{N}|{\rm coth}\sqrt{\mu_{2}}|D_{N}|)_{lj}&=&{\rm Re}\left(\frac{\pi}{NL}\sum_{m=-M}^{M-1} \omega_{N}^{m(j-l)}|m|{\rm coth}\sqrt{\mu_{2}}|m|\right),
\end{eqnarray*}
and therefore
\begin{equation}\label{BFD49}
\mathcal{L}_{\mu_{2},h}D=D\mathcal{L}_{\mu_{2},h}.
\end{equation}
Let $(Z_{N},U_{N})$ be a solution of (\ref{BFD23}) satisfying (\ref{BFD46}). Then, for all $t\geq 0$
\begin{eqnarray*}
D(Z_{N}.U_{N})&=&(DZ_{N}).(DU_{N})=Z_{N}.U_{N},\\
D(U_{N}.U_{N})&=&(DU_{N}).(DU_{N})=U_{N}.U_{N},
\end{eqnarray*}
and using Lemma 3.4 of \cite{Cano2006} we have
\begin{equation}\label{BFD410}
\langle D_{N}(Z_{N}.U_{N}),U_{N}\rangle_{N}=0,\quad 
\langle Z_{N},D_{N}(U_{N}.U_{N})\rangle_{N}=0.
\end{equation}
We now apply (\ref{BFD23}), (\ref{BFD48})-(\ref{BFD410}) to obtain
\begin{eqnarray*}
\frac{d}{dt}I_{h}(Z_{N},U_{N})&=&\langle\frac{d}{dt}Z_{N},J_{b,h}U_{N}\rangle_{N}+\langle Z_{N},\frac{d}{dt}J_{b,h}U_{N}\rangle_{N}\\
&=&\langle J_{b,h}\frac{d}{dt}Z_{N},U_{N}\rangle_{N}+\langle Z_{N},J_{b,h}\frac{d}{dt}U_{N}\rangle_{N}\\
&=&-\langle \mathcal{L}_{\mu_{2},h}D_{N}U_{N}-\frac{\epsilon}{\gamma}D_{N}(Z_{N}.U_{N}),U_{N}\rangle_{N}\\
&&-\langle Z_{N},(1-\gamma)J_{c,h}D_{N}Z_{N}-\frac{\epsilon}{2\gamma}D_{N}(U_{N}.U_{N})\rangle_{N}\\
&=&-\langle \mathcal{L}_{\mu_{2},h}D_{N}U_{N},U_{N}\rangle_{N}-\langle Z_{N},(1-\gamma)J_{c,h}D_{N}Z_{N}\rangle_{N}\\
&=&-\langle D_{N}U_{N},\mathcal{L}_{\mu_{2},h}U_{N}\rangle_{N}-(1-\gamma)\langle D_{N}Z_{N},J_{c,h}Z_{N}\rangle_{N}.
\end{eqnarray*}
Note now that, from (\ref{BFD46}), (\ref{BFD48}), and (\ref{BFD49})
\begin{eqnarray*}
D\mathcal{L}_{\mu_{2},h}U_{N}&=&\mathcal{L}_{\mu_{2},h}DU_{N}=\mathcal{L}_{\mu_{2},h}U_{N},\\
DJ_{c,h}Z_{N}&=&J_{c,h}DZ_{N}=J_{c,h}Z_{N},
\end{eqnarray*}
and again Lemma 3.4 of \cite{Cano2006} and (\ref{BFD46}) imply
\begin{equation*}
\langle D_{N}U_{N},\mathcal{L}_{\mu_{2},h}U_{N}\rangle_{N}=0,\quad 
\langle D_{N}Z_{N},J_{c,h}Z_{N}\rangle_{N}=0.
\end{equation*}
Therefore, $\frac{d}{dt}I_{h}(Z_{N},U_{N})=0$ and (\ref{BFD47}) follows.
\end{proof}
\subsubsection{Full discretization}
The ode semidiscrete system (\ref{BFD23}) is integrated in time by the Runge-Kutta Composition method of Butcher tableau
\begin{eqnarray}
\label{IMRC}
\begin{array}{c|ccc}
&b_{1}/2& \\[2pt]
&b_{1}&b_{2}/2\\[2pt]
&b_{1}&b_{2}&b_{3}/2\\[2pt]
\hline
\\[-9pt]
 & b_{1}&b_{2}&b_{3} 
 \end{array},
\end{eqnarray}
where $b_{1}=1/(2-2^{1/3}), b_{2}=1-2b_{1}, b_{3}=b_{1}$. The method (\ref{IMRC}) is based on the composition of the implicit midpoint rule with step sizes $b_{j}\Delta t, j=1,2,3$, where $\Delta t$ denotes the time step. The scheme has order four of convergence, it is symplectic and symmetric, \cite{Yoshida1990,SanzSC1994,HLW2004}. Its performance when applied to nonlinear dispersive equations has been analyzed, theoretically and computationally, in several papers, \cite{FrutosS1992,DD,DDS1,DDS_arxiv}. Note that its symplectic character implies the preservation of (\ref{BFDdMom})
\begin{equation*}
I_{h}(\zeta^{n},u^{n})=I_{h}(\zeta^{0},u^{0}),
\end{equation*}
for approximations $(\zeta^{n},u^{n})$ to $(\zeta^{N}(t_{n}), u^{N}(t_{n})), t_{n}=n\Delta t, n=0,1,\ldots,N_{T}$, satisfying
\begin{equation*}
D\zeta^{n}=\zeta^{n},\quad Du^{n}=u^{n},
\end{equation*}
for all $n=0,\ldots,N_{T}, T=N_{T}\Delta t$, and $D$ given by (\ref{Dop}). As in other cases, \cite{DDS1,DDS_arxiv}, for the experiments below we observed experimentally that a Courant condition of the form $Nk=O(1)$ was enough to ensure stability and convergence of the full discretization to smooth solutions of the periodic ivp for (\ref{BFD1}).
\subsubsection{Some experiments of validation}
In this section we check the performance of the fully discrete method. To this end, we consider the approximate solitary wave profile computed in section \ref{sec3} for the values specified in (\ref{321}), $\gamma=0.8$ and speed $c_{s}=0.4$ (cf. Figure \ref{Fig_two}(a)). This solitary wave is taken as initial condition of the scheme and some elements in the evolution of the resulting numerical solution are monitored, namely:
\begin{itemize}
\item The character of the approximation as a solitary wave, that is, as a wave of permanent form and negligible disturbances (Figure \ref{Fig_six}). The experiment also serves as an additional evidence of the accuracy of the technique introduced in section \ref{sec3} to generate numerically solitary wave profiles.
\item The behaviour of the full discretization with respect to the amplitude and speed of the solitary wave. This is illustrated in Figure \ref{Fig_seven}, where the evolution of the corresponding errors is shown for $\Delta t=6.25\times 10^{-3}$. Both figures show the high accuracy of the numerical method when computing the main parameters of the wave and retaining them during the simulation, as another proof of the good performance of the numerical solution to approximate the solitary-wave form.
\item The evolution of the error with respect to the quantities (\ref{BFDdEnergy}) and (\ref{BFDdMom}) is shown in Figure \ref{Fig_eight} for $\Delta t=6.25\times 10^{-3}$. The good behaviour observed is a last guarantee of the full discretization as a reliable choice to perform the experiments about the dynamics of (\ref{BFD1}) in the following sections.
\end{itemize}
\begin{figure}[htbp]
\centering
{\includegraphics[width=0.8\columnwidth]{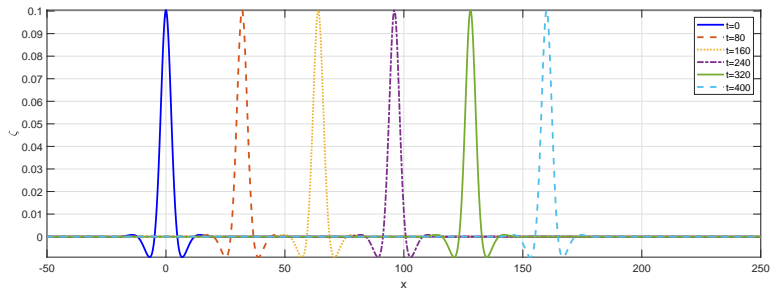}}
\caption{Evolution of the $\zeta$ component of the numerical solution with initial solitary wave profile given in Figure  \ref{Fig_two}(a).}
\label{Fig_six}
\end{figure}
\begin{figure}[htbp]
\centering
\subfigure[]
{\includegraphics[width=6.2cm]{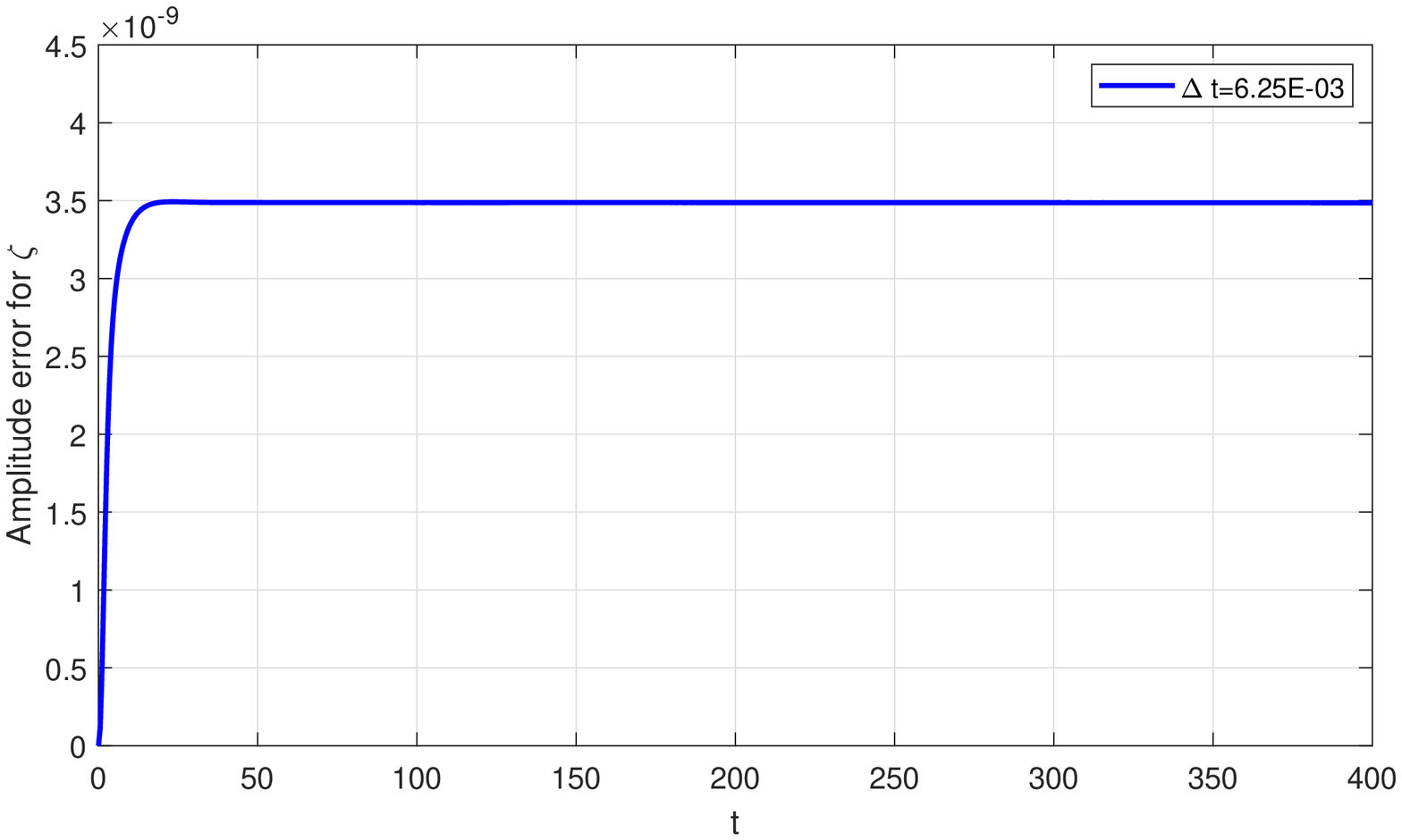}}
\subfigure[]
{\includegraphics[width=6.2cm]{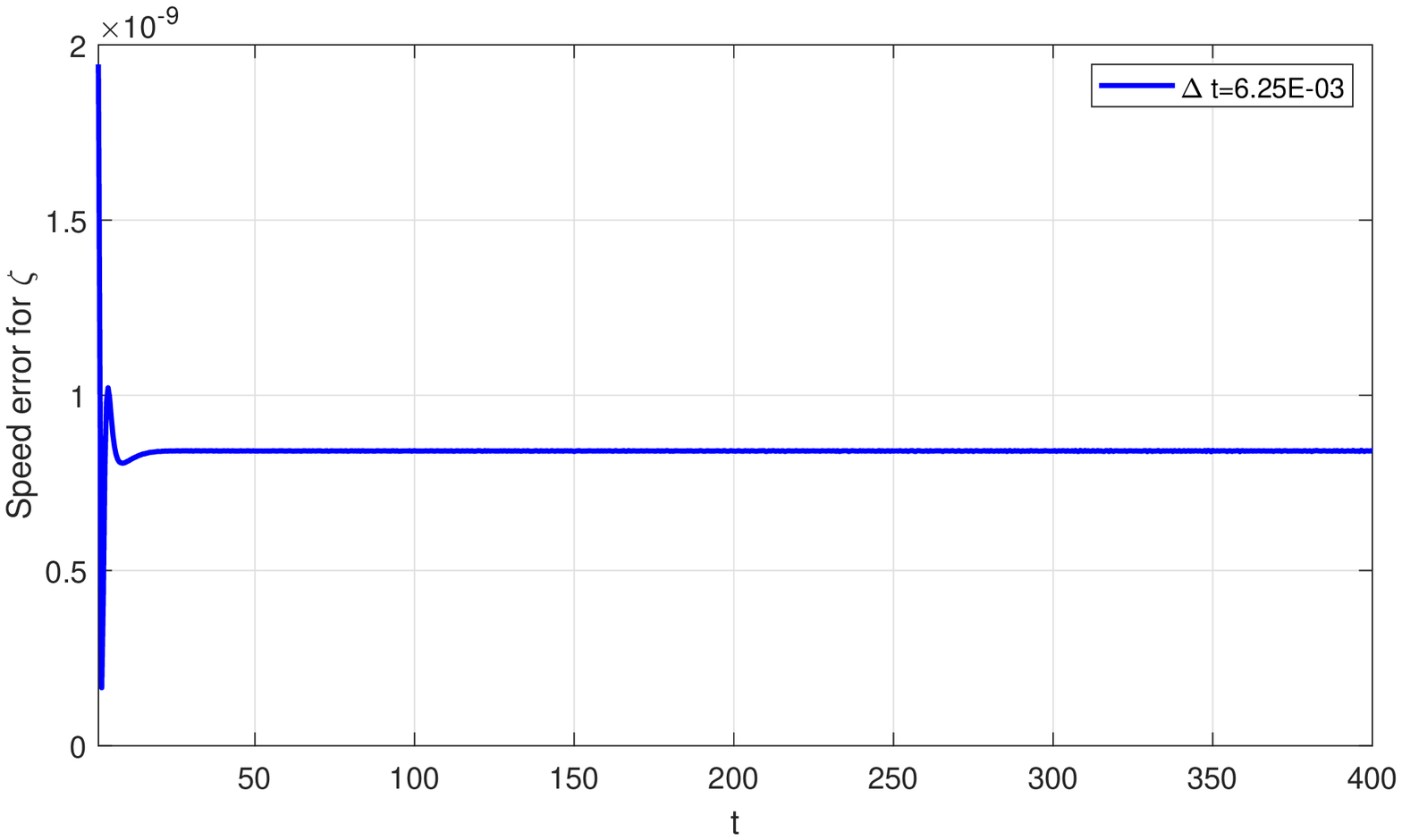}}
\caption{Evolution of the error between the (a) amplitude and (b) speed of the $\zeta$ component of the numerical solution and that of the initial solitary wave profile given in Figure  \ref{Fig_two}(a).} 
\label{Fig_seven}
\end{figure}

\begin{figure}[htbp]
\centering
\subfigure[]
{\includegraphics[width=6.3cm]{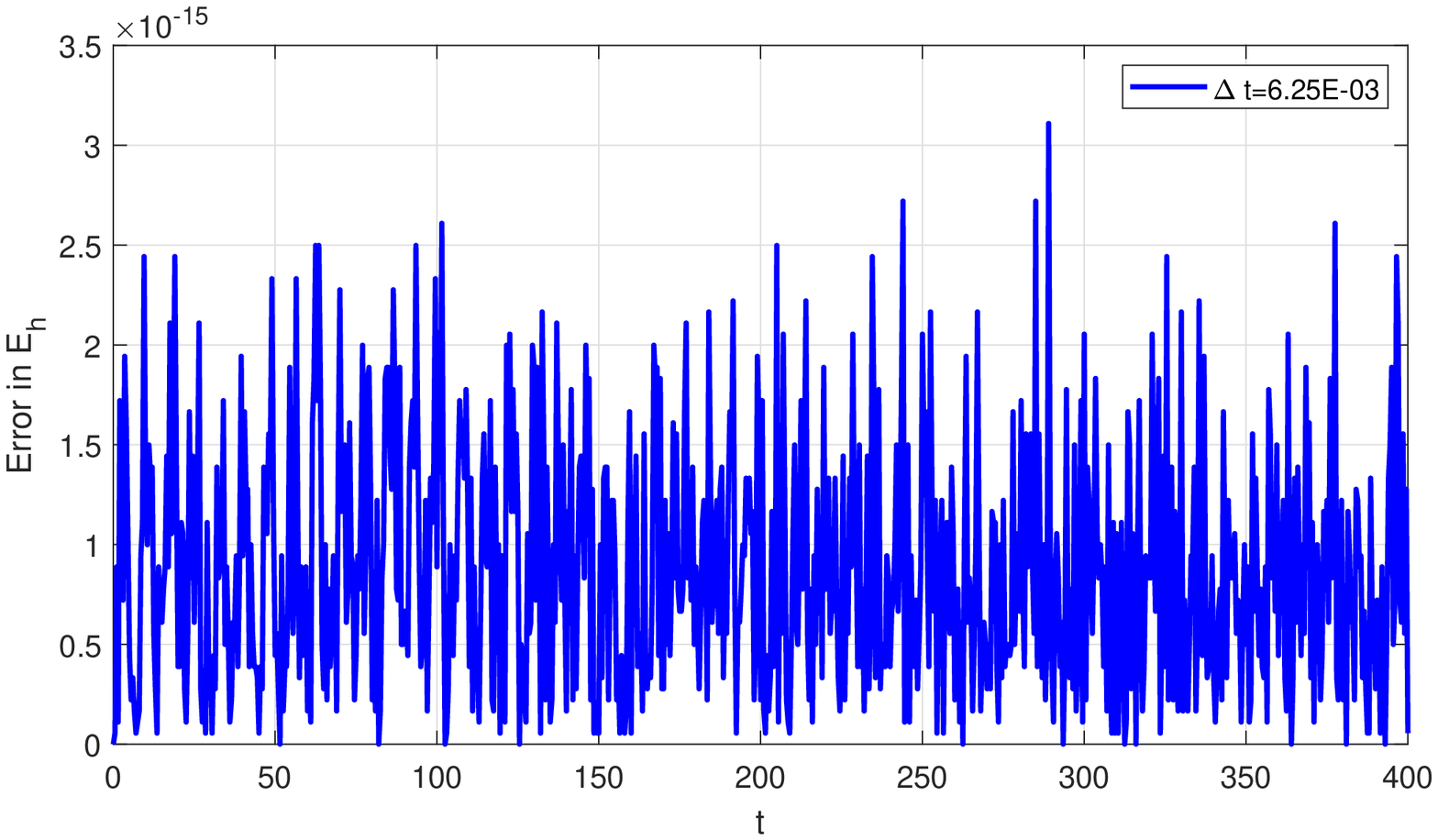}}
\subfigure[]
{\includegraphics[width=6.1cm]{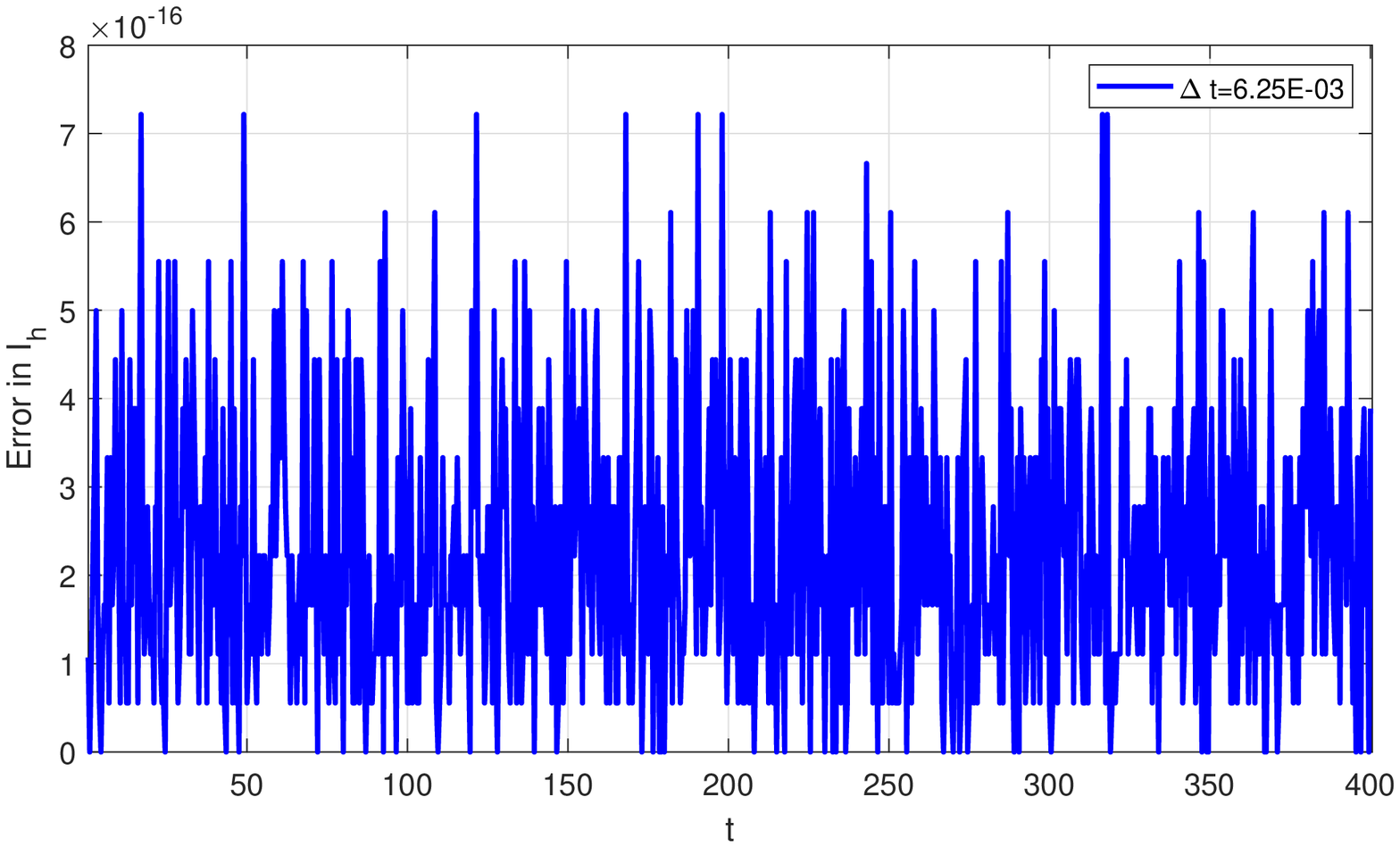}}
\caption{Evolution of the error in the invariants (a) (\ref{BFDdEnergy}), and (b) (\ref{BFDdMom}), for the initial solitary wave profile given in Figure  \ref{Fig_two}(a).}
\label{Fig_eight}
\end{figure}
\subsection{Small perturbations of solitary waves}
\label{sec42}
We start the computational study with the analysis of the behaviour of the solitary waves under small perturbations. Here the experiments consist of perturbing the amplitude of the computed solitary wave used in the previous section by some small quantity $A$, taking the perturbed wave as initial condition of the code and monitoring the evolution of the corresponding numerical solution and its parameters up to a final time. In the cases below, if $(\zeta_{0},u_{0})$ denotes the computed solitary wave profile, then the perturbed initial condition will be of the form $(A\zeta_{0},Au_{0})$. (Indeed perturbations of the form $(A\zeta_{0},u_{0})$ or $(\zeta_{0},Au_{0})$ can also be taken.)
We have considered the cases $A$ close to $1$ with $A<1$ and $A>1$. The conclusions from the experiments are similar in both cases and then only the results corresponding to $A=1.2$ will be shown. 
\begin{figure}[htbp]
\centering
\subfigure[]
{\includegraphics[width=\columnwidth]{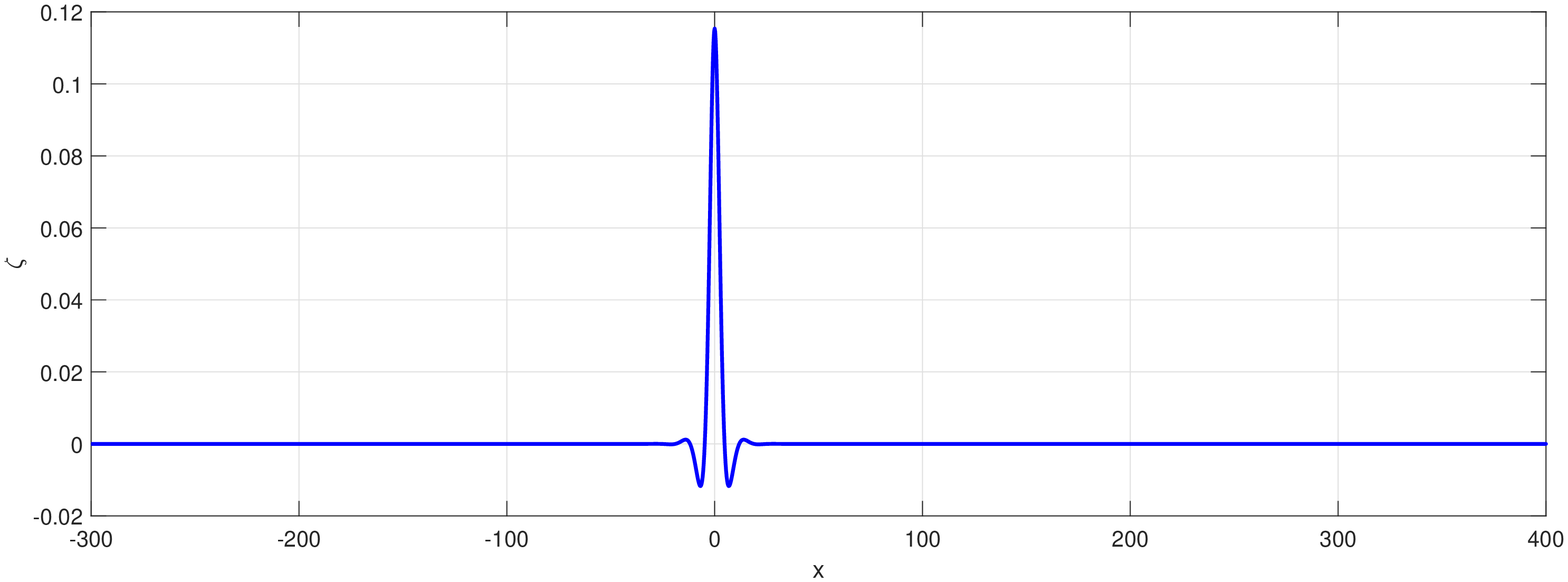}}
\subfigure[]
{\includegraphics[width=\columnwidth]{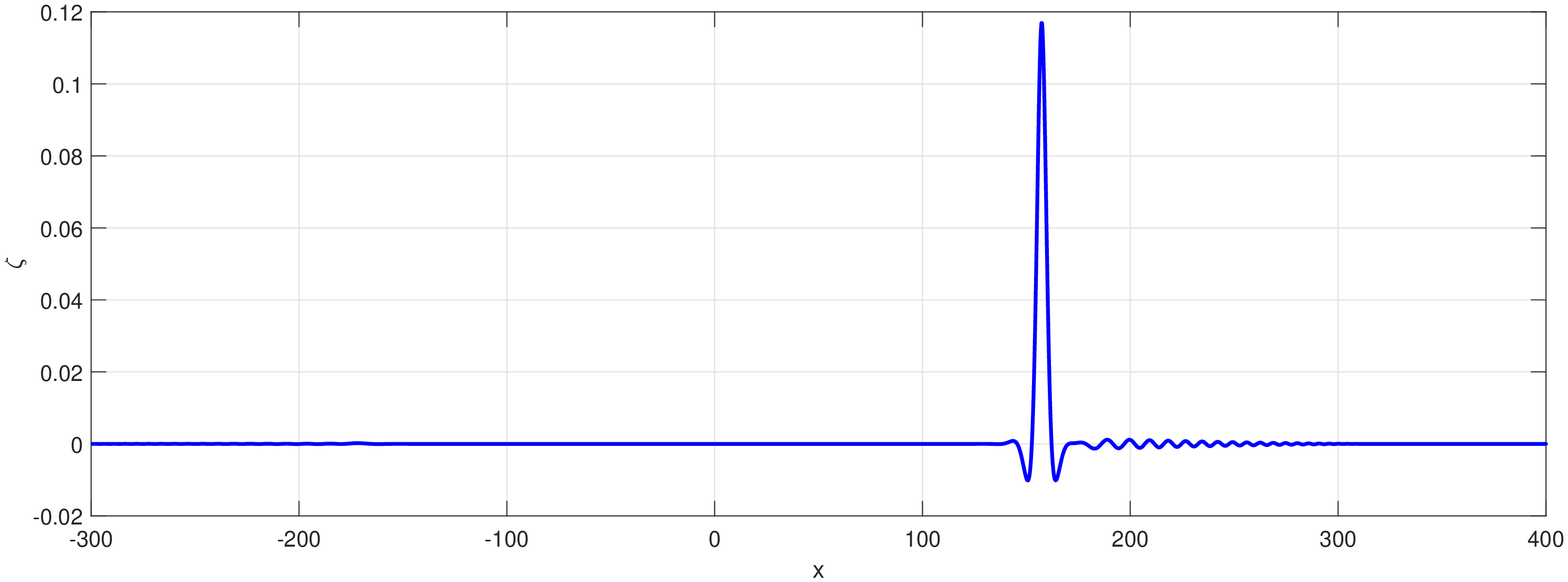}}
\subfigure[]
{\includegraphics[width=\columnwidth]{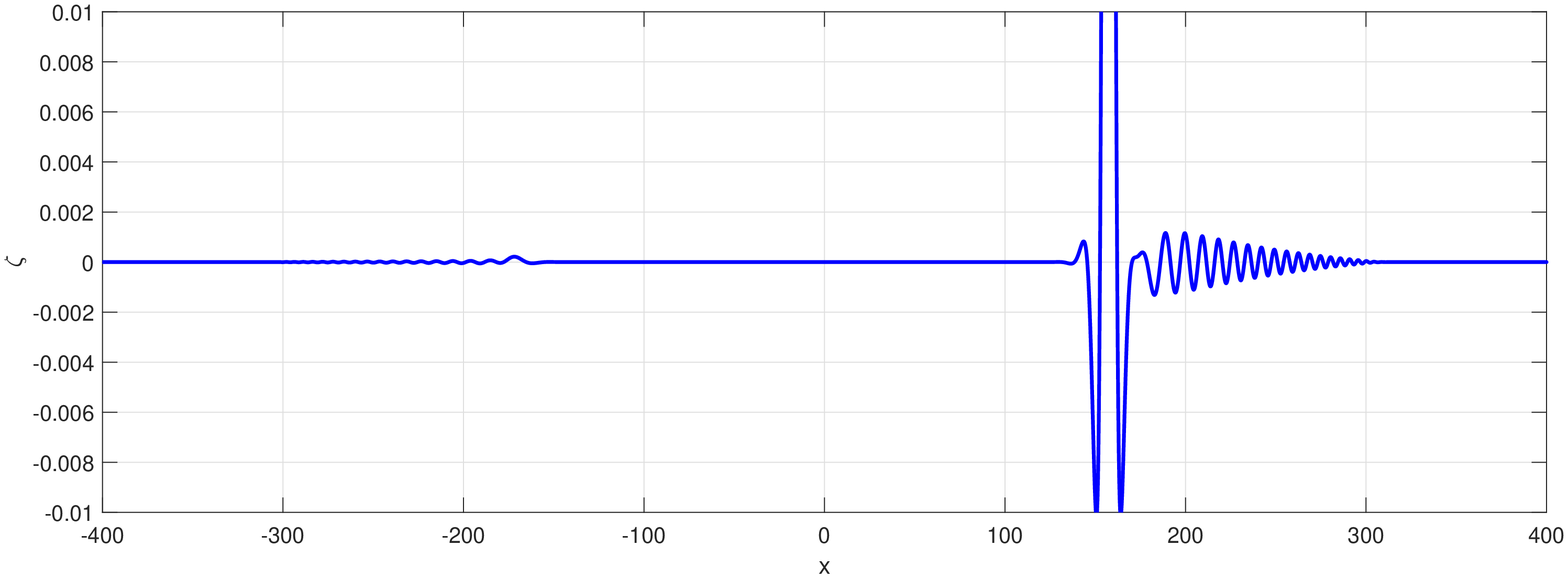}}
\caption{Evolution of the $\zeta$ component of the numerical solution from initial perturbed solitary wave profile with $A=1.2$. (a) $t=0$; (b) $t=400$; (c) Magnification of (b).}
\label{Fig_nine}
\end{figure}
\begin{figure}[htbp]
\centering
\subfigure[]
{\includegraphics[width=6.2cm]{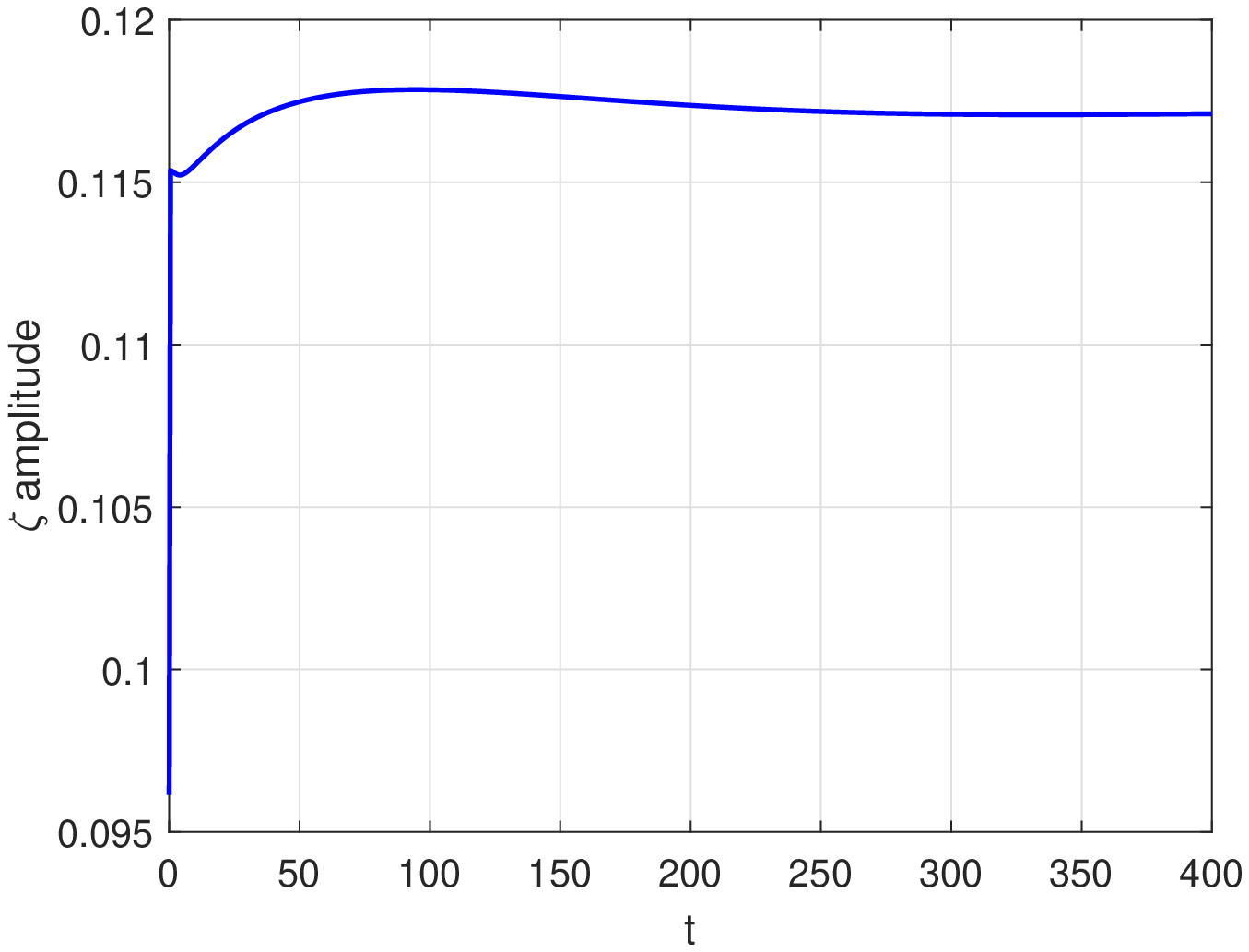}}
\subfigure[]
{\includegraphics[width=6.2cm]{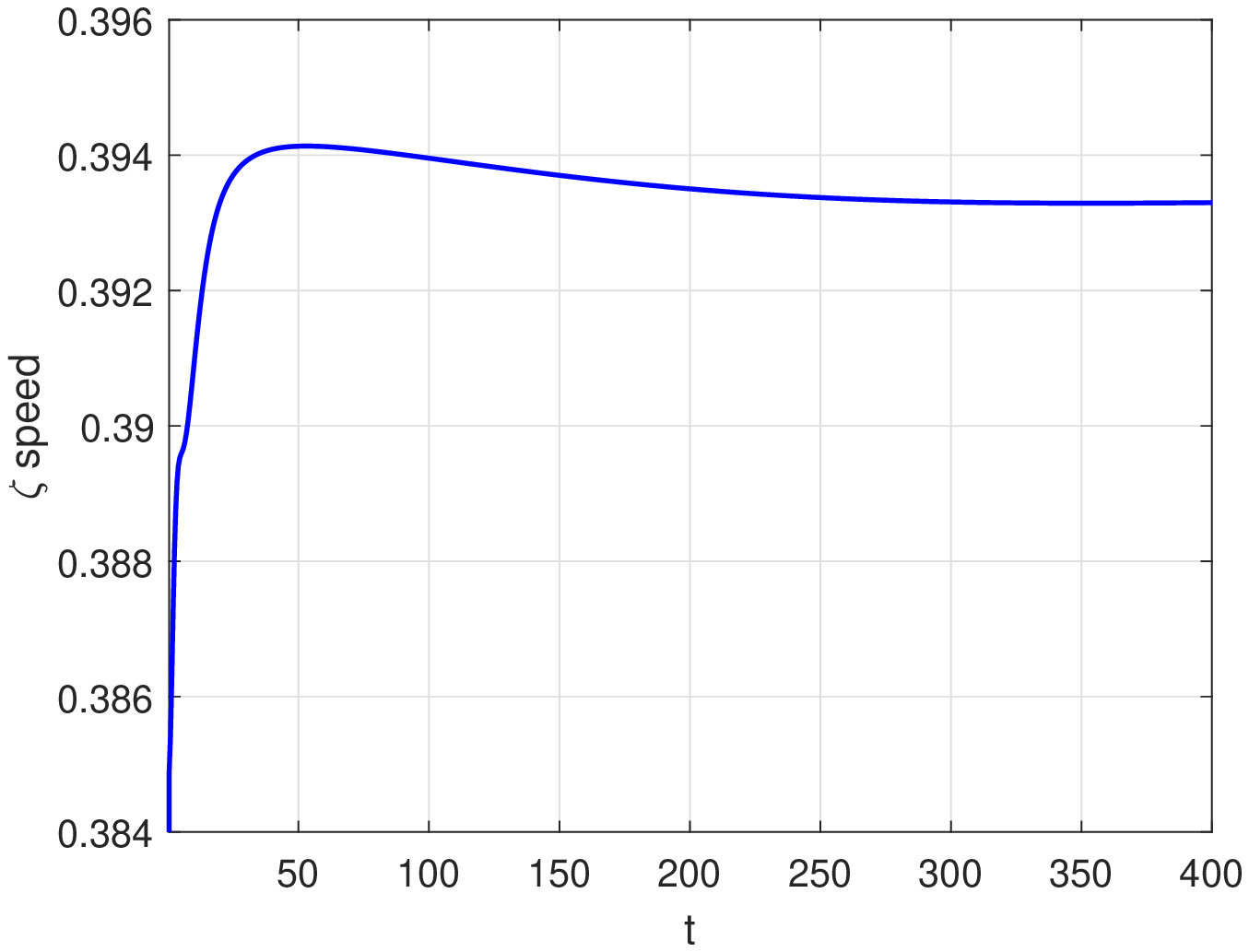}}

\caption{Amplitude and speed of the $\zeta$ component of the main pulse of the numerical solution from initial perturbed solitary wave profile with $A=1.2$. The initial amplitude and speed are $\max\zeta_{0}=1.38444e-01, c_{s}=4e-01$.}
\label{Fig_ten}
\end{figure}
\begin{figure}[htbp]
\centering
{\includegraphics[width=\columnwidth]{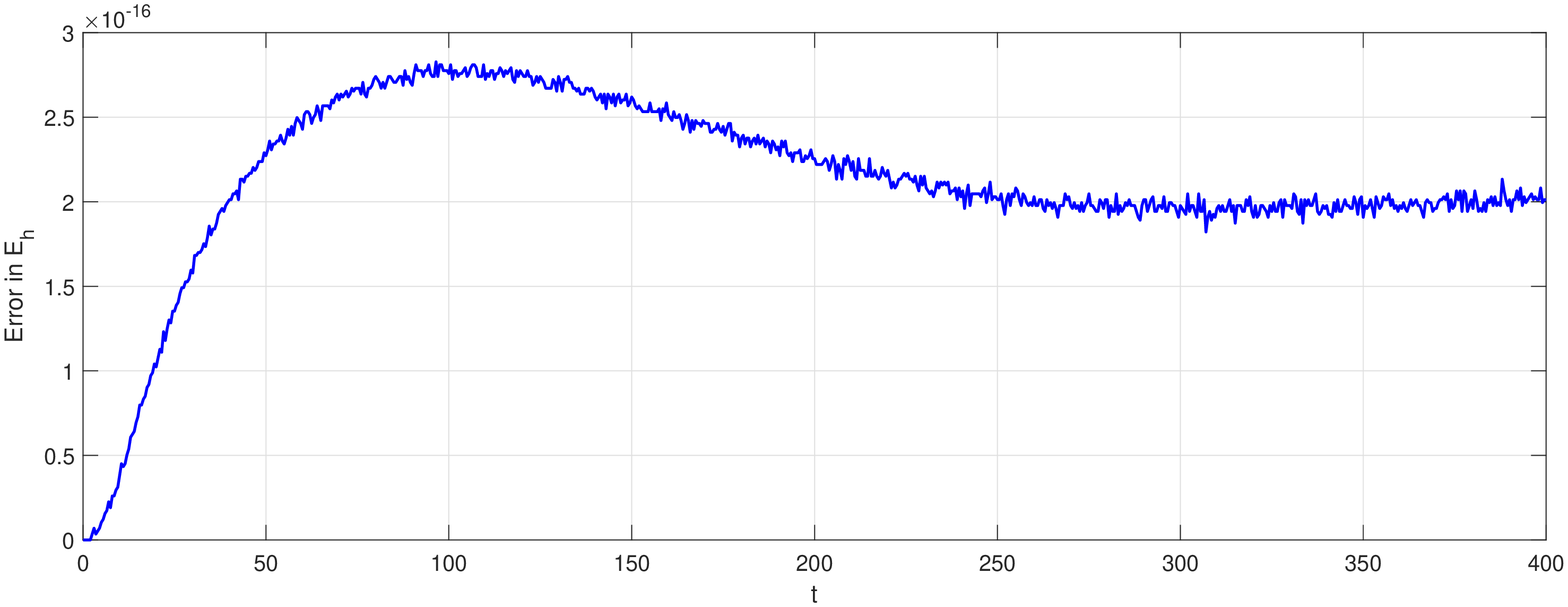}}
\caption{Evolution of the error in the energy (\ref{BFDdEnergy}) between the numerical solution and the initial perturbed solitary wave profile with $A=1.2$.}
\label{Fig_eleven}
\end{figure}
In order to illustrate the computational study in this case, we selected the experiments concerning the evolution of the $\zeta$ component of the numerical solution (Figure \ref{Fig_nine}), the evolution of amplitude and speed of the main pulse emerged in the numerical solution (Figure \ref{Fig_ten}) and the evolution of the error in the energy (\ref{BFDdEnergy}) (Figure \ref{Fig_eleven}). While the behaviour observed in this last figure ensures somehow the accuracy of the simulation, Figures \ref{Fig_nine} and \ref{Fig_ten} illustrate that the main effects of the small perturbations seem to be the generation of a modified solitary-wave pulse and the formation of dispersive tails in front of and behind it. The corresponding evolution of the speed $\widetilde{c_{s}}$ of the emerging wave suggest a small change from the original one $c_{s}$. It was observed that $\widetilde{c_{s}}>c_{s}$ when $A<1$ and $\widetilde{c_{s}}<c_{s}$ when $A>1$. (According to the speed-amplitude relation suggested in section \ref{sec3}, the corresponding comments on the relation between the amplitudes can be made.)

The generation of the dispersive tails can be studied, as usual, from the linearized system  of (\ref{BFD}) moving with a solitary wave with speed $c_{s}$
 \begin{eqnarray*}
J_{b}(\partial_{t}-c_{s}\partial_{y})\zeta+\mathcal{L}_{\mu_{2}}\partial_{y}v_{\beta}&=&0,\nonumber\\
J_{b}(\partial_{t}-c_{s}\partial_{y})v_{\beta}+(1-\gamma)J_{c}\partial_{y}\zeta&=&0,\label{BFD31}
\end{eqnarray*}
($y=x-c_{s} t$) and which can be simplified to
 \begin{eqnarray}
J_{b}^{2}(\partial_{t}-c_{s}\partial_{y})^{2}\zeta-(1-\gamma)\mathcal{L}_{\mu_{2}}J_{c}\partial_{y}^{2}\zeta=0.\label{BFD32}
\end{eqnarray}
Plane wave solutions $\zeta=e^{i(ky-\omega(k)t)}$ will satisfy
\begin{eqnarray*}
-(1+b\mu k^{2})^{2}(\omega(k)+kc_{s})^{2}+(1-\gamma)l_{\mu_{2}}(k)k^{2}(1-c\mu k^{2})=0,
\end{eqnarray*}
that is
\begin{eqnarray}
\omega_{\pm}(k)&=&-c_{s} k\pm k\varphi(k),\nonumber\\
\varphi(k)&=&\frac{\sqrt{(1-\gamma)l_{\mu_{2}}(k)(1-c\mu k^{2}})}{1+b\mu k^{2}},\label{BFD33}
\end{eqnarray}
with the corresponding phase speed
\begin{eqnarray*}
v_{\pm}(k):=\frac{\omega_{\pm}(k)}{k}=-c_{s} \pm \varphi(k),
\end{eqnarray*}
and group velocity
\begin{eqnarray*}
\omega_{\pm}^{\prime}(k)=-c_{s}\pm (k\varphi^{\prime}(k)+\varphi(k)).
\end{eqnarray*} 
The function $\varphi$ in (\ref{BFD33}) is analyzed in Appendix \ref{appendixA} as a key part for the proof of existence of solitary wave solutions of (\ref{BFD1}). Specifically, $\varphi(k)$ is shown to be positive for $k\in\mathbb{R}$ and for $\gamma\in (0,1)$ defines the speed limit
\begin{equation*}
c_{\gamma}=\sqrt{{\inf}_{k}\varphi(k)}>0,
\end{equation*}
for which solitary wave solutions of speed $c_{s}$ with $|c_{s}|<c_{\gamma}$ are proved to exist. This implies that $v_{\pm}(k)<0$ for all $k\in\mathbb{R}$.
Thus we would have the existence of components 
$e^{i(ky-\omega(k)t)}$ leading the pulse (with positive phase speed in the coordinate system moving with the main wave) and components trailing the main pulse. This justifies the presence of the two groups of dispersive tails observed in the experiments.
\subsection{Large perturbations of solitary waves and resolution property}
As the perturbation factor $A$ grows, some other phenomena are observed in the experiments. The most relevant one is the resolution property: for $A$ large enough, the initial perturbed solitary wave profile evolves generating a main solitary wave, additional solitary waves behind and in front of it, and dispersive tails. A good example of this is shown in Figure \ref{Fig_twelve}, corresponding to the evolution of the initial solitary wave used in Figure \ref{Fig_nine} and perturbed with a factor $A=6$. In this case, a main solitary wave, traveling to the right, emerges (the evolution of the amplitude and speed is shown in Figure \ref{Fig_13}), and, at the final time of computation, the formation of two additional solitary waves are observed, one traveling to the right and one to the left. According to the computational study performed in section \ref{sec3}, in all the cases the magnitude of their speeds is larger than that of the main wave (the amplitude decreases and the nonmonotone behaviour of the decay increases). Some groups of waves are generated in front of these additional solitary waves. They mainly seem to have a dispersive nature although, due to the form of the solitary waves observed in some experiments in section \ref{sec3}, the generation of some new nonlinear structures is not discarded.
\begin{figure}[htbp]
\centering
\subfigure[]
{\includegraphics[width=\columnwidth]{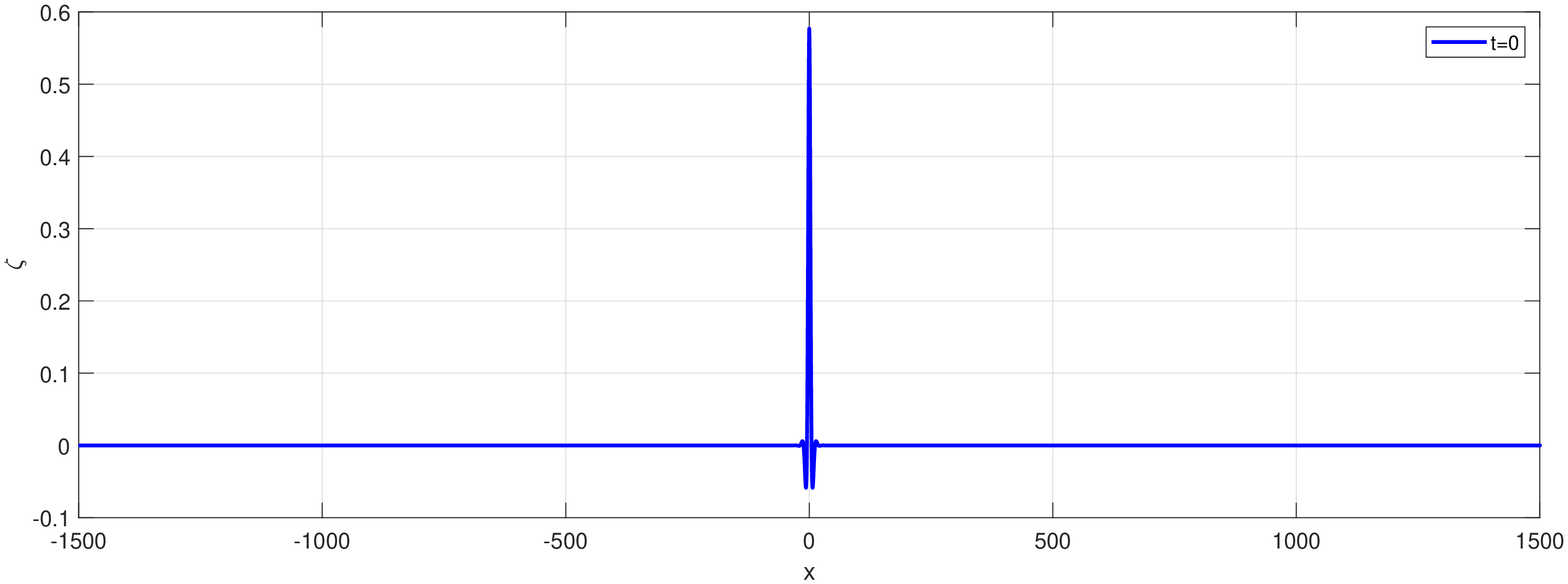}}
\subfigure[]
{\includegraphics[width=\columnwidth]{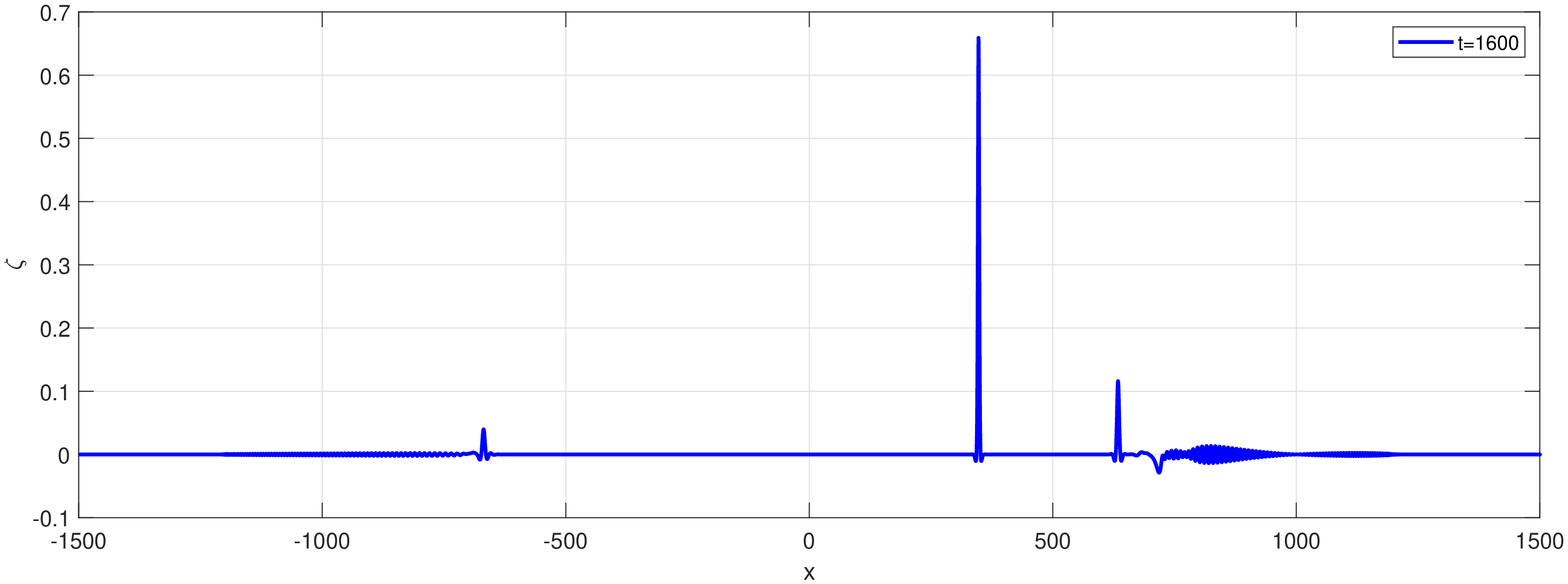}}
\subfigure[]
{\includegraphics[width=\columnwidth]{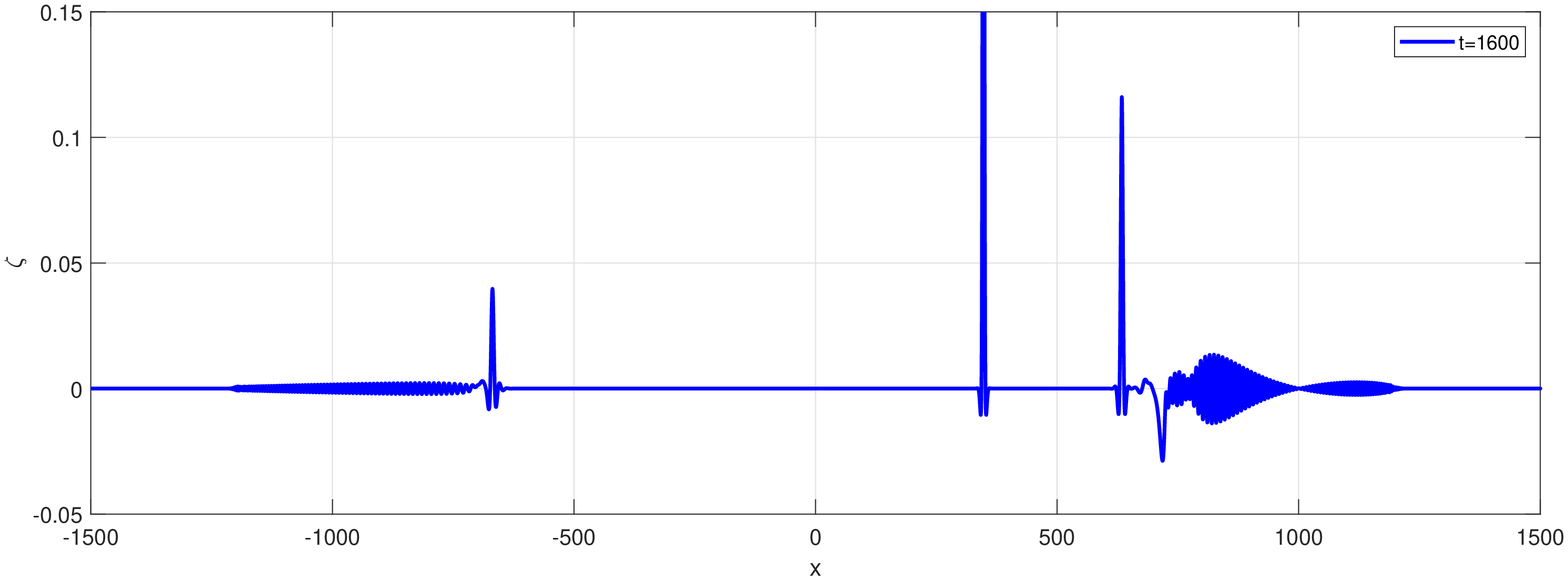}}
\caption{Evolution of the $\zeta$ component of the numerical solution from initial perturbed solitary wave profile with $A=6$. (a) $t=0$; (b) $t=1600$; (c) Magnification of (b).}
\label{Fig_twelve}
\end{figure}
\begin{figure}[htbp]
\centering
\subfigure[]
{\includegraphics[width=6.2cm]{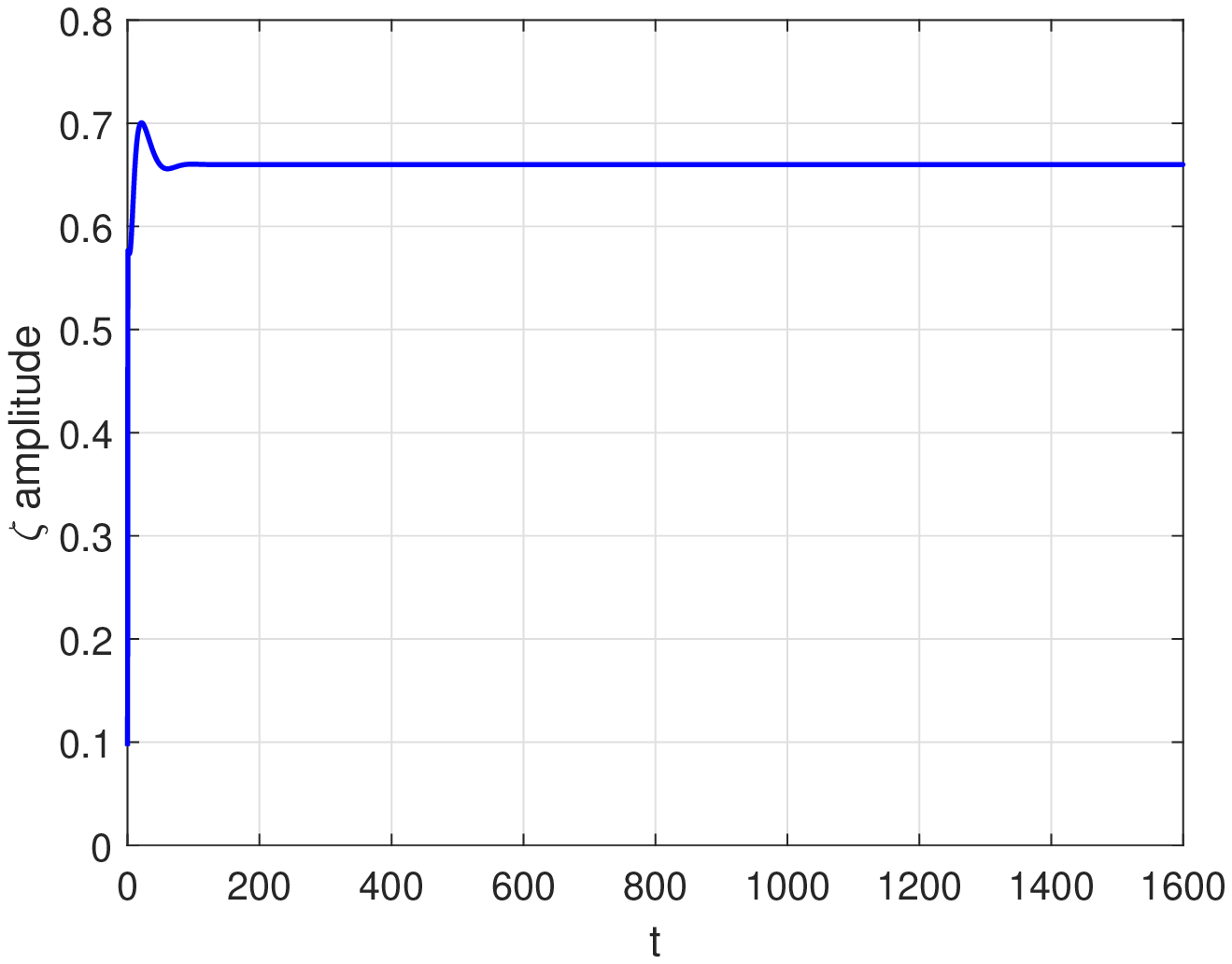}}
\subfigure[]
{\includegraphics[width=6.2cm]{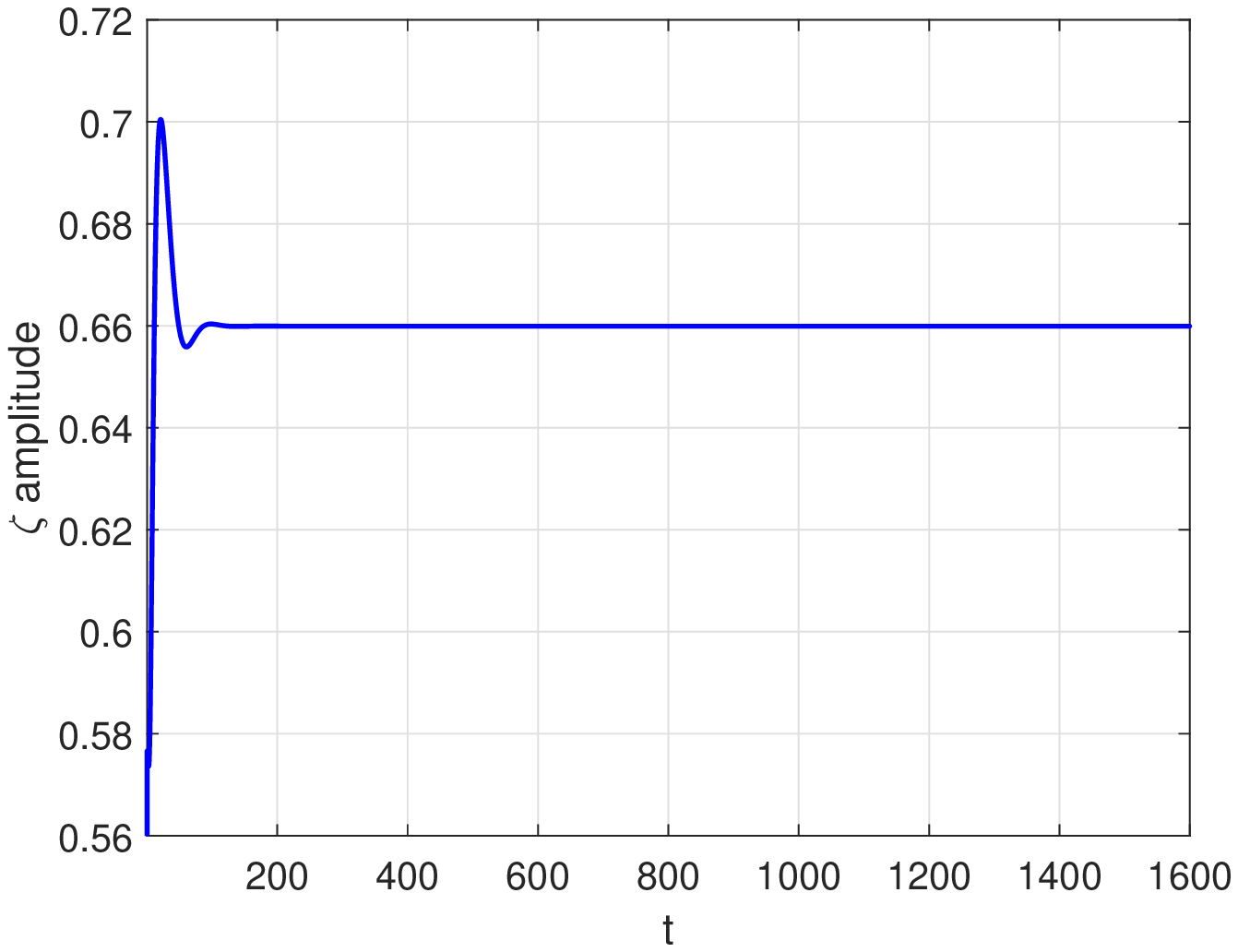}}
\caption{Amplitude and speed of the $\zeta$ component of the main pulse of the numerical solution from initial perturbed solitary wave profile with $A=6$.}
\label{Fig_13}
\end{figure}
This version of the resolution property is also observed from other initial conditions such as those of Gaussian type. This is illustrated in Figures \ref{Fig_14} and \ref{Fig_15}, which shows the evolution of an initial condition of the form $\zeta(x,0)=u(x,0)=Ae^{-\tau x^{2}}$ with $A=0.8$ and $\tau=0.01$. The initial Gaussian profile evolves in a similar way to that of Figure \ref{Fig_twelve}, breaking into a main solitary wave (cf. Figure \ref{Fig_16} for the evolution of its amplitude and speed) and additional solitary waves traveling to the right and to the left, of smaller amplitude (and then faster than the main one, cf. section \ref{sec3}) and tails which seem to be of dispersive nature, but that may evolve generating new solitary waves, see Figure \ref{Fig_15}. 
\begin{figure}[htbp]
\centering
\subfigure[]
{\includegraphics[width=\columnwidth]{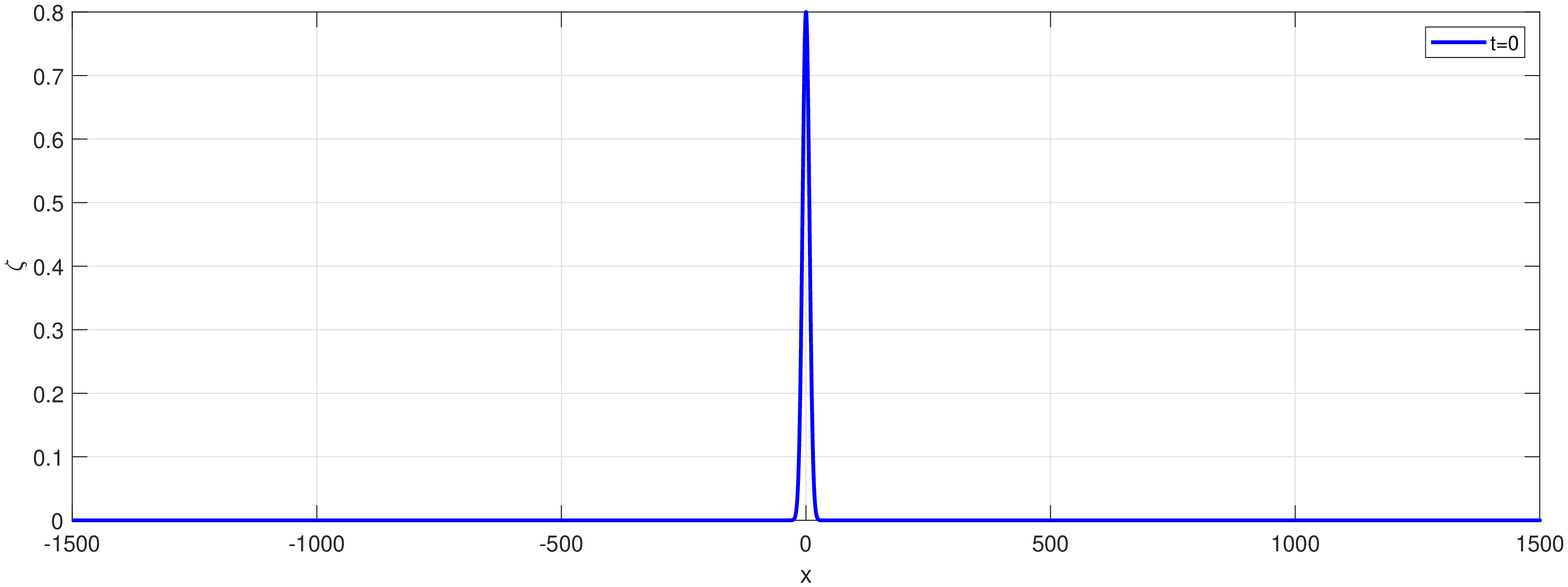}}
\subfigure[]
{\includegraphics[width=\columnwidth]{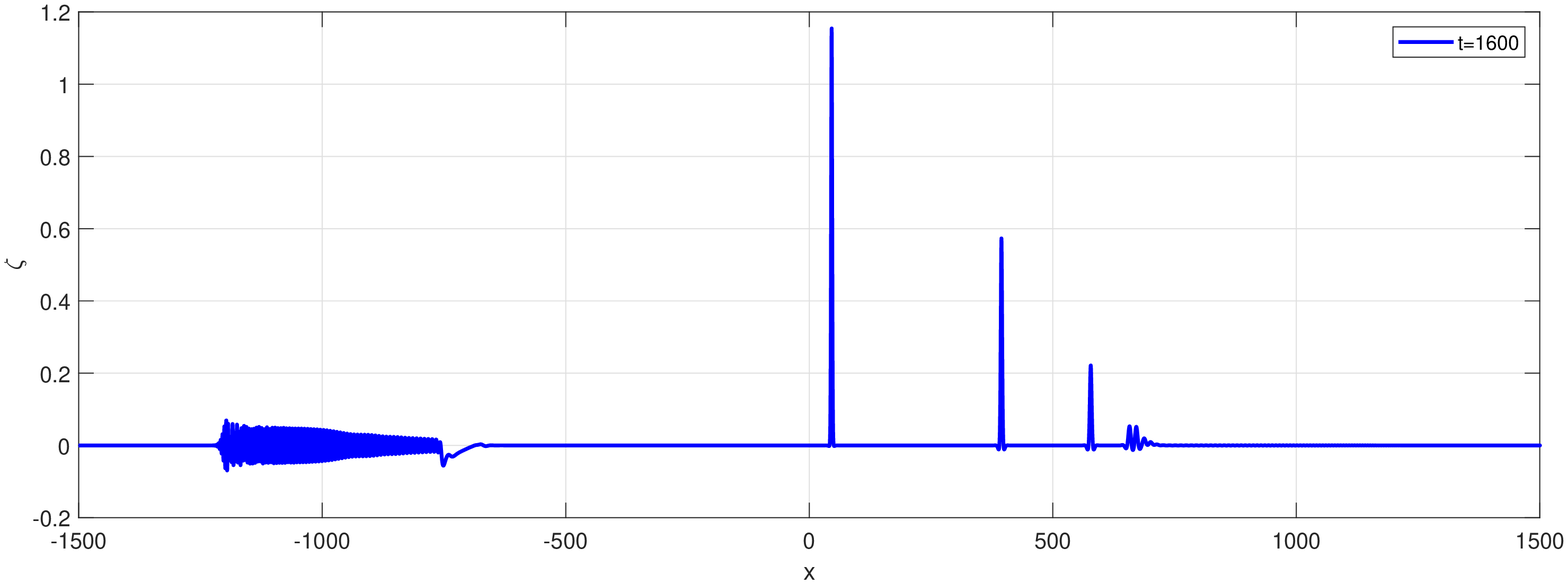}}
\subfigure[]
{\includegraphics[width=\columnwidth]{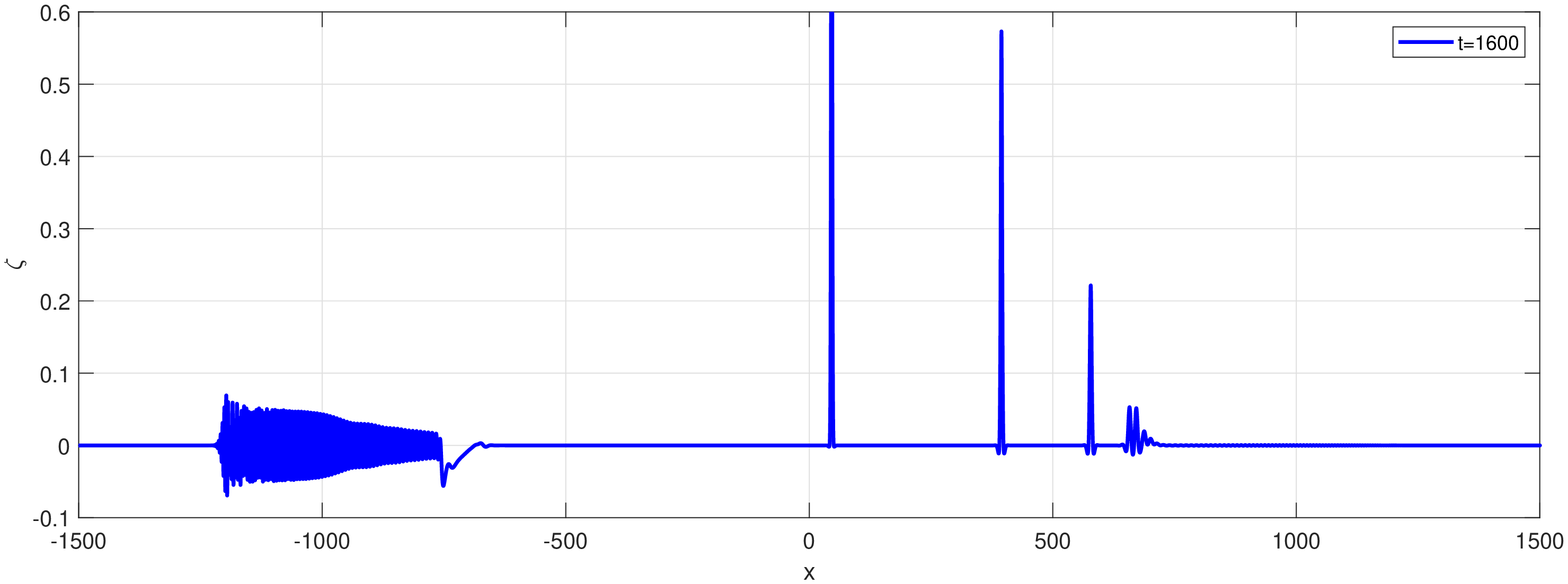}}
\caption{Evolution of the $\zeta$ component of the numerical solution from initial perturbed solitary wave profile with $A=0.8, \tau=0.01$. (a) $t=0$; (b) $t=1600$; (c) Magnification of (b).}
\label{Fig_14}
\end{figure}
\begin{figure}[htbp]
\centering
\subfigure[]
{\includegraphics[width=\columnwidth]{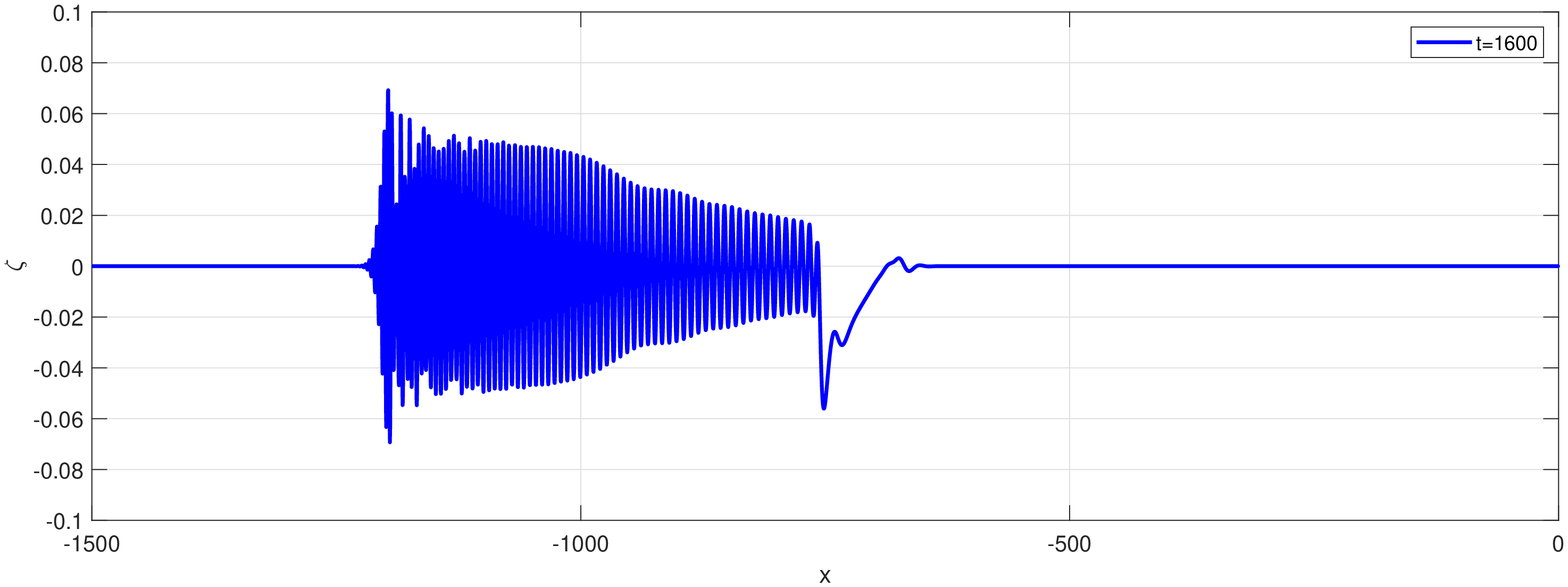}}
\subfigure[]
{\includegraphics[width=\columnwidth]{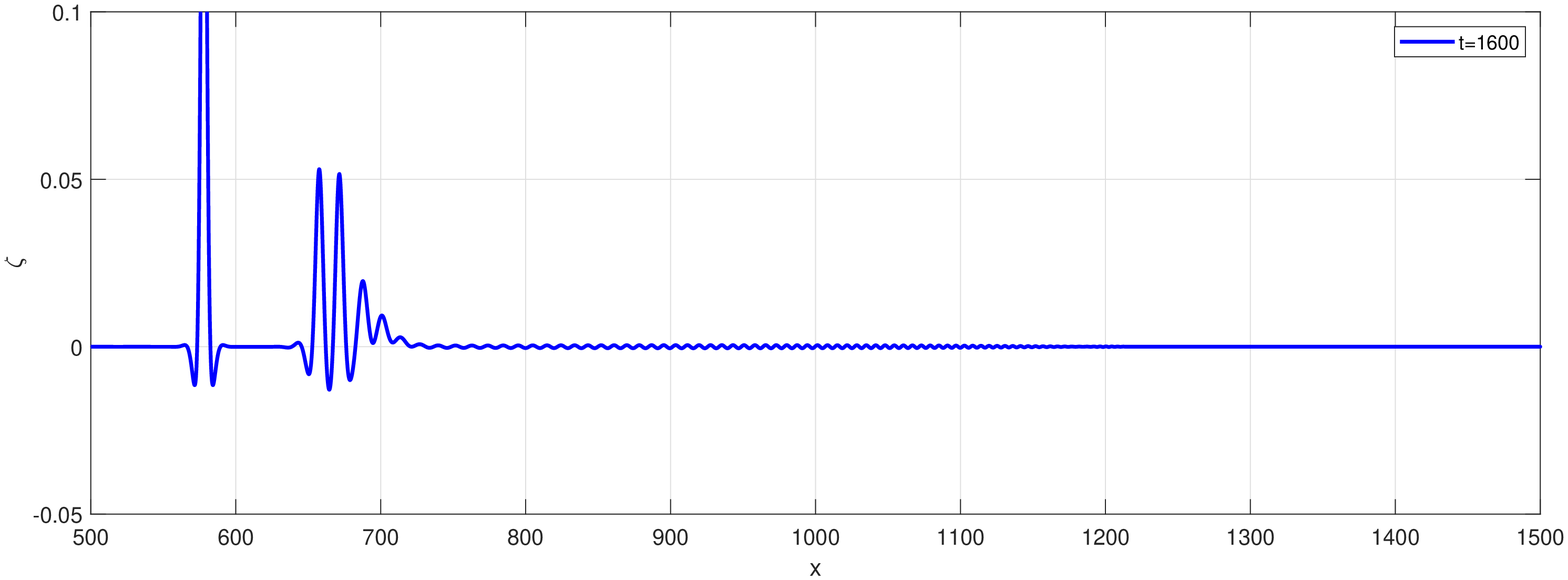}}
\caption{Evolution of the $\zeta$ component of the numerical solution from initial perturbed solitary wave profile with $A=0.8, \tau=0.01$. Magnifications of Figure \ref{Fig_14}(c).}
\label{Fig_15}
\end{figure}
\begin{figure}[htbp]
\centering
\subfigure[]
{\includegraphics[width=6.2cm]{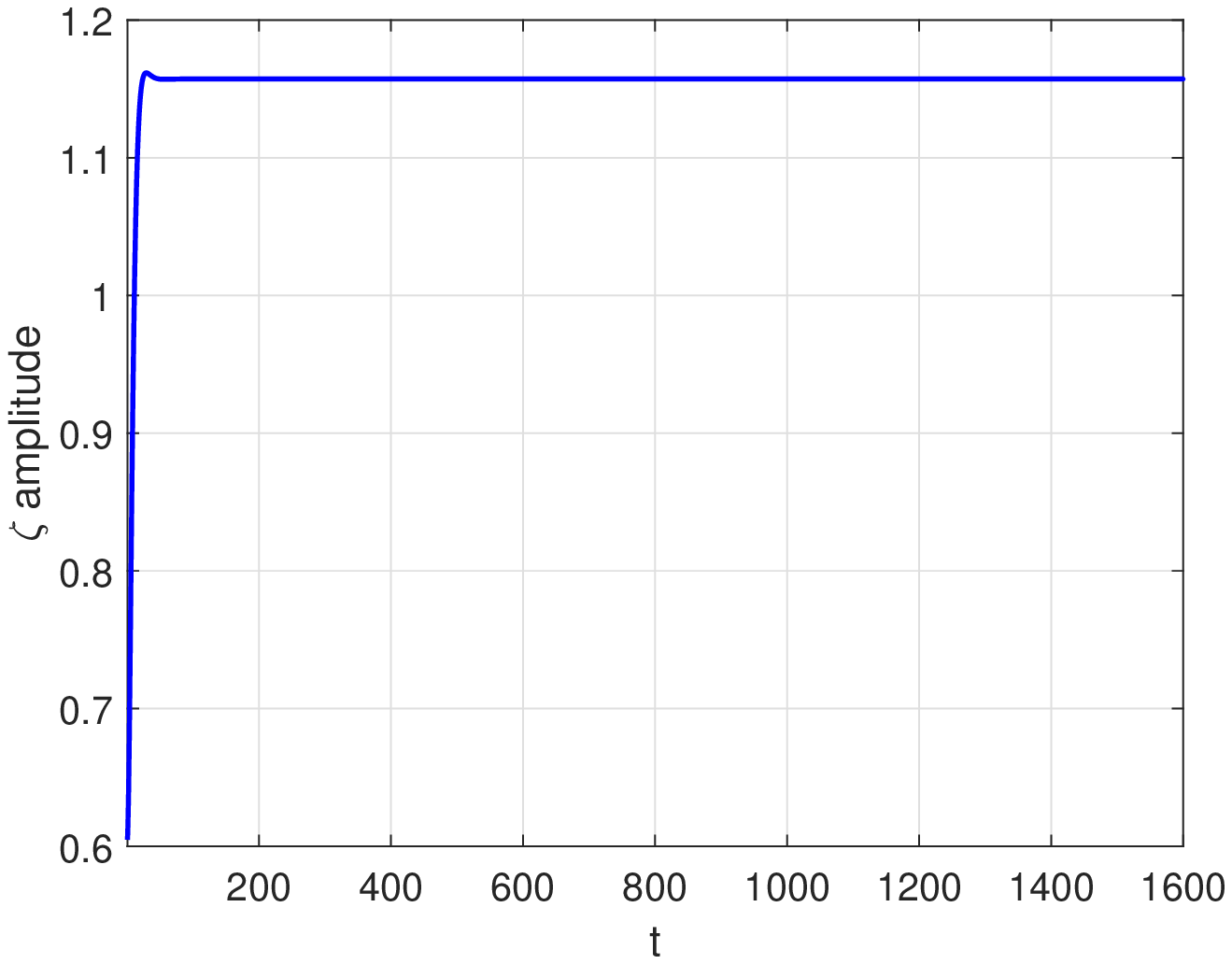}}
\subfigure[]
{\includegraphics[width=6.2cm]{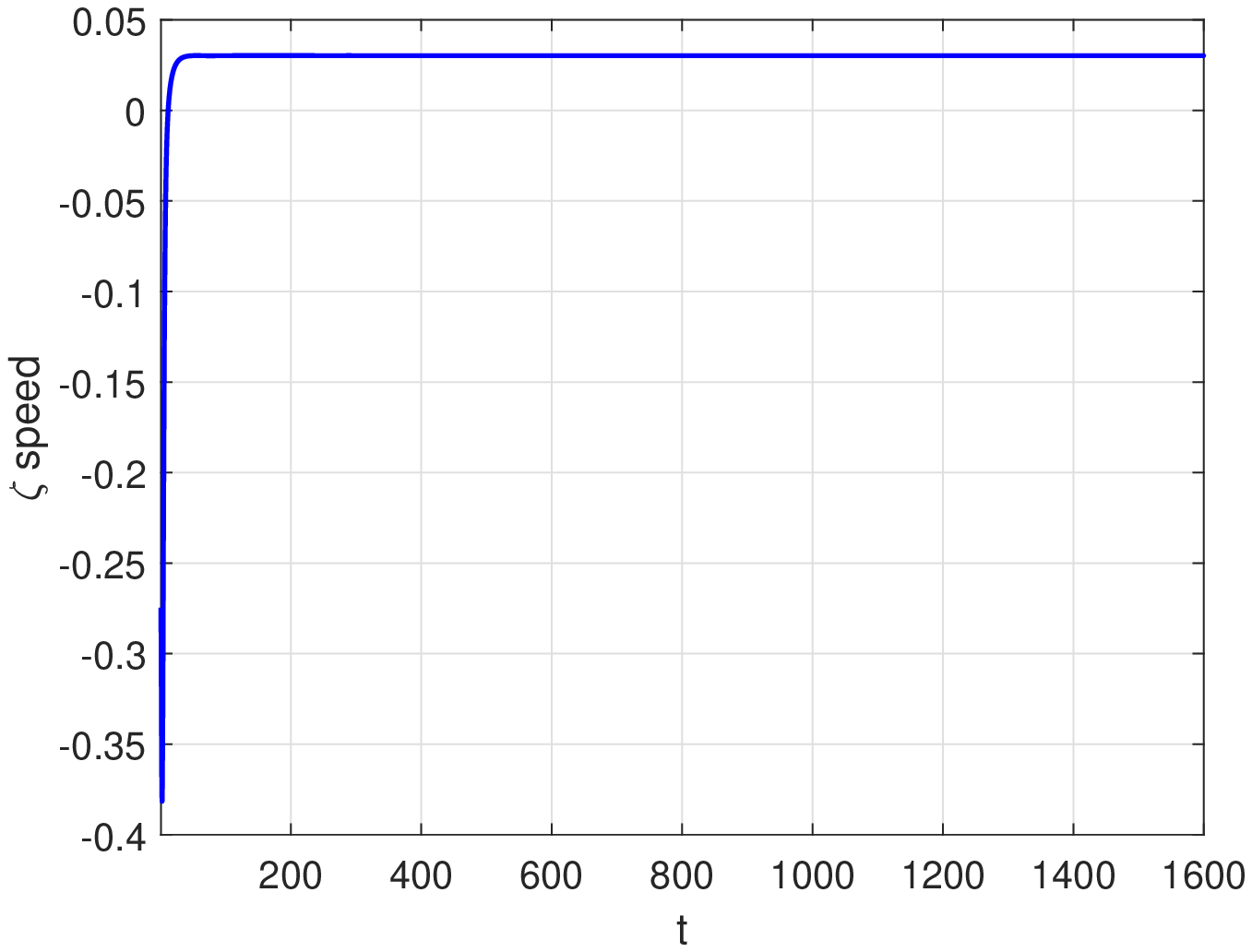}}
\caption{Amplitude and speed of the $\zeta$ component of the main pulse of the numerical solution of Figure \ref{Fig_14}.}
\label{Fig_16}
\end{figure}
It seems that this type of resolution property occurs from small amplitude Gaussian pulses. It was not observed when $A$ is large enough.
\subsection{Solitary wave interactions}
An additional point on the dynamics of (\ref{BFD1}), illustrated in this computational study, concerns the interaction of solitary wave solutions. The bi-directional character of the model invites to analyze both head-on and overtaking collisions, depicted in Figures \ref{Fig_17} and \ref{Fig_18} (in the first case) and in Figures \ref{Fig_19} and \ref{Fig_20} (in the second case). Figure \ref{Fig_17} shows the evolution of an initial data consisting of two solitary wave profiles generated with speeds $c_{s}^{(1)}=0.1$ (traveling to the right) and $c_{s}^{(2)}=0.2$ (traveling to the left), and centered at $x_{0}^{(1)}=-20$ and $x_{0}^{(2)}=20$ respectively. By the time $t=200$, the profiles interact inelastically, generating two modified solitary waves with different speeds and amplitudes to the corresponding partners before the collision. For the experiment at hand, both waves reduces their amplitude after the interaction; from $0.982658$ to $0.971565$ in the case of the taller wave and from $0.709789$ to $0.701182$ for the second wave. According to the speed-amplitude relation shown in section \ref{sec3}, the corresponding speeds increase slightly. This is illustrated in Figure \ref{Fig_18} for the taller (slower) wave.
\begin{figure}[htbp]
\centering
\subfigure[]
{\includegraphics[width=\columnwidth]{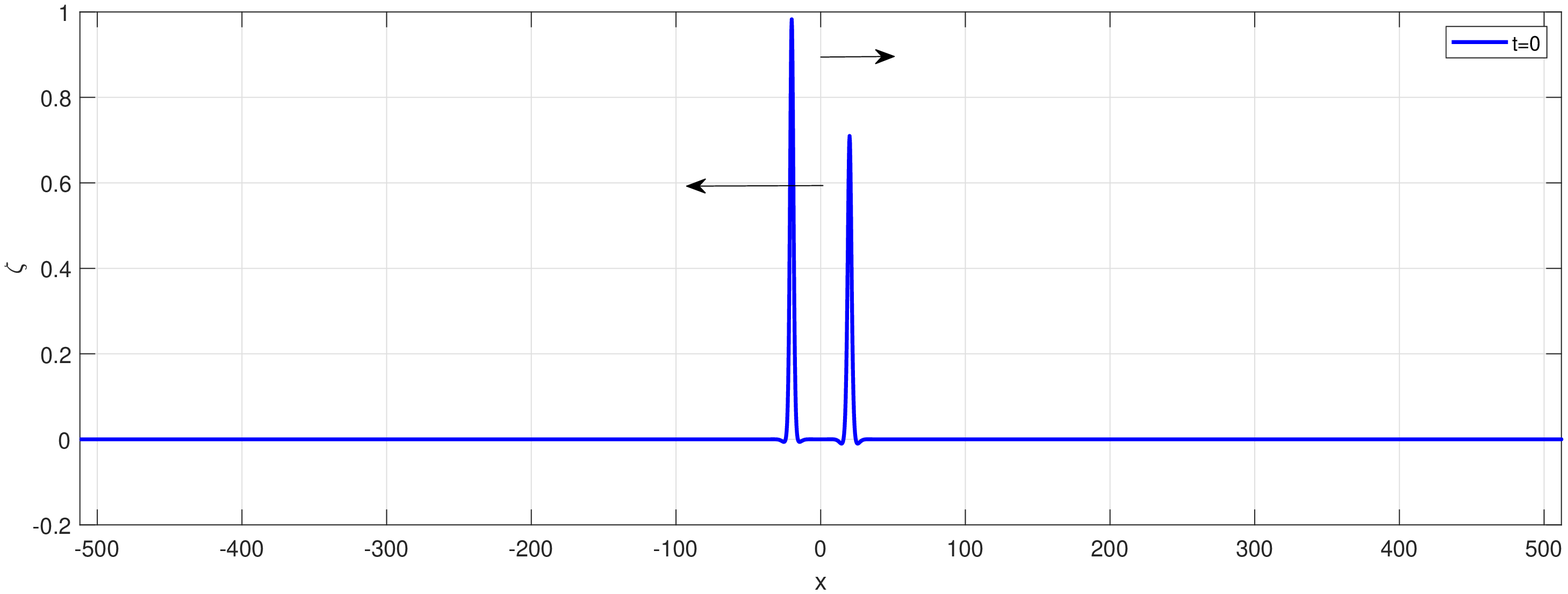}}
\subfigure[]
{\includegraphics[width=\columnwidth]{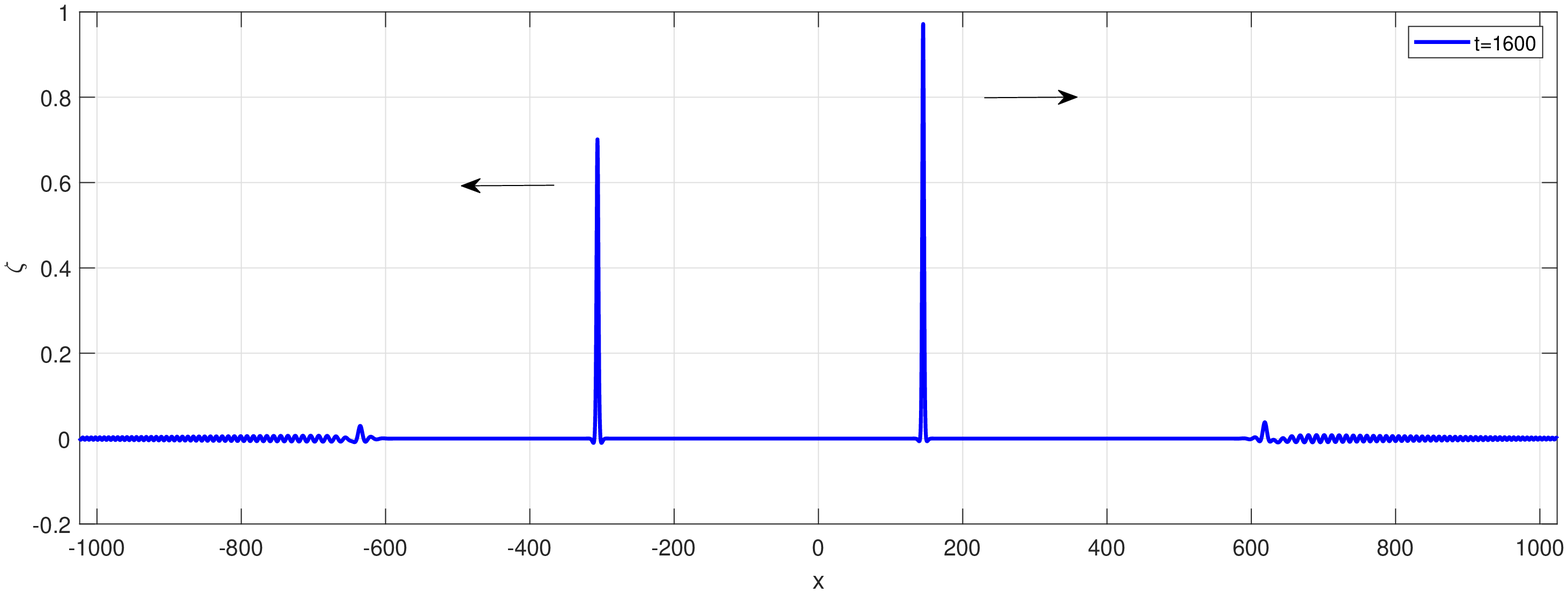}}
\subfigure[]
{\includegraphics[width=\columnwidth]{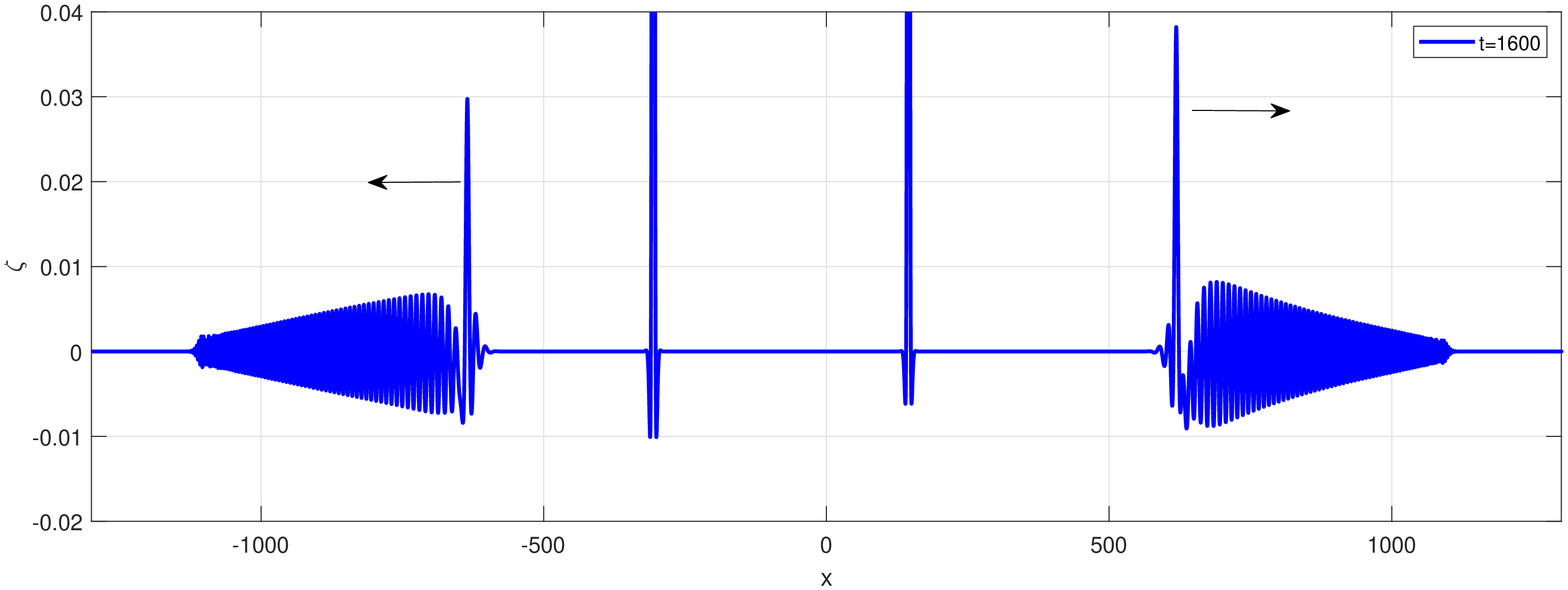}}
\caption{Head-on collision: evolution of the $\zeta$ component of the numerical solution from the superposition of two solitary wave profiles with speeds $c_{s}^{(1)}=0.1, c_{s}^{(2)}=0.2$. (a) $t=0$; (b) $t=1600$; (c) Magnification of (b).}
\label{Fig_17}
\end{figure}

A second consequence of the inelastic interaction is the formation of additional waves in front of each of these emerging waves. Both seem to develop, as time goes by, a similar behaviour, with the generation of a wave of solitary type (faster, with smaller amplitude and more nonmonotone decay, cf. section \ref{sec3}), and some dispersive tail.
\begin{figure}[htbp]
\centering
\subfigure[]
{\includegraphics[width=6.2cm]{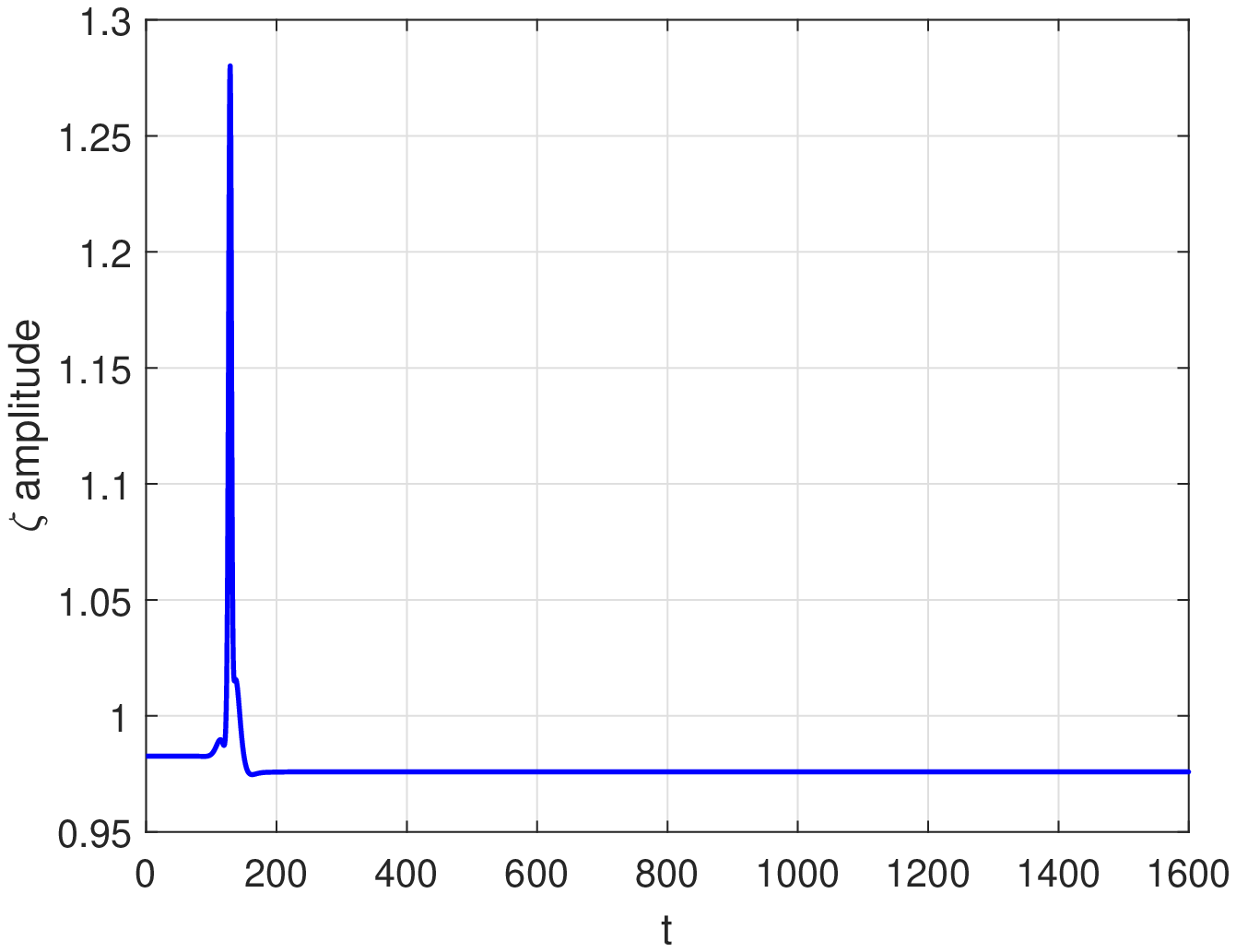}}
\subfigure[]
{\includegraphics[width=6.2cm]{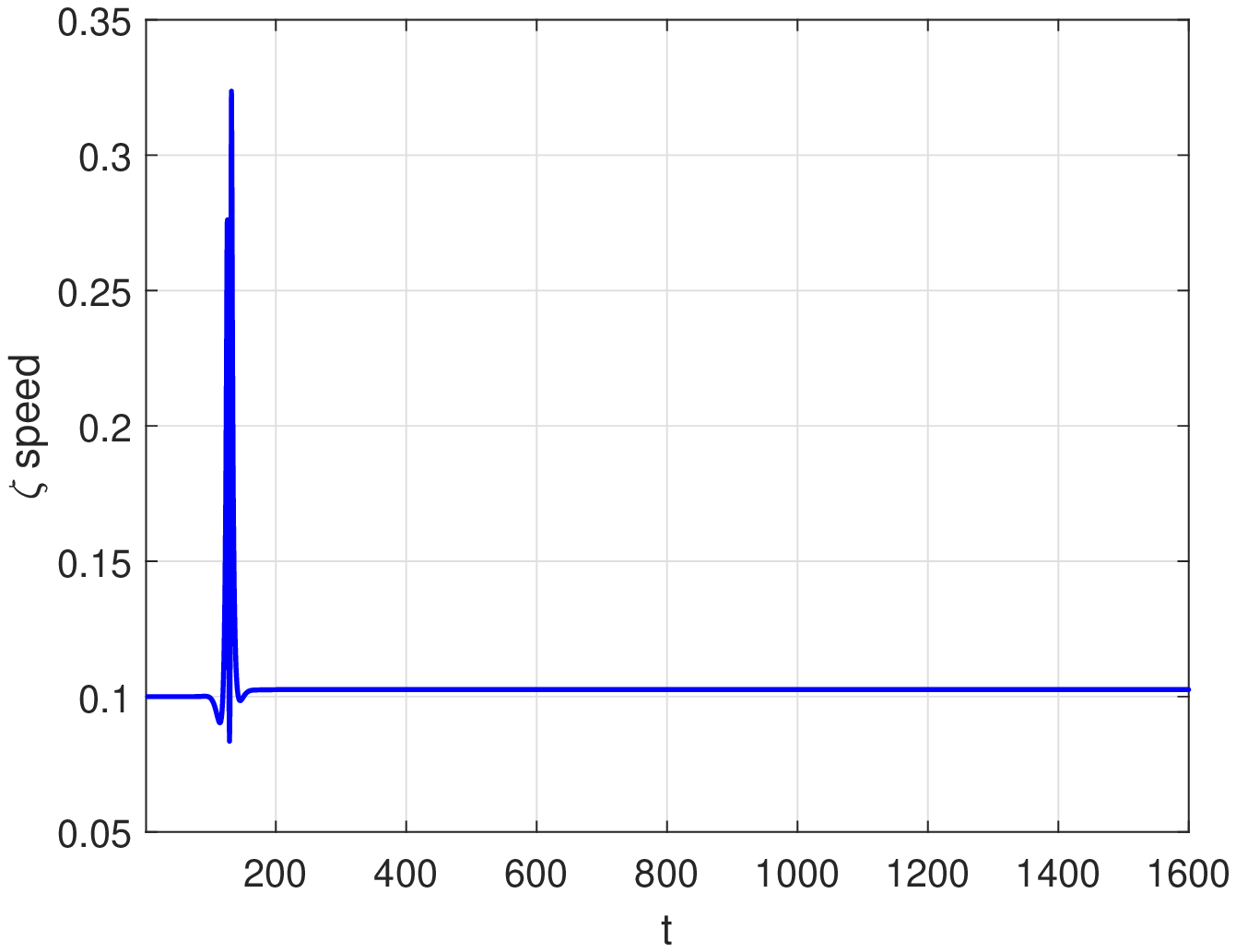}}
\caption{Amplitude and speed of the $\zeta$ component of the taller pulse of the numerical solution of Figure \ref{Fig_17}.}
\label{Fig_18}
\end{figure}

A similar pattern is observed in the case of overtaking collisions. The dynamics is represented in Figure \ref{Fig_19}. In this case, the initial condition is a superposition of the waves of speeds  $c_{s}^{(1)}=0.1, c_{s}^{(2)}=0.2$, but both traveling to the right and centered at $x_{0}^{(1)}=20$ and $x_{0}^{(2)}=-20$ respectively. By $t=200$, the second wave overtakes the first one. As a consequence two modified solitary waves emerge. For this experiment, the taller wave increases its amplitude, from  $0.982658$ to $0.983337$ (and therefore the emerging wave is slightly slower, see Figure \ref{Fig_20}), while the shorter wave reduces its amplitude after the collision from $0.709789$ to $0.707239$ (and then it goes slightly faster).
\begin{figure}[htbp]
\centering
\subfigure[]
{\includegraphics[width=\columnwidth]{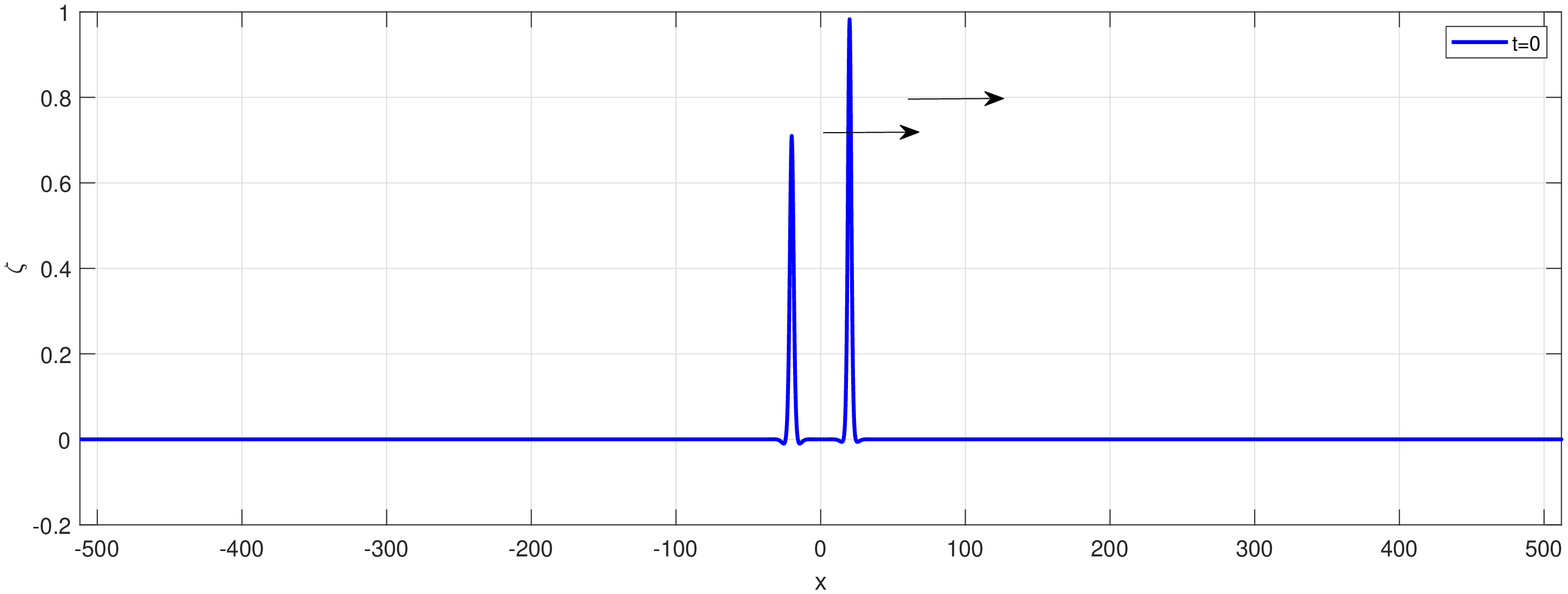}}
\subfigure[]
{\includegraphics[width=\columnwidth]{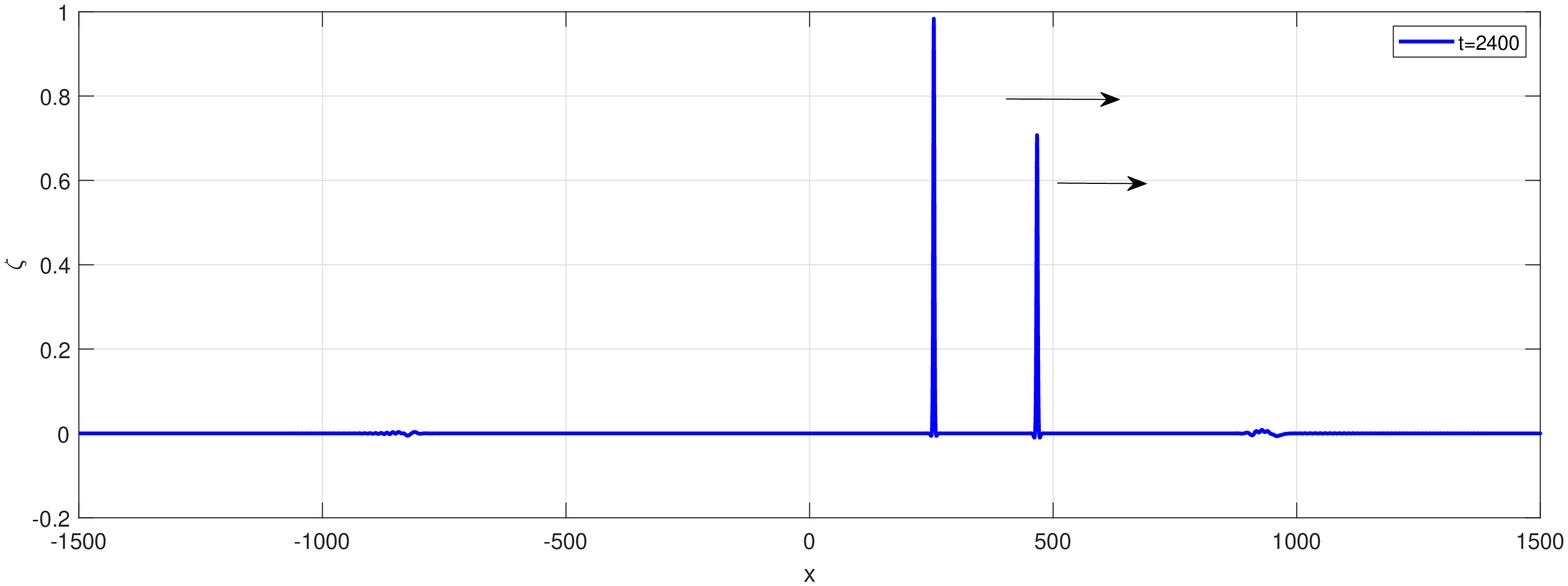}}
\subfigure[]
{\includegraphics[width=\columnwidth]{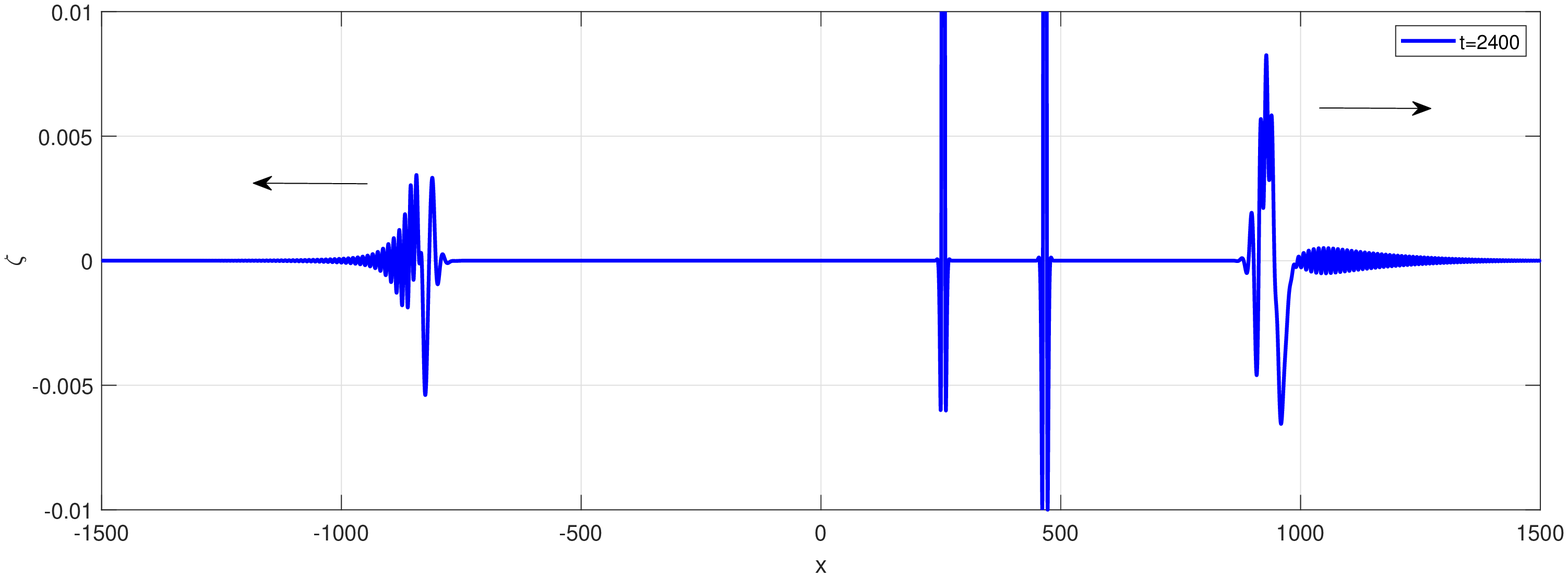}}
\caption{Overtaking collision: evolution of the $\zeta$ component of the numerical solution from the superposition of two solitary wave profiles with speeds $c_{s}^{(1)}=0.1, c_{s}^{(2)}=0.2$. (a) $t=0$; (b) $t=1600$; (c) Magnification of (b).}
\label{Fig_19}
\end{figure}

As in the head-on collision case, two new structures behind the taller wave (traveling to the left) and in front of the second one (traveling to the right) appear. The dispersive tails seem to evolve along with nonlinear wavelets which may hide a solitary form that requires a longer time to be formed.
\begin{figure}[htbp]
\centering
\subfigure[]
{\includegraphics[width=6.2cm]{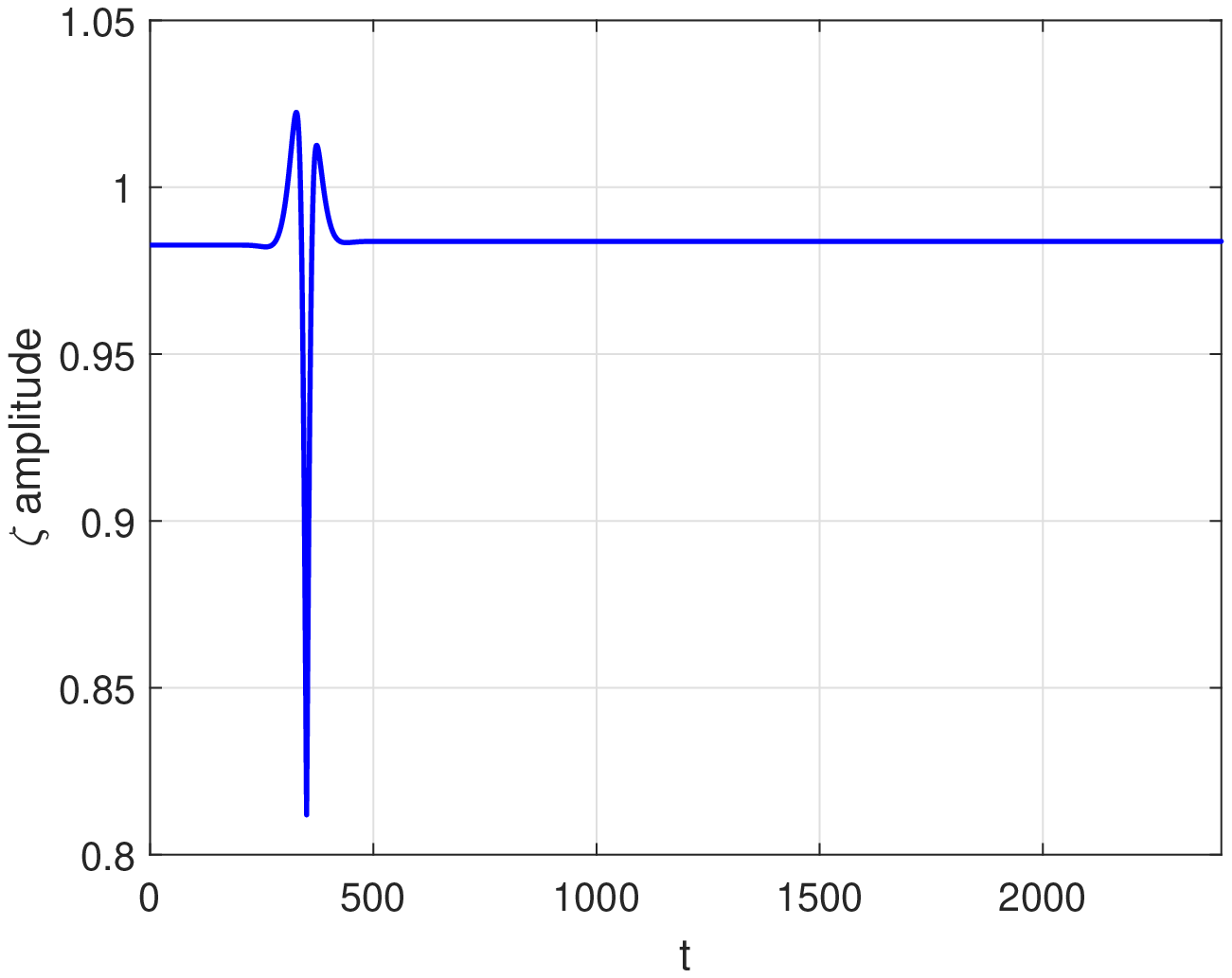}}
\subfigure[]
{\includegraphics[width=6.2cm]{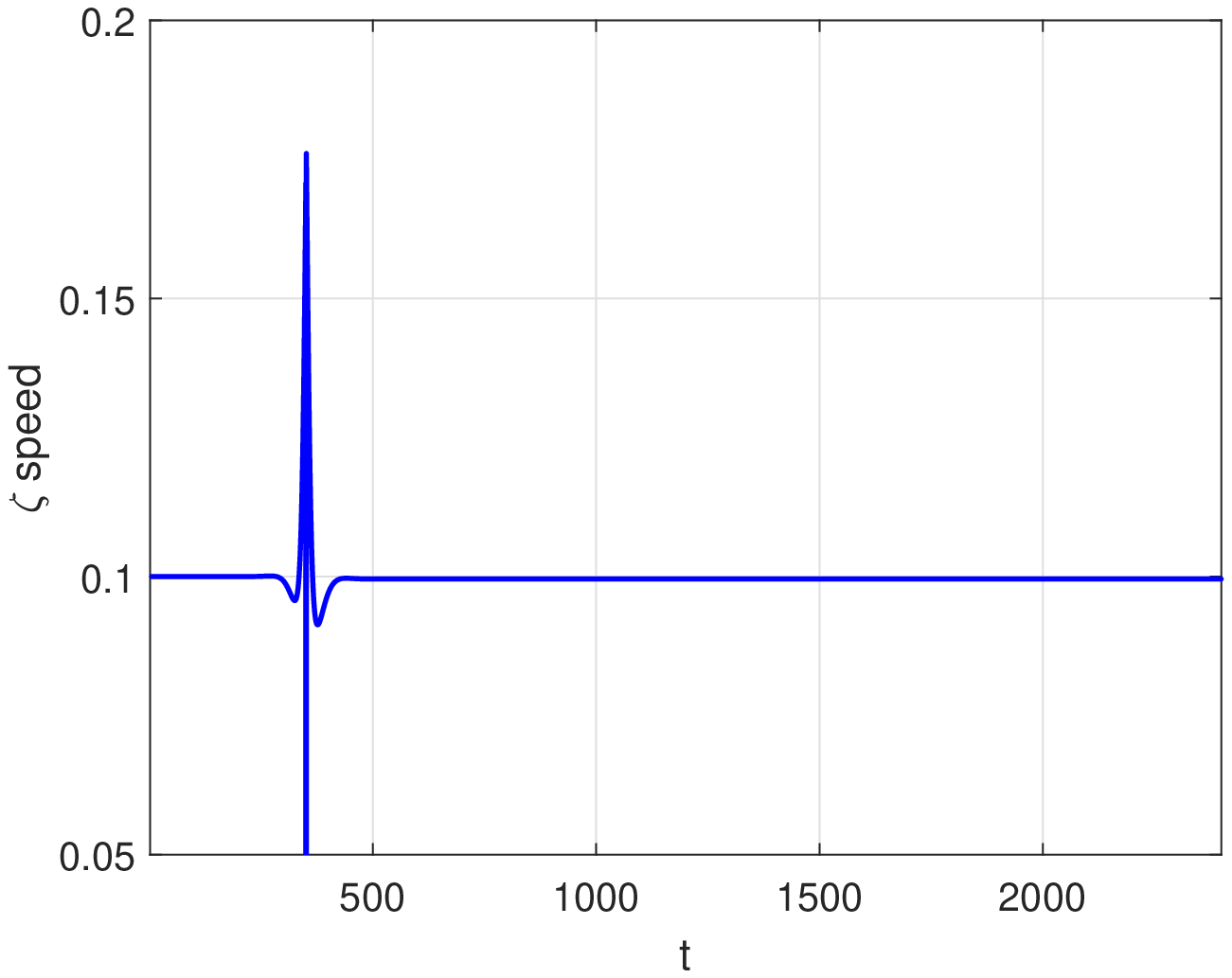}}
\caption{Amplitude and speed of the $\zeta$ component of the taller pulse of the numerical solution of Figure \ref{Fig_19}.}
\label{Fig_20}
\end{figure}

\section{Concluding remarks}
\label{sec5}
The present paper is devoted to study some properties of the so-called Boussinesq-Full Dispersion (BF/D) systems (\ref{BFD}). They are introduced as a model for the two-way propagation of internal waves alomg the interface of a two-layer system of fluids with rigid-lid condition for the upper layer and under the assumptions of a Boussinesq type physical regime for the upper layer and a full-dispersion regime for the lower one, \cite{BLS2008}. The system depends on four modelling parameters $a, b, c, d$. In Section \ref{sec2}, we approximate the corresponding periodic ivp by a Fourier-Galerkin spectral method and analyze the convergence of the semidiscretization. We establish $L^{2}$ error estimates for three cases of B/FD systems: the BBM-BBM case (corresponding to the values $b, d>0, a=c=0$), the generic case ($b,d>0, a,c<0$), and a third case with $b,d>0$ and either $a=0, c<0$ or $a<0, c=0$. In the three cases, the error estimates depend on the regularity of the solution of the periodic ivp, in the sense that, if the solution belongs to $H^{\mu}\times H^{\mu}, \mu\geq 1$, then the $L^{2}$ error behaves like $O(N^{-\mu})$, being $N$ the degree of the Fourier trigonometric approximation. In particular, for smooth solutions, spectral convergence holds.

Section \ref{sec3} is devoted to the existence and numerical generation of solitary wave solutions of (\ref{BFD}). We first describe the numerical technique of computation of solitary wave profiles, based on the approximation to the ode system for the solitary waves with periodic boundary conditions using a Fourier collocation scheme and the iterative resolution, via Petviashvili's method with extrapolation, of the resulting algebraic system for the discrete Fourier coefficients of the approximation. With this technique, we illustrate the main theoretical result about the existence of solitary wave solutions of (\ref{BFD}) by Angulo-Pava and Saut, \cite{AnguloS2019}, which concerns the Hamiltonian case $b=d>0, a,c\leq 0$. Then we observe that the numerical method generates approximate solitary waves beyond the hypotheses for the existence established in \cite{AnguloS2019}. This motivated us to study new conditions that might fit the numerical experiments. A modification of the proof of \cite{AnguloS2019} leads, in the case $b=d>0, a\leq 0, c<0$, to the determination of a limit value $c_{\gamma}$ such that a solitary wave of speed $c_{s}$ with $|c_{s}|<c_{\gamma}$ exists. The result is proved in Appendix \ref{appendixA}. As for the speed-amplitude relation, it is shown numerically that the amplitude of the solitary waves is an increasing function of $c_{\gamma}-c_{s}$. Additional computations in Section \ref{sec3} suggest the existence of some speed limit also in the nonhamiltonian case ($b\neq d$).

In Section \ref{sec4} we study computationally some aspects of the dynamics of solitary wave solutions of (\ref{BFD}). We first describe the numerical method to approximate the periodic ivp of (\ref{BFD}). This is based on the Fourier-Galerkin spectral discretization in space, analyzed in Section \ref{sec2}, and formulated here in the equivalent collocation form. (This was done to take advantage of the numerical technique of generation of solitary wave profiles introduced in Section \ref{sec3}.) After proving several properties of the semidiscrete system, concerning the preservation of invariants and a Hamiltonian structure when $b=d$, we introduce the fully discrete scheme that will be used in the computational study of the dynamics. The ode semidiscrete system is integrated in time by a fourth-order RK Composition method base on the implicit midpoint rule. The scheme was shown, theoretically and computationally, to be efficient when approximating nonlinear dispersive wave equations, \cite{FrutosS1992,DD,DDS1}. Some experiments of validation with computed solitary waves confirm here this accuracy.

The computational study of the dynamics of solitary waves presented in Section \ref{sec4} analyzes their stability from several points of view. Focused on the Hamiltonian case, the first experiments suggest the asymptotic stability of the solitary waves, in the sense that small perturbations of a solitary wave profile evolve into the generation of a main wave, which tends asymptotically to a modified solitary wave, accompanied by small-amplitude dispersive tails behind and in front of it. The existence of these two dispersive groups is analyzed from the linearized system of (\ref{BFD}). For larger perturbations, it is worth mentioning that the perturbed initial solitary wave profile may evolve generating a main wave of solitary type along with other waves, traveling to the left and to the right. A longer evolution reveals that these waves consist of nonlinear structures of solitary form and dispersive tails in front of them. This sort of resolution property is also observed experimentally when dealing with other initial conditions, such as small-amplitude Gaussian pulses and superpositions of solitary waves, traveling to the same or to opposite directions. In this last case, which concerns the dynamics of head-on and overtaking collisions, this resiolution seems to be slower;  their inelastic character may generate the formation of solitary-wave structures, additional to those main emerged after the collision, but requiring to this end a longer time than in other cases.


\section*{Acknowledgements}
The authors are supported by the Spanish Agencia Estatal de Investigaci\'on under Research Grant PID2020-113554GB-I00/AEI/10.13039/501100011033. They
would like to acknowledge travel support, that made possible this collaboration, from the Institute of Applied and Computational Mathematics of FORTH and the Institute of Mathematics (IMUVA) of the University of Valladolid.
Angel Dur\'an is also supported by the Junta de Castilla y Le\'on and FEDER funds (EU) under Research Grant VA193P20. 
Leetha Saridaki is also supported by the grant \lq\lq Innovative Actions in Environmental Research and Development (PErAn)\rq\rq (MIS5002358), implemented under the \lq\lq Action for the strategic development of the Research and Technological sector'' funded by the Operational Program \lq\lq Competitiveness, and Innovation'' (NSRF 2014-2020) and co-financed by Greece and the EU (European Regional Development Fund). The grant was issued to the Institute of Applied and Computational Mathematics of FORTH.

\begin{appendices}
\section{Appendix}
\label{appendixA}
In this appendix, the emergence of a speed limit for the existence of solitary wave solutions of (\ref{BFD}), suggested by some numerical experiments in section \ref{sec3}, is justified by some theoretical arguments. They are based on the analysis made in \cite{AnguloS2019}, where existence results are derived from the application of the Concentration-Compactness theory, \cite{Lions}, to the ode system for the solitary wave profiles given by (\ref{BFD2}), and where the corresponding version in terms of the Fourier transform is given in (\ref{BFD2b}), (\ref{BFD2c}).

Let $\gamma\in (0,1), \epsilon>0, b=d>0, a,c\leq 0$. Our starting point is the minimization problems considered in \cite{AnguloS2019}
\begin{equation}\label{appe1}
I_{\lambda}={\rm inf}\{E_{\mu_{2}}(\zeta,u): (\zeta,u)\in H^{1}\times H^{1}, F(\zeta,u)=\lambda\},
\end{equation}
for $\lambda>0$ and the functionals
\begin{eqnarray}
E_{\mu_{2}}(\zeta,u)&=&\int_{-\infty}^{\infty}\left(\frac{(1-\gamma)}{2}\zeta {J}_{c}\zeta+\frac{1}{2}u\mathcal{L}_{\mu_{2}}u-c_{s}\zeta J_{b}u\right)dx,\label{appe2}\\
F(\zeta,u)&=&\frac{\epsilon}{2\gamma}\int_{-\infty}^{\infty}\zeta u^{2}dx.\nonumber
\end{eqnarray}
The key to the proof given in \cite{AnguloS2019}, and which determines the conditions on the speed $c_{s}$ for the existence of solitary waves, is the coercivity property of the energy functional (\ref{appe2}) (cf. Proposition 2.1(b) and Proposition 3.1(d) in \cite{AnguloS2019}). Our purpose here is to study this property by writing $E_{\mu_{2}}$ in the form
\begin{equation}\label{appe4}
E_{\mu_{2}}=\frac{1}{2}\langle Q\begin{pmatrix}\zeta\\u\end{pmatrix}, \begin{pmatrix}\zeta\\u\end{pmatrix}\rangle,
\end{equation}
where $\langle\cdot,\cdot\rangle$ denotes the inner product in $H^{1}\times H^{1}$ given by
\begin{equation*}
\langle \begin{pmatrix}\zeta_{1}\\u_{1}\end{pmatrix}, \begin{pmatrix}\zeta_{2}\\u_{2}\end{pmatrix}\rangle=\int_{-\infty}^{\infty}(\zeta_{1}\overline{\zeta_{2}}+u_{1}\overline{u_{2}})dx,
\end{equation*}
and $Q$ is the matrix operator with Fourier symbol
\begin{equation}\label{appe5}
\widehat{Q}(k)=\begin{pmatrix} (1-\gamma)j_{c}(k)&-c_{s}j_{b}(k)\\-c_{s}j_{b}(k)&l_{\mu_{2}}(k)\end{pmatrix},\quad k\in\mathbb{R}.
\end{equation}
We will use the Fourier representation (\ref{appe5}) to study the operator $Q$.
\begin{lemma}
\label{lemA1}
Under the hypotheses on the parameters above and for $x\geq 0$, we have
\begin{equation*}
l_{\mu_{2}}(x)\geq \frac{3}{4\gamma}.
\end{equation*}
\end{lemma}
\begin{proof}
We write $l_{\mu_{2}}(x)=l(\sqrt{\mu_{2}}x)$ where for $y\geq 0$
\begin{equation*}
l(y)=\frac{1}{\gamma}P\left(\frac{\alpha}{\gamma}g(y)\right)+\frac{\alpha^{2}}{\gamma}|a|y^{2},
\end{equation*}
with $g(0)=1$ and
\begin{equation*}
g(y)=y{\rm coth}{y}, y>0,\quad P(z)=1-z+z^{2}.
\end{equation*}
The polynomial $P$ attains a global minimum at $z=1/2$ and therefore $P(z)\geq P(1/2)=3/4, z\in\mathbb{R}$. Thus, for $y\geq 0$
\begin{equation*}
l(y)\geq \frac{3}{4\gamma}+\frac{\alpha^{2}}{\gamma}|a|y^{2},
\end{equation*}
and the lemma holds.
\end{proof}

For $x\geq 0$ we define the function
\begin{equation}\label{appe6}
R_{\gamma}(x)=\frac{(1-\gamma)j_{c}(x)l_{\mu_{2}}(x)}{j_{b}(x)^{2}}
\end{equation}
The coercivity property for $E_{\mu_{2}}$ is based on the following result on (\ref{appe6}):
\begin{lemma}
\label{lemA2} Let $\gamma\in (0,1), b=d>0, c<0,a\leq 0$. Then there exists $m=m(\gamma)>0$ such that
\begin{equation}\label{appe6b}
R_{\gamma}(x)\geq m,\quad \forall x\geq 0.
\end{equation}
\end{lemma}
\begin{proof}
We note that, by Lemma \ref{lemA1}, $R_{\gamma}$ is positive; furthermore
\begin{equation*}
\lim_{x\rightarrow+\infty}R_{\gamma}(x)=R_{m}:=\frac{(1-\gamma)}{\gamma}\frac{|c|}{b^{2}}\left(|a|+\frac{1}{\gamma^{2}}\right)>0.
\end{equation*}
Therefore, there is $M>0$ such that $R_{\gamma}(x)>R_{m}/2$ for $x>M$. Moreover, since $R_{\gamma}$ is continuous, there exists $x_{\gamma}\in [0,M]$ with
\begin{equation*}
R_{\gamma}(x_{\gamma})=\min_{0\leq x\leq M}R_{\gamma}(x).
\end{equation*}
Thus, (\ref{appe6b}) holds by taking $m\leq \min\{\frac{R_{m}}{2},R_{\gamma}(x_{\gamma})\}$.
\end{proof}
\begin{remark}
Note that the previous result is independent of the value of $\nu:=\sqrt{\mu/\mu_{2}}$. For modelling arguments, it is reasonable to assume that $\nu<1$. In such case, all the computations performed suggest that $m=R_{\gamma}(x_{\gamma})$. This is illustrated in Figure \ref{Fig_A1}, that represents (\ref{appe6}) for two values of $\nu<1$, for which $x_{\gamma}=0$ and when $x_{\gamma}>0$ respectively. When $\nu\geq 1$, one can find experimentally some cases for which $R_{\gamma}$ does not attain the minimum.
\begin{figure}[htbp]
\centering
\subfigure[]
{\includegraphics[width=6.2cm]{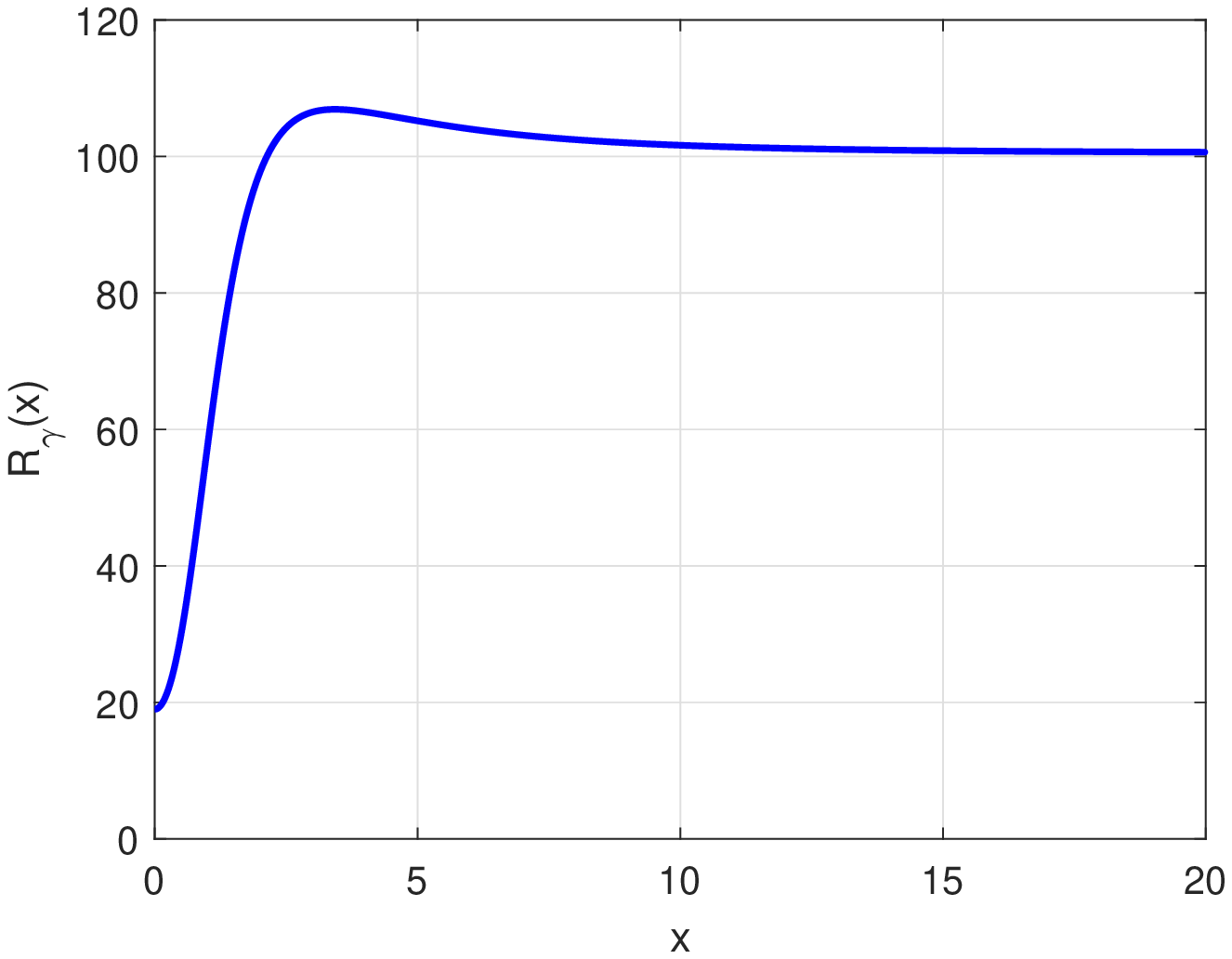}}
\subfigure[]
{\includegraphics[width=6.2cm]{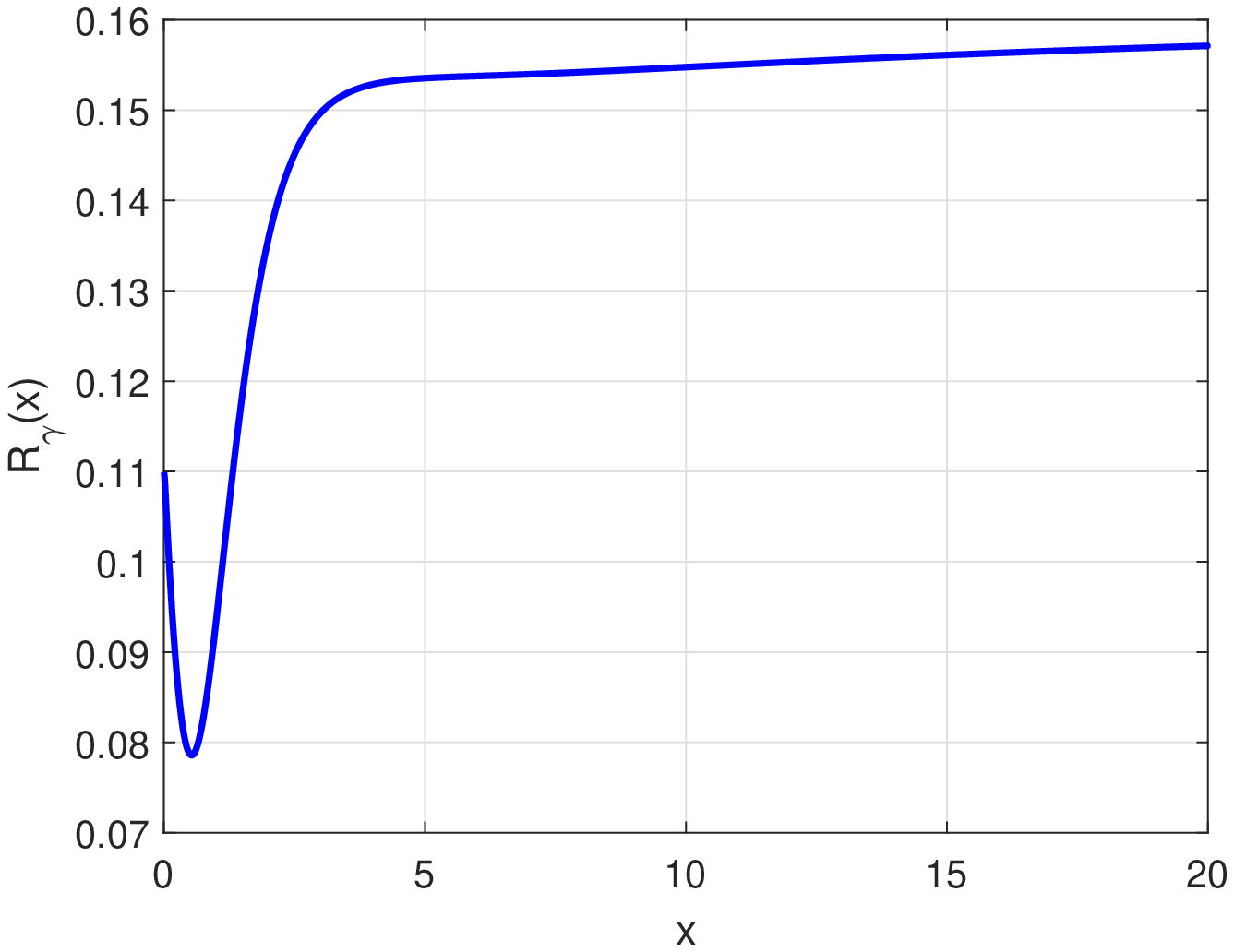}}
\caption{Function (\ref{appe6}) for the values $\epsilon_{d-b}=0$, $a,b,c,d$ given by (\ref{321}) and (a) $\gamma=0.2, \nu:=\sqrt{\mu/\mu_{2}}=0.5$; (b) $\gamma=0.9, \nu=0.01$.}
\label{Fig_A1}
\end{figure}
\end{remark}
\begin{lemma}
\label{lemA3} Let $\gamma\in (0,1), b=d>0, c<0,a\leq 0$. Let $m=m(\gamma)>0$ be defined in Lemma \ref{lemA2}. If
\begin{equation}\label{appe6c}
|c_{s}|<c_{\gamma}:=\sqrt{m(\gamma)}.
\end{equation}
then the operator $Q$ defined in (\ref{appe4}), (\ref{appe5}) is positive definite and defines a norm which is equivalent to the usual $H^{1}\times H^{1}$ norm.
\end{lemma}

\begin{proof}
Note first that $Q$ is Hermitian from the representation (\ref{appe5}). For $k\in\mathbb{R}$, let
$\delta(k):=(1-\gamma)j_{c}(k)l_{\mu_{2}}(k)-c_{s}^{2}j_{b}(k)^{2}$ and
\begin{eqnarray}
\lambda_{\pm}(k)&=&\frac{1}{2}\left((1-\gamma)j_{c}(k)+l_{\mu_{2}}(k)\pm\sqrt{((1-\gamma)j_{c}(k)+l_{\mu_{2}}(k))^{2}-4\delta(k)}\right)\label{appe6d}\\
&=&\frac{1}{2}\left((1-\gamma)j_{c}(k)+l_{\mu_{2}}(k)\pm\sqrt{((1-\gamma)j_{c}(k)-l_{\mu_{2}}(k))^{2}+4c_{s}^{2}j_{b}(k)^{2}}\right),\nonumber
\end{eqnarray} 
be the (real) eigenvalues of the matrix (\ref{appe5}). Then
\begin{eqnarray}
\lambda_{-}(k)+\lambda_{+}(k)&=&(1-\gamma)j_{c}(k)+l_{\mu_{2}}(k),\label{appe7a}\\
\lambda_{-}(k)\lambda_{+}(k)&=&\delta(k).\label{appe7b}
\end{eqnarray}
Therefore, (\ref{appe6b})-(\ref{appe7b}) and lemmas \ref{lemA1} and \ref{lemA2} imply that $\lambda_{+}(k)>\lambda_{-}(k)>0$. The positive definite character of $Q$ follows from Parseval identity.

Now we prove the equivalence between the $H^{1}\times H^{1}$ norm and the norm defined by $Q$. To this end, we show the existence of positive constants $c_{j}, d_{j}, j=0,1$ such that
\begin{equation}\label{appe8}
c_{0}+c_{1}k^{2}\leq \lambda_{-}(k)<\lambda_{+}(k)\leq d_{0}+d_{1}k^{2},\quad k\in\mathbb{R}.
\end{equation}
From (\ref{appe7a}) we have
\begin{equation*}
\lambda_{+}(k)<(1-\gamma)j_{c}(k)+l_{\mu_{2}}(k),
\end{equation*}
and the existence of $d_{0}$ and $d_{1}$ is derived from the last inequality, the form of $j_{c}$ and $l_{\mu_{2}}$ given by (\ref{BFD2c}) and the inequality, cf. \cite{AnguloS2019}
\begin{equation}\label{appe9}
y{\rm coth}{y}\leq 1+y,\quad y>0.
\end{equation}
On the other hand, from (\ref{BFD2c}) and (\ref{appe9}) we can find positive constants $A,B$ with $l_{\mu_{2}}(k)\leq A+Bk^{2}$ for all $k\in\mathbb{R}$. Then, taking $c_{0}$ and $c_{1}$ such that
$$c_{0}>(1-\gamma)+A,\quad c_{1}>(1-\gamma)\mu |c|+B\mu_{2},$$ and after some computations, the first inequality of (\ref{appe8}) is satisfied and the lemma holds.
\end{proof}
From (\ref{appe4}) and Lemma \ref{lemA2} we can then derive the coercivity property for the functional $E_{\mu_{2}}$. One can check that the rest of the arguments used in \cite{AnguloS2019}, concerning the application of the Concentration-Compactness theory for the existence of solitary waves, as well as the proofs on the regularity and asymptotic decay, are also valid here. This leads to the following result.
\begin{theorem}
\label{theorA3}
Let $\gamma\in (0,1), \epsilon>0, b=d>0, c<0, a\leq 0$. We assume that (H1) holds and that $c_{s}$ satisfies (\ref{appe6c}). Then the system (\ref{BFD}) admits a smooth solitary wave solution $(\zeta,u)$ of speed $c_{s}$ that decays exponentially to zero at infinity.
\end{theorem}

\end{appendices}

\end{document}